\documentclass[UTF-8,reqno]{amsart}
\usepackage[T1]{fontenc}
\usepackage{enumerate}
\usepackage{mhequ,wasysym,stmaryrd}
\usepackage{tikz}
\usetikzlibrary{matrix}
\usepackage[all]{xy}
\usepackage{amssymb,url,color,booktabs,nccmath,amsmath,amsthm}
\usepackage{bm}
\usepackage[left=2.7cm,right=2.7cm,top=3.4cm,bottom=3.4cm]{geometry}
\usepackage{mathrsfs}
\usepackage{dsfont}
\usepackage[shortlabels]{enumitem}
\usepackage{float}
\usepackage{longtable}
\usepackage{color}
\usepackage[colorlinks=true]{hyperref}
\hypersetup{
    linkcolor=blue,          % color of internal links
    citecolor=red,        % color of links to bibliography
    filecolor=blue,      % color of file links
    urlcolor=cyan
}
\newtheorem{theorem}{Theorem}[section]
\newtheorem{lemma}[theorem]{Lemma}
\newtheorem{proposition}[theorem]{Proposition}
\newtheorem{remark}[theorem]{Remark}

\newtheorem{corollary}[theorem]{Corollary}
\newtheorem{assumption}{Assumption}[section]
\newtheorem{definition}[theorem]{Definition}
\theoremstyle{definition}

\newcommand{\Rnum}[1]{\uppercase\expandafter{\romannumeral #1\relax}}
\usepackage{esint}
\usepackage{cite}
\usepackage{verbatim}
\usepackage{times}
\usepackage{bm}
\usepackage{mathtools}
\numberwithin{equation}{section}
\usepackage{tikz-cd}
\usepackage{caption}
\usepackage{nicematrix}
\usepackage{mathtools}
\usepackage{accents}

\newcommand{\abs}[1]{\left\vert#1\right\vert}
\newcommand{\norm}[1]{\left\Vert#1\right\Vert}
\newcommand{\norms}[1]{{\left\vert\kern-0.25ex\left\vert\kern-0.25ex\left\vert #1 
    \right\vert\kern-0.25ex\right\vert\kern-0.25ex\right\vert}}
\newcommand\laweq{\mathrel{\overset{\makebox[0pt]{\mbox{\normalfont\tiny\sffamily law}}}{=}}}

\newcommand\lawcon{\mathrel{\overset{\makebox[0pt]{\mbox{\normalfont\tiny\sffamily law}}}{\longrightarrow }}}

\newcommand\wcon{\mathrel{\overset{\makebox[0pt]{\mbox{\normalfont\tiny\sffamily w}}}{\longrightarrow }}}

\def\d{{\mathrm d}}
\def\p{{\mathfrak p}}
\def\N{{y:y\sim x}}
\def\L{{\mathfrak L}}
\def\X{{\mathcal X}}

\def\dF{{\mathbf{d}_{\Lambda}}}
\def\dI{{\mathbf{d}_{\rm I}}}

\def\SSF{{\mathbb{S}_{\Lambda}}}
\def\SSI{{\mathbb{S}_{\rm I}}}
\def\bd{{\mathbf{d}}}
\def\dL{{\mathbf{d}^*_{\rm L}}}
\def\WF{{\mathbf{W}_{2,{\Lambda}}}}
\def\WI{{\mathbf{W}_{2,{\rm I}}}}
\def\WFT{{\mathbf{W}^2_{2,{\Lambda}}}}
\def\WIT{{\mathbf{W}^2_{2,{\rm I}}}}
\def\dLF{{\mathbf{d}^*_{{\rm L},{\Lambda}}}}
\def\dLI{{\mathbf{d}^*_{{\rm L},{\rm I}}}}
\def\C{{\mathbf{C}}}
\def\O{{{1}}}

\newcommand{\muc}{\widetilde{\bm{\mu}}}

\begin{document}

\title{large $N$ limit of the Langevin dynamics for the spin $O(N)$ model}
\author{Wenjie Ye}
\address[W. Ye]{School of Mathematics and Statistics \& Key Laboratory of Analytical Mathematics and Applications (Ministry of Education), Fujian Normal University, Fuzhou 350117, China}
\email{yewenjie@fjnu.edu.cn}

\author{Rongchan Zhu}
\address[R. Zhu]{Department of Mathematics, Beijing Institute of Technology, Beijing 100081, China 
}
\email{zhurongchan@126.com}
\begin{abstract} 
In this paper, we prove that the large $N$ limit of the Langevin dynamics for the spin $O(N)$ model is given by a mean-field stochastic differential equation (SDE) in both finite and infinite volumes.  We establish uniform in $N$ bounds for the dynamics, which enable us to demonstrate convergence to the mean-field SDE with polynomial interactions. Furthermore, the mean-field SDE is shown to be globally well-posed for suitable initial distributions. We also prove the existence of stationary measures for the mean-field SDE. For small inverse temperatures, we characterize the large $N$ limit of the spin $O(N)$ model through stationary coupling. Additionally, we establish the uniqueness of the stationary measure for the mean-field SDE.
\end{abstract}

\subjclass[2020]{60H10, 82C22}
\keywords{Spin $O(N)$ model, propagation of chaos, mean-field limit, massive Gaussian free field}
\maketitle
\tableofcontents
\section{Introduction}
\subsection{Model and main results}
The spin $O(N)$ model is a fundamental model in statistical mechanics, first introduced by Stanley \cite{Stanley1968}. Let $\mathbb{T}^d$ denote the $d$-dimensional torus, $\mathbb{Z}^d$ the $d$-dimensional integer lattice, and $\Lambda_{\mathfrak{L}} := \mathbb{Z}^d \cap \mathfrak{L}\mathbb{T}^d$ the finite $d$-dimensional lattice with side length $\mathfrak{L}>0$ and unit lattice spacing. For notational simplicity, we sometimes write $\Lambda=\Lambda_\L$.  The spin  $O(N)$ model on $\Lambda$  is defined by a probability measure 
\begin{equation}\label{spin-O-N}
    \begin{aligned}
         & \d\mu_{\Lambda ,N,\kappa}(\Phi):=\frac{1}{Z_{\Lambda,N,\kappa}}\exp\left(2\kappa\sum_{x\sim y}\langle\Phi_x,\Phi_y \rangle\right)\prod_{z\in \Lambda}\d \lambda_{\mathbb S^{N-1}(\sqrt{N})}(\Phi_z),
    \end{aligned}
\end{equation}
where $\Phi=(\Phi_x)\in \mathcal{S}:=(\mathbb{S}^{N-1}(\sqrt{N}))^{\Lambda}$\footnote{ Here, $\mathbb{S}^{N-1}(\sqrt{N})$ is the $(N{-}1)$-dimensional sphere of radius $\sqrt{N}$ centered at the origin, embedded in the Euclidean space $\mathbb{R}^N$; $(\mathbb{S}^{N-1}(\sqrt{N}))^\Lambda$ denotes the set of all maps $\Phi \colon \Lambda \to \mathbb{S}^{N-1}(\sqrt{N})$. 
}; $\kappa>0$ is the inverse temperature; $\lambda_{\mathbb S^{N-1}(\sqrt{N})}$ denotes the uniform measure on the sphere $\mathbb S^{N-1}(\sqrt{N})$; $x \sim y$ means $x$ and $y$ are adjacent in the lattice, the sum over $x\sim y$ runs over all neighbouring points $x,y\in \Lambda$;  $\langle \cdot,\cdot\rangle$ denotes the inner product on the Euclidean space $\mathbb{R}^N$; and  $Z_{\Lambda,N,\kappa}$ is the normalization constant ensuring that $\mu_{\Lambda,N,\kappa}$ is a probability measure. An important feature of the model is its $O(N)$-invariance, i.e. if $A$ is an $N$-dimensional orthogonal matrix, then $\mu_{\Lambda, N, \kappa} \circ A^{-1} = \mu_{\Lambda, N, \kappa}$, where $A\Phi:=(A\Phi_x)$ for $\Phi\in \mathcal S$. The choice of the sphere's radius leads to a nontrivial limit as $N \to \infty$.  In this paper, we study the large $N$ limit of the related Langevin dynamics.

The Langevin dynamics for the spin $O(N)$ model \eqref{spin-O-N} is given by the following SDEs (see \cite[Section $3.2$]{Shen2024} for more details on the derivation\footnote{In this paper, we use a slightly different formulation from \cite[$(3.1)$]{Shen2024}
$\d\Phi = \frac{1}{2}\nabla \mathcal{V}(\Phi)\, \d t + \d \mathfrak{B}$
to derive \eqref{the-new-model-i}, where $\mathcal{V} := 2\kappa \sum_{x\sim y} \langle \Phi_x, \Phi_y \rangle$ and $\mathfrak{B} = (\mathfrak{B}_x)$ denotes a family of independent Brownian motions on $\mathbb{S}^{N-1}(\sqrt{N})$. Note that in \eqref{spin-O-N}, the sum over $x\sim y$ runs over all neighbouring points $x,y\in \Lambda$, which leads to the factor $2\kappa$ appearing in both \eqref{spin-O-N} and \eqref{the-new-model-i}.}): 
\begin{equation}\label{the-new-model-i}
    \left\{\begin{aligned}
         \d\Phi^i_x(t)= &2\kappa\sum_{\N}\left( \Phi_y^i(t)-\frac{\sum^N_{k=1} \Phi_x^k(t) \Phi_y^k(t)}{N} \Phi_x^i(t)\right)\,\d t-\frac{N-1}{2N} \Phi_x^i(t)\,\d t\\
&+\,\d W_x^i(t)-\frac{\sum^N_{k=1} \Phi_x^k(t)\,\d W_x^k(t) }{N}\Phi_x^i(t),\quad 0\le t\le T,\\
\Phi_x^i(0)=&\varphi_x^i,\quad  i=1,\cdots,N,
\end{aligned}\right.
\end{equation}
where $x\in \Lambda$, for each $x \in \Lambda$, $\bm{W}_x := (W_x^1, \cdots, W_x^N)$ is an $N$-dimensional standard Brownian motion, and for $x \neq y$, the processes $\bm{W}_x$ and $\bm{W}_y$ are independent; the sum over $\N$ runs over all points adjacent to $x$; and for $\varphi_x = (\varphi_x^1, \cdots, \varphi_x^N) \in \mathbb{S}^{N-1}(\sqrt{N})$ with $x \in \Lambda$, the solutions $\Phi(t)$, $0\le t\le T$ to the Langevin dynamics stay in $\mathcal{S}$  (see  Proposition \ref{invariance-SDE} below).  We note that $\Phi^i_x$ depends on $N$, but we omit this dependence for notational simplicity.   By replacing $\Lambda$ with $\mathbb{Z}^d$, the system \eqref{the-new-model-i} can be extended to the infinite lattice $\mathbb{Z}^d$.
Our focus in this paper is on the asymptotic behavior of the Langevin dynamics as $N\rightarrow \infty$. The limiting dynamics is governed by the following mean-field SDE:
\begin{equation}\label{mean-field-limit-n-i}
    \left\{\begin{aligned}
         \d\Psi_x(t)= &2\kappa\sum_{\N}\Big(  \Psi_y(t)-\mathbb{E}[ \Psi_x(t) \Psi_y(t)] \Psi_x(t)\Big)\,\d t-\frac{1}{2}\Psi_x(t)\,\d t+\d W_x(t),\quad 0\le t\le T,\\
\Psi_x(0)=& \psi_x, \quad x\in \mathfrak{A}, 
\end{aligned}\right.
\end{equation}
where $\mathfrak{A}=\Lambda \text{\ or\ } \mathbb{Z}^d$. 
On the formal level, \eqref{mean-field-limit-n-i} arises naturally:
if the initial data $\{ (\varphi_x^i) \}^N_{i=1}$ is symmetric\footnote{This means that for any permutation $\pi$, the sequences  $((\varphi_x^1),\cdots, (\varphi_x^N))$ and $((\varphi_x^{\pi(1)}),\cdots, (\varphi_x^{\pi(N)}))$ have the same joint probability distribution, where $(\varphi^i_x)$, $i=1,\ldots, N$, denotes the vector indexed by $x \in \mathfrak{A}$. }, then the solution components $\{ (\Phi^i_x) \}^N_{i=1}$ to the SDEs \eqref{the-new-model-i}
share the same laws. As $N$ approaches infinity, the interaction term $({1}/{N})\sum^N_{k=1}\Phi_x^k \Phi_y^k$ becomes weaker.
Consequently, the empirical average behavior $({1}/{N}) \sum^N_{k=1} \Phi_x^k \Phi_y^k$ tends to the expectation $\mathbb{E}[\Phi^1_x \Phi^1_y]$. Similarly, as $N \rightarrow \infty$, the term $(1/N)\sum_{k=1}^N \int_0^t\Phi_x^k\,\d W_x^k$ vanishes, since, by  It\^o's isometry, its quadratic variation tends to zero.   Our first main result is that the large $N$ limit of the Langevin dynamics for the spin $O(N)$ model \eqref{the-new-model-i} can be given by the mean-field SDE \eqref{mean-field-limit-n-i}. Specifically,  let $S^N := \frac{1}{N}\sum_{j=1}^N \delta_{\Phi^j}$ be the empirical measure of $(\Phi^1, \ldots, \Phi^N)$, where $\Phi^j := (\Phi^j_x)$ and $\delta_{\Phi^j}$ is the Dirac measure at $\Phi^j$. We have the following theorem:
\begin{theorem}\label{simple-1}
On a finite lattice $\Lambda$,  suppose that the initial data of the system \eqref{the-new-model-i} is symmetric, satisfies a uniform $\p$-th moment bound in $N$ for some $\p>2$, and its empirical measure converges weakly in probability to some measure $\mu_0$, then the empirical measure $S^N$ of the solutions to the system \eqref{the-new-model-i} converges in law to $S$. The limit  $S$ is the unique martingale solution to the mean-field SDE \eqref{mean-field-limit-n-i} with initial distribution $\mu_0$.  Moreover, we have 
$$
\int \abs{v_x}^2\,S(t)(\d v)= 1,\  \mathbb{P}\text{-\rm{a.s.}}, \ \forall\, x\in \Lambda, \, 0\le t\le T,$$ and
$$
\lim_{N\rightarrow \infty}\mathbb{E}\left[\WFT(S^N,S)\right]=0,
$$ 
where $\WF(S^N,S)$ denotes the $2$-Wasserstein distance between  $S^N$ and $S$.  Furthermore, the results still hold on the infinite lattice $\mathbb{Z}^d$.
\end{theorem}
See Theorem \ref{Main-Theorem} and Theorem \ref{Main-Theorem-Z} below for precise statements.  We also refer to Remark \ref{typical-examples} and Remark \ref{typical-examples-Z} for examples of initial data.  

In general, to derive Theorem \ref{simple-1}, the global well-posedness of the mean-field SDE \eqref{mean-field-limit-n-i} is required. However, the nonlinear term $\mathbb{E}[\Psi_x\Psi_y]\Psi_x$ appearing in \eqref{mean-field-limit-n-i} is a polynomial function of $(\Psi_x(t))$, and does not satisfy the dissipative condition; specifically, the inequality $\mathbb{E}[\Psi_x\Psi_y]\Psi_x^2\ge 0$ does not always hold. Consequently, the general well-posedness results for mean-field SDEs (see, e.g., \cite{hong2023}) cannot be directly applied.  On the other hand, due to the presence of the non-dissipative term, it is not easy to derive a priori uniform estimates and construct global-in-time solutions to the mean-field SDE \eqref{mean-field-limit-n-i} for general initial data.

To overcome these difficulties and prove Theorem \ref{simple-1}, we instead reformulate the mean-field SDE \eqref{mean-field-limit-n-i} as a martingale problem, and construct solutions to the martingale problem based on the empirical measure $S^N$ induced by the Langevin dynamics \eqref{the-new-model-i}. We use the fact that the solution to \eqref{the-new-model-i} $\Phi_x \in \mathbb{S}^{N-1}(\sqrt{N})$ for all $x \in \Lambda$ to obtain a uniform in $N$ estimate for \eqref{the-new-model-i}. Then, using compactness arguments, we show that every tight limit of the sequence $\{S^N\}_{N \in \mathbb{N}}$ is a martingale solution to the mean-field SDE \eqref{mean-field-limit-n-i}, thereby establishing the existence of solutions to \eqref{mean-field-limit-n-i}.  Moreover, we establish the pathwise uniqueness of solutions for the mean-field SDE \eqref{mean-field-limit-n-i}, which implies the convergence of the entire sequence $\{S^N\}_{N \in \mathbb{N}}$ and subsequently the propagation of chaos.

To compare the Langevin dynamics \eqref{the-new-model-i} with the mean-field SDE  \eqref{mean-field-limit-n-i} directly, we rewrite \eqref{mean-field-limit-n-i} as
\begin{equation}\label{Psi-i}
    \left\{\begin{aligned}
         \d\Psi^i_x(t)= &2\kappa\sum_{\N}\Big(  \Psi_y^i(t)-\mathbb{E}[ \Psi_x^i(t) \Psi_y^i(t)] \Psi_x^i(t)\Big)\,\d t-\frac{1}{2}\Psi_x^i(t)\,\d t+\d W^i_x(t),\quad 0\le t\le T,\\
\Psi_x^i(0)=& \psi_x^i, \quad x\in \Lambda,\quad i\in \mathbb{N},
\end{aligned}\right.
\end{equation}
with $\{ W_x^i \}$ being the same as in \eqref{the-new-model-i}. 

By appropriately choosing the initial data $\psi^i$ such that $\psi^i \laweq \mu_0$\footnote{The symbol $\psi^i\laweq \mu_0$ means that the law of $\psi^i$ is $\mu_0$.} , where $\mu_0$ is taken from a certain class of measures (see Theorem~\ref{strong-solution} below for the precise conditions on $\mu_0$),  we are able to obtain the convergence rate.

\begin{theorem}\label{simple-2}
Suppose that the initial data of the system \eqref{the-new-model-i} and $\mu_0$ satisfy the assumptions given in Theorem \ref{simple-1}, $\psi^i\laweq \mu_0$, and that  $\lim_{N\rightarrow \infty}\sum_{x\in \Lambda}\mathbb{E}[|\varphi^i_x-\psi^i_x|^2]=0$. Then, we have
\begin{equation}\label{con-1}
    \begin{aligned}
         & \lim_{N\rightarrow \infty}\sum_{x\in \Lambda}\mathbb{E}\left[ \sup_{0\le t\le T} \abs{\Phi_x^i(t)-\Psi_x^i(t)}^2 \right]=0,
    \end{aligned}
\end{equation}
where $\Phi^i$ and $\Psi^i$ denote the solutions to \eqref{the-new-model-i} and \eqref{Psi-i} with initial data $\varphi^i$ and $\psi^i$, respectively.

If the initial data of the system \eqref{the-new-model-i} and $\mu_0$ satisfy the assumptions given in Theorem \ref{simple-1} with $\p=8$, $\psi^i\laweq \mu_0$, in addition,  the random variables $\{\psi^i\}_{i\in \mathbb{N}}$ are independent and identically distributed, and the sequence $\{(\varphi^i,\psi^i)\}^N_{i=1}$ is symmetric, then we have
\begin{equation}
    \begin{aligned}
         & \sup_{0\le t\le T}\sum_{x\in \Lambda}\mathbb{E}\left[  \abs{\Phi_x^i(t)-\Psi_x^i(t)}^2 \right]\lesssim  \sum_{x\in \Lambda}\mathbb{E}\left[\abs{\varphi^i_x-\psi^i_x}^2\right]+ \frac{1}{N},
    \end{aligned}
\end{equation}
where the implicit constant depends on $\abs{\Lambda}$\footnote{The cardinality of $\Lambda$ is denoted by $\abs{\Lambda}$.}, but is independent of $N$. 

Moreover, the results still hold on the infinite lattice $\mathbb{Z}^d$.
\end{theorem}
See Theorems \ref{strong-solution} and \ref{strong-solution-Z} below for precise statements.  We also refer to Remark \ref{Phi-Psi} and Remark \ref{Phi-Psi-Z} for examples of initial data.  

To prove Theorem \ref{simple-2}, we compare these two dynamics \eqref{the-new-model-i} and \eqref{Psi-i}  directly, which cancels the noise term $\d W_x^i$. By It\^o's isometry, the extra noise term $(1/N)\sum_{k=1}^N \Phi_x^k\,\d W_x^k \cdot \Phi_x^i$ is of order $1/\sqrt{N}$.  According to Theorem~\ref{simple-1}, for the initial distribution $\mu_0$ considered therein,
the solution to the mean-field SDE~\eqref{Psi-i} satisfies $\mathbb{E}[|\Psi^i_x(t)|^2] = 1$ for any $x \in \Lambda$ and $0 \le t \le T$. In particular, this implies that the term $\mathbb{E}[\Psi_x^i(t)\Psi_y^i(t)]$ in \eqref{Psi-i} is uniformly bounded, which, in turn, enables us to control the difference between the two dynamical systems.  Moreover, to obtain the convergence rate of $\mathbb{E}[|\Phi^i-\Psi^i|^2]$, by the symmetry of $\Phi^k-\Psi^k$ in $k$,  it suffices to estimate
\begin{equation*}
    \begin{aligned}
         & \frac{1}{N}\sum^{N}_{k=1}\mathbb{E}\left[\abs{\Phi^k-\Psi^k}^2\right].
    \end{aligned}
\end{equation*}

{The second part of this paper (Section \ref{Invariant measure}) is devoted to the stationary measures of the dynamics \eqref{the-new-model-i} and \eqref{mean-field-limit-n-i}, and their large $N$ limits.}  Our objective is to investigate the existence and uniqueness of the stationary measure to the mean-field SDE \eqref{mean-field-limit-n-i}, as well as to study the large $N$ behavior of the spin $O(N)$ model using a dynamical method. Note that the existence of stationary measures does not follow from the usual Krylov-Bogolyubov method, since the mean-field SDE \eqref{mean-field-limit-n-i} depends on the laws of the solutions, and the solutions themselves do not form Markov processes.

We first present the existence of stationary measures to the mean-field SDE \eqref{mean-field-limit-n-i} in the finite volume case. Define
\begin{equation}\label{Pconst-i}
    \begin{aligned}
         & \mathscr{P}_{\mathrm{const}}(\mathbb{R}^{\abs{\Lambda}}) := \left\{
            \mu \in \mathscr{P}(\mathbb{R}^{|\Lambda|})
            \;\middle|\;
            \text{there exists } u \in \mathbb{R} \text{ such that } 
            \int v_x v_y\, \mathrm{d}\mu(v) = u, \text{ for all } x \sim y
         \right\}.
    \end{aligned}
\end{equation}
Let $\mu^\L$ be a massive Gaussian free field scaled by ${1}/{(2\sqrt{\kappa})}$ with mass $m^2>0$, where $\kappa$ is the inverse temperature in \eqref{spin-O-N}, $m^2$ is the unique positive solution to 
\begin{equation*}
    \begin{aligned}
         & G_{\Lambda_\L,m^2}(x,x)=4\kappa.
    \end{aligned}
\end{equation*}
We refer to the definition of $G_{\Lambda_\L,m^2}$ in Section \ref{sec2.2} below.  By reducing the study of stationary solutions to the discrete stochastic heat equation and verifying the corresponding self-consistency condition (see \eqref{fixed-poit} below), we obtain the following result.

\begin{theorem}\label{GaFF}
The field $\mu^\L$ is the unique stationary measure in $\mathscr{P}_{\mathrm{const}}(\mathbb{R}^{\abs{\Lambda}}) $ to the mean-field SDE \eqref{mean-field-limit-n-i} in the finite volume case. 
\end{theorem}

In what follows, we will establish the existence of stationary measures to the mean-field SDE \eqref{mean-field-limit-n-i} in the infinite volume case.  
We first extend Theorem \ref{GaFF} to the infinite volume case by verifying the self-consistency condition. Define 
\begin{equation*}
    \begin{aligned}
         & \mathscr{P}_{\rm const}(\mathbb{H}) := \left\{
            \mu \in \mathscr{P}(\mathbb{H})
            \;\middle|\;
            \text{there exists } u \in \mathbb{R} \text{ such that } 
            \int v_x v_y\, \mathrm{d}\mu(v) = u, \text{ for all } x \sim y
         \right\}.
    \end{aligned}
\end{equation*}
where $\mathbb{H}$ denotes a weighted $\ell^p$ space, which will be introduced in Section~\ref{sec4} below. In addition, we define  $$\kappa_c:=\frac{1}{4}G_{\mathbb{Z}^d,0}(\bm{0},\bm{0}),$$ 
where $G_{\mathbb{Z}^d,0}$ denotes the Green function of the Gaussian free field on $\mathbb{Z}^d$.  We refer to the definition of $G_{\mathbb{Z}^d,m^2}$ in Section \ref{sec2.2} below.   Let $\muc$ be a  massive Gaussian free field on $\mathbb{Z}^d$ scaled by ${1}/{(2\sqrt{\kappa})}$ with the mass $m^2$, which is the solution to 
\begin{equation*}
    \begin{aligned}
         & G_{\mathbb{Z}^d,m^2}(x,x)=4\kappa,\quad \text{for all\ } x\in \mathbb{Z}^d.
    \end{aligned}
\end{equation*}
\begin{theorem}
Suppose that $0<\kappa<\kappa_{c}$. Then, the field $\muc$ is the unique stationary measure in $\mathscr{P}_{\mathrm{const}}(\mathbb{H})$ to the mean-field SDE \eqref{mean-field-limit-n-i} in the infinite volume case. 
\end{theorem}
For $\kappa\ge\kappa_c$, we can also find a stationary measure for the mean-field SDE \eqref{mean-field-limit-n-i}.  More precisely, let $\bm{\mu}$ be a field on $\mathbb{Z}^d$, defined as follows:
\begin{enumerate}
\item if $\kappa<\kappa_c$: let $\bm{\mu}=\muc$;

\item if $\kappa=\kappa_c$: let $\bm{\mu}$ be a  Gaussian free field on $\mathbb{Z}^d$ scaled by ${1}/{(2\sqrt{\kappa})}$;

\item if $\kappa>\kappa_c$: let $\bm{\mu}$ be a  Gaussian free field on $\mathbb{Z}^d$ scaled by ${1}/{(2\sqrt{\kappa})}$ plus an independent constant random drift $\sqrt{\frac{\kappa-\kappa_c}{\kappa}}\cdot \mathcal Z$ with $\mathcal Z$ being a standard normal random variable.
\end{enumerate}

\begin{theorem}\label{mu-sm-i}
The field $\bm{\mu}$ is a stationary measure to the mean-field SDE \eqref{mean-field-limit-n-i} in the infinite volume case. 
\end{theorem}
The field $\bm{\mu}$ arises as the infinite volume limit of $\mu^\L$ (see \cite{aru2024} for more details).  By combining the arguments from \cite[Section 4.2]{aru2024} with a periodic approximation of \eqref{mean-field-limit-n-i} (where $\mathfrak{A} = \mathbb{Z}^d$),  we obtain the stationary measure $\bm{\mu}$ to the mean-field SDE \eqref{mean-field-limit-n-i} via a family of periodic mean-field SDEs, as given in \eqref{mean-field-limit-L} below.

For small inverse temperatures, we derive the convergence rate of the spin $O(N)$ model by employing a dynamical method from \cite{Shen2022,Shen2023}. Let $\mathbb{V}^{(k)}: \mathfrak{X}^{N} \rightarrow \mathfrak{X}^{k}$ be the projection onto the first $k$ components, defined by $\mathbb{V}^{(k)}(\Phi)=(\Phi^1,\cdots,\Phi^k)$, where $\mathfrak{X}$ is a set, and $\mathfrak{X}^N$ denotes the $N$-fold Cartesian product of $\mathfrak{X}$. In the infinite volume case, by Proposition \ref{infinite-measure} below,  let $\mu_{N,\kappa}$ denote the invariant measure for the Langevin dynamics \eqref{the-new-model-i} on $\mathbb{Z}^d$. We define $\mu^{(k)}_{\Lambda,N,\kappa}:=\mu_{\Lambda,N,\kappa}\circ (\mathbb{V}^{(k)})^{-1}$ and $\mu^{(k)}_{N,\kappa}:=\mu_{N,\kappa}\circ (\mathbb{V}^{(k)})^{-1}$. Recall $\mu^\L$ and $\bm{\mu}$ from Theorems \ref{GaFF} and \ref{mu-sm-i}. Let $(\mu^\L)^{\otimes k}$ and $\bm{\mu}^{\otimes k}$ denote the $k$-fold product measures of $\mu^\L$ and $\bm{\mu}$, respectively.

\begin{theorem}\label{O-con}
Suppose $0 < \kappa < 1/(32d)$ and $k\in \mathbb{N}$ is fixed. Then $\mu^{(k)}_{\Lambda,N,\kappa}$ converges to $(\mu^{\L})^{\otimes k}$ at rate $N^{-1/2}$ in the $2$-Wasserstein distance. Furthermore, under the same condition, $\mu^{(k)}_{N,\kappa}$ converges to $\bm{\mu}^{\otimes k}$ at the same rate in the $2$-Wasserstein distance.  
\end{theorem}
We refer the reader to Section \ref{sec2} for the definition of the $2$-Wasserstein distance. See Theorems \ref{W2MU} and \ref{cov-GFF} below for precise statements.

The proof is similar in both the finite and infinite volume cases. We follow the coupling approach as in \cite[Lemma 5.7]{Shen2022} and construct a jointly stationary process $(\Phi, \Psi)$, whose components satisfy \eqref{the-new-model-i} and \eqref{Psi-i}, respectively. By the stationary coupling and the symmetry of $\Phi - \Psi$, we transform the problem to estimate $$\frac{1}{N}\sum_{i=1}^N \mathbb{E}\left[|\Phi^i - \Psi^i|^2\right].$$  This is similar to the approach used in Theorem \ref{simple-2}, where the rate of convergence is established.

In the following corollary, we consider the periodic extensions of $\mu_{\Lambda_\L,N,\kappa}$ and $\mu^\L$, which are still denoted as $\mu_{\Lambda_\L,N,\kappa}$ and $\mu^\L$ for simplicity.  Let $\mu^i_{\Lambda_\L,N,\kappa}$ denote the marginal distribution of $\mu_{\Lambda_\L,N,\kappa}$. Theorem \ref{O-con} allows us to derive an interesting conclusion: for small $\kappa$, the order of taking the limits of infinite spin-dimensionality and infinite volume can be interchanged. Specifically, 

\begin{corollary}\label{commutes}
If $0<\kappa<1/(32d)$, then the following diagram commutes:  
\begin{center}
      \begin{tikzcd}[row sep=3em, column sep=5em]
{\mu^{i}_{\Lambda_\L,N,\kappa}} \arrow[r, "\text{$N\rightarrow \infty$}"] \arrow[d, "\text{$\L\rightarrow \infty$}"'] & \mu^\L \arrow[d, "\text{$\L\rightarrow \infty$}"] \\
 \mu^i_{N,\kappa} \arrow[r, "\text{$N\rightarrow \infty$}"] & \bm{\mu},
\end{tikzcd}  
\end{center} 
where all limits are taken in the weak sense.  
\end{corollary}

Since the dynamics for $\Psi$ depends on the law of $\Psi$ itself  and the term $\mathbb{E}[\Psi_x\Psi_y]\Psi_x$ is nonlinear and non-dissipative, the uniqueness of the stationary measure for such equation is more involved. Fortunately, the mean-field SDE \eqref{mean-field-limit-n-i} contains a coercive term $-\frac{1}{2} \Psi_x$. Therefore, for small $\kappa$, we are able to establish the uniqueness of the stationary measure to the mean-field SDE \eqref{mean-field-limit-n-i}.

\begin{theorem}\label{unique-in-m}
If $0 < \kappa < 1/(32d)$,  then there exists at most one stationary measure to the mean-field SDE \eqref{mean-field-limit-n-i} for both the finite and infinite volume cases.
\end{theorem}
See Theorems \ref{unique-in-m-c} and \ref{unique-in-m-c-2} below for precise statements.

\subsection{Related results, literature, and discussions} 
The spin $O(N)$ model was first introduced by Stanley \cite{Stanley1968} in 1968. Stanley generalized the Ising model, the XY model, and the classical Heisenberg model (see \cite{Lenz1920},\cite{Vaks1966} and \cite{Heller1934}) by allowing
spins to be on the sphere $\mathbb{S}^{N-1}(\sqrt{N})$ in dimension $N$.  Stanley also observed a connection between the spin $O(N)$ model and the spherical model in \cite{Stanley1968a}: as $N \to \infty$, the spin $O(N)$ model is considerably simplified, and the free energy of the spin $O(N)$ model approaches that of the sphere model. The spin $O(N)$ model also appears in \cite{Kupiainen1980c}, where asymptotic expansions in $1/N$ for correlation functions and the free energy at all temperatures above the critical temperature of the spherical model have been established. Recently, Aru and Korzhenkova revisited the relationship between the sphere model and the spin $O(N)$ model in the limit as $N \rightarrow \infty$, explaining how they are connected through the discrete Gaussian free field. It is obtained in  \cite[Theorem 4.1]{aru2024} that the spin $O(N)$ measures are chaotic.
Theorems \ref{simple-1} and \ref{simple-2} in our paper can be viewed as propagation of chaos results for the Langevin dynamics of the spin $O(N)$ model.

The Large $N$ limit (or "$1/N$ expansion") is widely used in models of quantum field theory and statistical physics.
Wilson first introduced this idea in quantum field theory \cite{Wilson1973} for $\Phi^4$-type and fermionic models, and it was soon popularized and extended to many other systems (see the review articles \cite{Witten1980}, \cite[Chapter 8]{Coleman1985}, and \cite{Moshe2003} for summaries of this progress). 

% The rigorous mathematical study of large $N$ limits was initiated by Kupiainen \cite{Kupiainen1980c,Kupiainen1980b,Kupiainen1980a}. 

Recently, employing dynamics to control fields and obtain qualitative results for quantum field theory models has become increasingly popular (see, for example, \cite{Albeverio2020,Gubinelli2021,Hairer2022,Shen2022,Shen2022a,Shen2023a,Shen2023,Shen2024,Shen2025}). 
In \cite{Shen2024}, the Langevin dynamics for the Yang-Mills-Higgs model were studied. Shen, Zhu, and the second-named author focused on ergodicity properties of the dynamics, functional inequalities for Yang-Mills-Higgs measures, and exponential decay of correlations. The Langevin dynamics for the spin $O(N)$ model can be seen as a specific case where Yang-Mills component $Q\equiv 1$, and the Higgs component is valued  in the sphere $\mathbb{S}^{N-1}(\sqrt{N})$.

As mentioned above,  Theorems \ref{simple-1} and \ref{simple-2} in our paper can be regarded as results on the propagation of chaos of interacting particle system. To the best of our knowledge, Kac, in his seminal article \cite{Kac1956}, provided the first rigorous mathematical definition of chaos and introduced the concept of the propagation of chaos. Shortly thereafter, McKean \cite{McKean1967} introduced a class of diffusion models satisfying Kac's propagation of chaos property. There is now an extensive literature on mean-field SDEs, and we do not attempt to list all relevant papers here. Interested readers are referred to the survey articles \cite{Sznitman1991,Jabin2017,Chaintron2022a,Chaintron2022b}. In what follows, we discuss only those results closely related to our work. Recall the Poincar\'e lemma (see, for instance, \cite{McKean1973}), which states that the uniform distributions on $\mathbb{S}^{N-1}(\sqrt{N})$ are chaotic, and that the marginals of these distributions in high dimensions converge to standard normal distributions. This well-known lemma 
can be seen as a special case of \cite[Theorem 4.1]{aru2024} where the inverse temperature $\beta=0$, and there is only a single spin on the sphere $\mathbb{S}^{N-1}(\sqrt{N})$. In \cite{Shen2022,Shen2022a,Shen2025}, the large $N$ limit of an $N$ interacting dynamical $\Phi^4$ equation and the related $O(N)$ linear sigma model over two-dimensional and three dimensional torus were studied, where the large $N$ limit is a mean-field singular SPDE with polynomial growth.  
Our work can be viewed as an extension of \cite{Shen2022,Shen2025} from the $\Phi^4$ field to the spin $O(N)$ model. Compared with \cite{Shen2022,Shen2025}, our model is defined on a lattice, so we do not need to handle the singularity from small scales. However, our model takes values in the sphere, which introduces new difficulties. For example, the well-posedness of the mean-field SDE \eqref{mean-field-limit-n-i} does not follow from uniform estimates as in \cite{Shen2022}, but instead requires approximation by the Langevin dynamics \eqref{the-new-model-i}.

\subsection{Structure of the paper}
This paper is structured as follows. In Section \ref{sec2}, we introduce notations used throughout the paper and provide the definition of the massive Gaussian free field. Section \ref{sec3} is devoted to the proof of Theorem \ref{simple-1} and Theorem \ref{simple-2} for the finite volume ($\mathfrak{A} = \Lambda$).  First, we give the definition of the martingale problem associated with the mean-field SDE \eqref{mean-field-limit-n-i} and state precisely Theorem \ref{simple-1} and Theorem \ref{simple-2}. Second, we provide uniform in $N$ bounds on the solutions to the dynamics \eqref{the-new-model-i} and prove the tightness of $\{S^N\}_{N\in \mathbb{N}}$ in $(\mathscr{P}_2(\SSF), \WF)$. Finally, we combine the compactness method, the Yamada-Watanabe theorem argument, and a coupling argument to complete the proofs of Theorem \ref{simple-1} and Theorem \ref{simple-2}. 
Section \ref{sec4} is devoted to the proof of Theorem \ref{simple-1} and Theorem \ref{simple-2} for the infinite volume ($\mathfrak{A}=\mathbb{Z}^d$),
it has exactly the same structure as Section \ref{sec3}, and the proofs are also very similar, so we will only present the essential details.  In Section \ref{Invariant measure},  we focus on the stationary measures to the mean-field SDE. First, we study the existence of stationary measures for the mean-field SDE \eqref{mean-field-limit-n-i}. Second, we investigate the convergence of the stationary measures using Langevin dynamics and It\^o's calculus. Finally,
we prove that there exists at most one stationary measure to the mean-field SDE \eqref{mean-field-limit-n-i} for both the finite and infinite volume cases. 

\section*{List of Notations}
Below we collect some frequently used notations for the readers' convenience.

\renewcommand{\arraystretch}{1.15}

\begin{longtable}{p{4cm} p{11cm}}
% \caption{List of Notations} \\
\hline
\textbf{Symbol} & \textbf{Description} \\
\hline
\endfirsthead

% \multicolumn{2}{c}%
% {{\bfseries \tablename\ \thetable{} -- continued from previous page}} \\
\hline \textbf{Symbol} & \textbf{Description} \\
\hline
\endhead

\hline \multicolumn{2}{r}{{Continued on next page}} \\
\endfoot

\hline
\endlastfoot

$\Lambda_\mathfrak{L}:=\mathbb{Z}^d\cap \mathfrak{L}\mathbb{T}^d$ or $\Lambda$ & a finite $d$-dimensional lattice with side length $\mathfrak{L}>0$ and unit lattice spacing. \\
$\mathbb S^{N-1}(\sqrt{N})$ & $(N-1)$-dimensional sphere of radius $\sqrt{N}$ centered at the origin. \\
$\mathcal{S}$ &  $(\mathbb{S}^{N-1}(\sqrt{N}))^{\Lambda}\subset (\mathbb{R}^{\abs{\Lambda}})^{N}$. \\
$\mathbb{R}^{\abs{\Lambda}}$ or $\mathbb{R}^{\Lambda}$ & $\L^d$-dimensional Euclidean space. \\
$\abs{\cdot}$ & Euclidean distance, graph distance, or the cardinality of a set. \\
$\sum_{x\sim y}$ & sum over all neighbouring points $x,y$. \\
$\sum_{\N}$ & sum over all points adjacent to $x$. \\
$\SSF=C([0,T];\mathbb{R}^{\abs{\Lambda}})$ & the space of all continuous functions on $[0,T]$ with values in $\mathbb{R}^{\abs{\Lambda}}$. \\
$\dF$ &   $\dF(w_1,w_2):=\sup_{0\le t\le T} \abs{w_1(t)-w_2(t)}$, for all $w_1,w_2\in \SSF$.\\
$\mathscr{P}(\SSF)$& the set of all probability measures on $\SSF$.\\
$\dLF$ & the dual-Lipschitz metric on $\mathscr{P}(\SSF)$. \\
$\mathscr{P}_2(\SSF)$   & the set of all $\mu\in \mathscr{P}({\SSF})$ s.t. $\int_{\SSF} \dF(x_0, x)^2\, \mu(\mathrm{d}x) < \infty$, for some $x_0 \in \SSF$.\\
$\WF$   & the $2$-Wasserstein distance on $\mathscr{P}_2(\SSF)$.  \\
$S^N$ & the empirical measure to the Langevin dynamics \eqref{the-new-model} or \eqref{the-new-model-Z}. \\
$v=(v_x)$ & the vector indexed by $x\in \Lambda$ or by $x\in \mathbb{Z}^d$. \\
$\ell^p$   & the space $\{v:\mathbb{Z}^d\rightarrow \mathbb{R}:\sum_{x\in \mathbb{Z}^d} \abs{v_x}^p<\infty \}$.    \\
$\ell^p_a$ & the space $\{ v:\mathbb{Z}^d\rightarrow \mathbb{R}: \sum_{x\in \mathbb{Z}^d}\frac{1}{a^{\abs{x}}}\cdot\abs{v_x}^p<\infty\}$, for $a>1$. \\
$\norm{\cdot}_p$ &  $\norm{v}_p:=\left(\sum_{x\in \mathbb{Z}^d} \frac{1}{a^{\abs{x}}}\abs{v_x}^p\right)^{\frac{1}{p}}$, for some fixed $a>1$, $v\in \ell^p_a$. \\
$\norm{\cdot}$  & $\norm{v}:=(\sum_{x\in \mathbb{Z}^d} \abs{v_x}^2)^{\frac{1}{2}}$, for $v\in \ell^2$. \\
$\langle\cdot,\cdot\rangle$ & the standard inner product on $\mathbb{R}^{\abs{\Lambda}}$ or the $\ell^2$ inner product.   \\
$\mathbb{H}$ & the Hilbert space $\ell^2_a$. \\
$\SSI=C([0,T];\mathbb{H})$ & the space of all continuous functions on $[0,T]$ with values in $\mathbb{H}$. \\
$\dI$ &   $\dI(w_1,w_2):=\sup_{0\le t\le T} \norm{w_1(t)-w_2(t)}_2$, for all $w_1,w_2\in \SSI$.\\
$\mathscr{P}(\SSI)$& the set of all probability measures on $\SSI$.\\
$\dLI$ & the dual-Lipschitz metric on $\mathscr{P}(\SSI)$. \\
$\mathscr{P}_2(\SSI)$   &  the set of all $\mu\in \mathscr{P}({\SSI})$ s.t. $\int_{\SSI} \dI(x_0, x)^2\, \mu(\mathrm{d}x) < \infty$, for some $x_0 \in \SSI$.\\
$\WI$   & the $2$-Wasserstein distance on $\mathscr{P}_2(\SSI)$. \\
$\mathcal Q$ & $\ell^p_a(\mathbb{Z}^d;\mathbb{S}^{N-1}(\sqrt N))\subset(\ell^p_a)^{N}$, for some fixed $a>1$. \\
$\Phi=(\Phi_x)$ & the solution to the Langevin dynamics \eqref{the-new-model} or \eqref{the-new-model-Z}. \\
$\Psi=(\Psi_x)$ & the solution to the mean-field SDE \eqref{mean-field-limit} or \eqref{mean-field-limit-Z}.\\
$\Psi^\mathfrak{L}=(\Psi_x^\mathfrak{L})$ & the solution to the mean-field SDE \eqref{mean-field-limit-L}. \\
$\mathbf{P}^\mathfrak{L}$ & the law of $\Psi^\mathfrak{L}$. \\
$\mathbf{P}$ or $S$ & the martingale solution to the mean-field SDE \eqref{mean-field-limit} or \eqref{mean-field-limit-Z}. \\

\end{longtable}

\section{Preliminaries} \label{sec2}
\subsection{Definitions and notation}\label{sec2.1}
We use the symbol $:=$ to denote definitions. Let $\mathbb{N}$ be the set of all positive integers. For any $a, b \in \mathbb{R}$, define $a \vee b := \max\{a, b\}$ and $a \wedge b := \min\{a, b\}$. We write $a \lesssim b$ to indicate that there exists a constant $C$ such that $a \leq Cb$. Let $\mathbb{T}^d$ denote the $d$-dimensional torus, $\mathbb{Z}^d$ the $d$-dimensional integer lattice, and $\Lambda_{\mathfrak{L}} := \mathbb{Z}^d \cap \mathfrak{L}\mathbb{T}^d$ the finite $d$-dimensional lattice with side length $\mathfrak{L}>0$ and unit lattice spacing. For notational simplicity, we sometimes write $\Lambda=\Lambda_\L$. Let $\mathfrak{X}$ be a set. We denote by $\mathfrak{X}^{\Lambda}$ (respectively, $\mathfrak{X}^{\mathbb{Z}^d}$) the set of all $\mathfrak{X}$-valued maps on $\Lambda$ (respectively, $\mathbb{Z}^d$).  For $u \in \mathfrak{X}^{\Lambda}$ or $u \in \mathfrak{X}^{\mathbb{Z}^d}$, we write $u=(u_x)$. The notation $\abs{\cdot}$ is used in several contexts: for $x,y\in\mathbb{Z}^d$, $\abs{x-y}$ denotes the graph distance; for $u,v\in\mathbb{R}^d$, $\abs{u-v}$ is the Euclidean distance; and for the finite set $\Lambda$, $\abs{\Lambda}$ denotes its cardinality.  For any $x, y \in \mathbb{Z}^d$, we write $x \sim y$ if $\abs{x-y} = 1$. The notation $\langle\cdot,\cdot\rangle$ denotes the Euclidean inner product or $\ell^2$ inner product.  The space of compactly supported continuous functions on $\mathbb{R}^{|\Lambda|}$ is denoted by $C_c(\mathbb{R}^{|\Lambda|}; \mathbb{R})$.  The symbol $\laweq$ is employed in two related contexts. First, for two random variables $Y_1$ and $Y_2$, the notation $Y_1 \laweq Y_2$ denotes that $Y_1$ and $Y_2$ have the same distribution. Second, for a random variable $Y$ and a probability measure $\mu$, the notation $Y \laweq \mu$ indicates that the law of $Y$ is $\mu$. We write $\mu_n \wcon \mu$ if the sequence $\{\mu_n\}_{n \in \mathbb{N}}$ converges weakly to $\mu$. Similarly, $Y_n \lawcon Y$ means that the sequence $\{Y_n\}_{n \in \mathbb{N}}$ converges in law to $Y$.

Let $\X$ be a Polish space equipped with a metric $\bd$. We denote $\mathscr{P}(\X)$ as the set of all probability measures on $\X$, and $BL(\X,\bd)$ as the set of all bounded, real-valued Lipschitz functions on $\X$, i.e.,
\[
BL(\X,\bd) := \left\{ f : \X \rightarrow \mathbb{R} : \norm{f}_{BL} < \infty \right\},
\]
where $\norm{f}_{BL} := \sup_{x \in \X} \abs{f(x)} + \sup_{x \neq y\in \X} \frac{\abs{f(x) - f(y)}}{\bd(x, y)}$. For any probability measures $\mu$ and $\nu$ on $\X$, we define
\[
\dL(\mu, \nu) := \sup_{\norm{f}_{BL} \le 1} \abs{\int_{\X} f\,\mu(\d x) - \int_{\X} f\,\nu(\d x)}.
\]
Then, $\mu_n \wcon \mu$ if and only if $\dL(\mu_n, \mu) \rightarrow 0$ (see \cite[Theorem 11.3.3]{Dudley2002}). The metric $\dL$ is called the dual-Lipschitz metric on $\mathscr{P}(\X)$.

For any $p\ge 1$, let
\[
\mathscr{P}_p(\X) := \left\{ \mu \in \mathscr{P}(\X) : \int_{\X} \bd(x_0,x)^p \, \mu(\d x) < \infty \text{ for some } x_0\in\X \right\}
\]
be the Wasserstein space of order $p$ (see \cite[Definition $6.4$]{Villani2009}). If $\mu, \nu \in \mathscr{P}(\X)$, let $\bm{\pi}(\mu, \nu)$ denote the set of probability measures $\tilde{\pi} \in \mathscr{P}(\X \times \X)$ whose marginals are $\mu$ and $\nu$. The $p$-Wasserstein distance between $\mu$ and $\nu$ is defined as
\[
\mathbf{W}_p(\mu, \nu) = \left( \inf_{\tilde{\pi} \in \bm{\pi}(\mu, \nu)} \int_{\X \times \X} \bd(x_1 , x_2)^p \,\tilde{\pi}(\d x_1 \d x_2) \right)^{\frac{1}{p}}.
\]
The space $\mathscr{P}_p(\X)$ is a Polish space when equipped with the $p$-Wasserstein distance.

\subsection{Massive Gaussian free field} \label{sec2.2}
In this section, we recall the definition of the massive Gaussian free field on $\Lambda$ and $\mathbb{Z}^d$ (see, for instance, \cite[Section 2.2]{aru2024}). Let $m^2 > 0$. A real-valued function on $\Lambda$, denoted by $\bm{\varphi} = (\varphi_x) \in \mathbb{R}^{\abs{\Lambda}}$, is called an $m^2$-massive Gaussian free field on $\Lambda$ if its law $\mathbb{P}_{\Lambda, m^2}$ is given by
\begin{equation*}\label{exp-GFF}
    \frac{1}{Z_{\Lambda, m^2}} \exp\left( -\frac{1}{4} \sum_{x \sim y} (\varphi_x - \varphi_y)^2 - \frac{m^2}{2} \sum_{x \in \Lambda} \varphi_x^2 \right) \prod_{x \in \Lambda} \d\varphi_x,
\end{equation*}
where the sum over $x\sim y$ runs over all neighbouring points $x,y\in \Lambda$, and $Z_{\Lambda, m^2}$ is the normalization constant that ensures $\mathbb{P}_{\Lambda, m^2}$ is a probability measure. Note that
\[
-\frac{1}{4} \sum_{x \sim y} (\varphi_x - \varphi_y)^2 - \frac{m^2}{2} \sum_{x \in \Lambda} \varphi_x^2
= -\frac{1}{2} \langle \bm{\varphi}, ( -\Delta_\Lambda + m^2) \bm{\varphi} \rangle,
\]
where $\Delta_\Lambda$ is the discrete Laplacian on $\Lambda$, defined by
\[
(\Delta_\Lambda f)(x) := \sum_{\N} \left( f(y) - f(x) \right), \quad \text{for any}\  x \in \Lambda,\, f: \Lambda \to \mathbb{R}.
\]
It is straightforward to verify that $(-\Delta_\Lambda + m^2)$ is invertible as a matrix in $\mathbb{R}^{\abs{\Lambda} \times \abs{\Lambda}}$. Therefore, the massive Gaussian free field can be identified with the Gaussian measure $N(0, (-\Delta_\Lambda + m^2)^{-1})$. We define $G_{\Lambda, m^2} := (-\Delta_\Lambda + m^2)^{-1}$ as the massive Green function on $\Lambda$.

Let $m^2 > 0$ if $d = 1, 2$, or $m^2 \geq 0$ otherwise. An $m^2$-massive Gaussian free field on $\mathbb{Z}^d$ is a centered Gaussian process indexed by $\mathbb{Z}^d$ with covariance function $(G_{\mathbb{Z}^d, m^2}(x, y))_{x, y \in \mathbb{Z}^d}$, where $G_{\mathbb{Z}^d, m^2}$ is the massive Green function on $\mathbb{Z}^d$. If $m^2=0$, we call this centered Gaussian process a Gaussian free field on $\mathbb{Z}^d$.

\section{Large $N$ limit of dynamics on finite lattice}\label{sec3}

In this section, we will prove Theorem \ref{simple-1} and Theorem \ref{simple-2} for the finite lattice $\Lambda$.  To prove Theorem \ref{simple-1}, the main idea is to formulate the mean-field SDE \eqref{mean-field-limit} as a martingale problem, and then to construct a solution to this martingale problem by employing the empirical measure $S^N$ associated with the Langevin dynamics \eqref{the-new-model}. In Section \ref{sec3.1},  we establish the well-posedness of the Langevin dynamics \eqref{the-new-model} and rewrite the SDEs \eqref{the-new-model} and \eqref{mean-field-limit} in terms of martingale problems.
In Sections \ref{sec3.3}, we provide bounds that are uniform in $N$ for the solutions to the SDEs \eqref{the-new-model}, which can help us establish the compactness of the empirical measure $S^N$.  In Section \ref{sec3.4}, we prove that the limit $S$ of the sequence $\{S^N\}_{N\in \mathbb{N}}$ is a martingale solution associated with the mean-field SDE \eqref{mean-field-limit}.   In Section~\ref{sec3.5}, in order to obtain a quantitative convergence rate in Theorem~\ref{strong-solution},  we employ a classical coupling argument and the following fact: for $p=1$ or $2$, if $U_1, \ldots, U_N$ are independent, identically distributed, mean-zero random variables with finite $2p$-th moment, then
$$
\mathbb{E}\left[\left| \frac{1}{N} \sum_{i=1}^N U_i \right|^{2p}\right] \lesssim \frac{1}{N^{p}},
$$
where the implicit constant is independent of $N$.

\subsection{Langevin dynamics} \label{sec3.1}
The Langevin dynamics for the spin $O(N)$ model is given by the following SDEs: 
\begin{equation}\label{the-new-model}
    \left\{\begin{aligned}
         \d\Phi^i_x(t)= &2\kappa\sum_{\N}\left( \Phi_y^i(t)-\frac{\sum^N_{k=1} \Phi_x^k(t) \Phi_y^k(t)}{N} \Phi_x^i(t)\right)\,\d t-\frac{N-1}{2N} \Phi_x^i(t)\,\d t\\
&+\,\d W_x^i(t)-\frac{\sum^N_{k=1} \Phi_x^k(t)\,\d W_x^k(t) }{N}\Phi_x^i(t),\quad 0\le t\le T,\\
\Phi_x^i(0)=&\varphi_x^i,\quad x\in \Lambda,\quad i=1,\cdots,N,
\end{aligned}\right.
\end{equation}
where $\kappa>0$;  for each $x \in \Lambda$, $\bm{W}_x := (W_x^1, \cdots, W_x^N)$ is an $N$-dimensional standard Brownian motion on the probability space $(\Omega, \mathscr{F}, \mathbb{P})$, and for $x \neq y$, the processes $\bm{W}_x$ and $\bm{W}_y$ are independent; the sum over $\N$ runs over all points adjacent to $x$ (see \cite[Section $3.2$]{Shen2024} for more details on the derivation).  We first consider the well-posedness of the SDEs \eqref{the-new-model}.  For convenience, in this section, we denote a vector indexed by $x \in \Lambda$ as $(v_x)$, which contains $\mathfrak L^d$ elements. Let $\mathbb{S}^{N-1}(\sqrt{N})$ be the $(N{-}1)$-dimensional sphere of radius $\sqrt{N}$ centered at the origin, embedded in the Euclidean space $\mathbb{R}^N$. The configuration space of our fields is given by the product of $\mathbb{S}^{N-1}(\sqrt{N})$ over the lattice vertices, denoted as $\mathcal{S} := (\mathbb{S}^{N-1}(\sqrt{N}))^{\Lambda} \subset (\mathbb{R}^{\abs{\Lambda}})^{N}$.  Note that $\mathcal S$ depends on $N$; however, we omit this dependence for notational simplicity. Furthermore, we define $C([0,T];\mathcal{S})$ as the space of all continuous functions from $[0,T]$ to $\mathcal{S}$.

We show in the next proposition that the system \eqref{the-new-model} is globally well-posed. 
\begin{proposition}\label{invariance-SDE}
  For fixed $N\in \mathbb{N}$, $\kappa>0$, and $T>0$, given any initial data $\Phi(0)=(\varphi_x)\in \mathcal{S}$, there exists a unique probabilistically strong solution $\Phi=(\Phi_x)\in C([0,T];\mathcal{S})$ to \eqref{the-new-model}. 
\end{proposition}
\begin{proof}
For a fixed $N$, the system \eqref{the-new-model} can be viewed as a finite-dimensional SDE with locally Lipschitz coefficients. By the classical theory of SDEs (see, for example, \cite{liu2015}),  there exists a local solution $\Phi=(\Phi_x)\in C([0,\tau];(\mathbb{R}^N)^\Lambda)$ to \eqref{the-new-model}, where $\tau$ is a stopping time. By the derivation of \eqref{the-new-model} in \cite[Section~3.2]{Shen2024} and the argument in \cite[Lemma~3.2]{Shen2024-0}, since $\Phi(0) = (\varphi_x) \in \mathcal{S}$, we obtain that $\Phi(t) \in \mathcal{S}$ for all $t$. Therefore, if $\Phi(0) = (\varphi_x) \in \mathcal{S}$, then $\tau=T$ a.s. and $\Phi=(\Phi_x)\in C([0,T];\mathcal S)$.
\end{proof}

As discussed in the introduction,  the system \eqref{the-new-model} formally converges to the following mean-field SDE:
\begin{equation}\label{mean-field-limit}
    \left\{\begin{aligned}
         \d\Psi_x(t)= &2\kappa\sum_{\N}\Big(  \Psi_y(t)-\mathbb{E}[ \Psi_x(t) \Psi_y(t)] \Psi_x(t)\Big)\,\d t-\frac{1}{2}\Psi_x(t)\,\d t+\d W_x(t),\quad 0\le t\le T,\\
\Psi_x(0)=& \psi_x, \quad x\in \Lambda,
\end{aligned}\right.
\end{equation}
where $(W_x)$ is an $\mathfrak{L}^d$-dimensional standard Brownian motion on the probability space $(\Omega, \mathscr{F}, \mathbb{P})$, and $(\psi_x)$ satisfies specific assumptions. Note that the term $-\mathbb{E}[ \Psi_x \Psi_y] \Psi_x$ in \eqref{mean-field-limit} is a polynomial function of $(\Psi_x)$, and it does not satisfy the usual coercivity condition. Therefore, we cannot directly apply results from the literature (see, e.g., \cite{hong2023}) to derive a priori estimate and obtain a global-in-time solution.  Instead, we transform \eqref{mean-field-limit} to a martingale problem.

We first write the Langevin dynamics \eqref{the-new-model} in terms of the following empirical measure:
\begin{equation*}
    \begin{aligned}
         &  S^N(t):=\frac{1}{N}\sum^N_{j=1} \delta_{\Phi^j(t)},\quad  0\le t\le T,
    \end{aligned}
\end{equation*}
where $\Phi^j(t)=(\Phi^j_x(t))$  and $\delta_{\Phi^j(t)}$ is the Dirac measure at $\Phi^j(t)$.  For any $x\in \Lambda$, $v=(v_x)\in \mathbb{R}^{\abs{\Lambda}}$ and $\nu\in \mathscr{P}(\mathbb{R}^{\abs{\Lambda}})$,  we define
\begin{equation}\label{vector-Lambda}
    \begin{aligned}
         & \mathbf{D}_x^N(v,\nu):=2\kappa\sum_{\N} v_y - 2\kappa \sum_{\N}\int  u_x u_y\,\nu(\d u) v_x -\frac{N-1}{2N} v_x.\\
    \end{aligned}
\end{equation}
For any $x\in \Lambda$ and $i=1,\cdots,N$, we define
\begin{equation}\label{vector-M}
    \begin{aligned}
         & \d \mathcal{M}_x^{N,i}(t):=\d W_x^i(t) -\frac{\sum^N_{k=1} \Phi_x^k(t)\,\d W_x^k(t) }{N} \Phi_x^i(t),\quad  0\le t\le T.
    \end{aligned}
\end{equation}
Then, \eqref{the-new-model} can be rewritten as
\begin{equation*}\label{the-new-model-rw}
    \left\{\begin{aligned}
         & \d\Phi_x^i(t)=\mathbf{D}^N_x(\Phi^i(t),S^N(t))\,\d t + \d\mathcal{M}_x^{N,i}(t),\quad 0\le t\le T,\\
& \Phi_x^i(0)=\varphi_x^i,\quad x\in \Lambda, \ i=1,\cdots, N.
\end{aligned}\right.
\end{equation*}
Similarly, for any $x\in \Lambda$, we define
\begin{equation}\label{vector-Lambda-limit}
    \begin{aligned}
         & \mathbf{D}_x(v,\nu):=2\kappa\sum_{\N} v_y - 2\kappa \sum_{\N} \int u_x u_y\,\nu(\d u) v_x - \frac{1}{2}v_x.\\
    \end{aligned}
\end{equation}
Equation \eqref{mean-field-limit} can be rewritten as
\begin{equation}\label{mean-field-limit-rw-2}
    \left\{\begin{aligned}
         & \d\Psi_x(t)=\mathbf{D}_x(\Psi(t),S(t))\,\d t + \d W_x(t),\quad  0\le t\le T,\\
& \Psi_x(0)=\psi_x,\quad x\in \Lambda,
\end{aligned}\right.
\end{equation}
where  $S(t)$ denotes the law of $\Psi(t)$. We then formulate the martingale problem to the mean-field SDE \eqref{mean-field-limit-rw-2}. To this end,
let $\Omega:=C([0,T];\mathbb{R}^{\abs{\Lambda}})=:\SSF$  be the space of all continuous functions on $[0,T]$ with values in $\mathbb{R}^{\abs{\Lambda}}$. For any $w_1,w_2 \in \SSF$, we define
\[
\dF(w_1,w_2):=\sup_{0\le t\le T} \abs{w_1(t)-w_2(t)}.
\]
Then, $(\SSF,\dF)$ is a Polish space. Let $\mathscr{P}(\SSF)$ denote the set of all probability measures on $\SSF$, and let $\mathscr{P}_2(\SSF)$ denote the Wasserstein space of order $2$, equipped with the $2$-Wasserstein distance $\WF$. Similarly, let $\mathscr{P}(\mathbb{R}^{\abs{\Lambda}})$ denote the set of all probability measures on $\mathbb{R}^{\abs{\Lambda}}$, and let $\mathscr{P}_2(\mathbb{R}^{\abs{\Lambda}})$ denote the Wasserstein space of order $2$, equipped with the $2$-Wasserstein distance $\mathbf{W}_{2,\mathbb{R}^{\abs{\Lambda}}}$. 
We use $w$ to denote a generic path in $\SSF$, and define the coordinate process by $\pi_t(w) = w_t$, $t\ge 0$. Let $\mathscr{F}_t:=\sigma\{ \pi_s: s\le t \}$, $t\ge 0$ denote the filtration generated by the coordinate process.

We now give the definition of the martingale problem associated with the mean-field SDE \eqref{mean-field-limit-rw-2}.

\begin{definition}\label{martingale-solution}
    Let $\mu_0 \in \mathscr{P}(\mathbb{R}^{\abs{\Lambda}})$. A probability measure  $\mathbf{P} \in \mathscr{P}(\SSF)$ is called a martingale solution to the mean-field SDE \eqref{mean-field-limit-rw-2} with initial distribution $\mu_0$ if the following conditions are satisfied:
\begin{enumerate}[(i)]
    \item $\mathbf{P} \circ \pi_0^{-1} = \mu_0$ and
    \begin{equation*}
        \mathbf{P}\left( w \in \SSF : \int_0^T \abs{\mathbf{D}(w_s, \mu_s)}\,\d s < \infty \right) = 1,
    \end{equation*}
    where $\mu_s := \mathbf{P} \circ \pi_s^{-1}$ and $\mathbf{D}(w_s, \mu_s) = (\mathbf{D}_x(w_s, \mu_s))$.
    \item For any $\theta \in \mathbb{R}^{\abs{\Lambda}}$, the process
    \begin{equation*}
        \mathscr{M}_\theta(t, w, \mu) := \langle w_t, \theta \rangle - \langle w_0, \theta \rangle - \int_0^t \langle \mathbf{D}(w_s, \mu_s), \theta \rangle\,\d s, \quad 0\le t\le T,
    \end{equation*}
    is a continuous square-integrable $(\mathscr{F}_t)$-martingale under $\mathbf{P}$, with quadratic variation process given by
    \begin{equation*}
        \langle \mathscr{M}_\theta \rangle (t, w, \mu) = \abs{\theta}^2 t, \quad 0\le t\le T.
    \end{equation*}
\end{enumerate}
\end{definition}

We state the assumptions imposed on the initial data.
\begin{assumption}\label{initial-value}
\begin{enumerate}[(1)]
    \item For any $N\in \mathbb{N}$,  the law of $$\Phi(0)=(\varphi^1,\cdots,\varphi^N)\in \mathcal S $$ is symmetric\footnote{This means that for any permutation $\pi$,  $(\varphi^1,\cdots,\varphi^N)   \laweq (\varphi^{\pi(1)},\cdots,\varphi^{\pi(N)})$. } on $\underbrace{\mathbb{R}^{\abs{\Lambda}}\times\cdots\times \mathbb{R}^{\abs{\Lambda}}}_{\text{$N$ in total}}$\,,  where $\varphi^i=(\varphi_x^i)$.       
    \item For the family of initial data $\{(\varphi^1,\cdots,\varphi^N)\}_{N\in \mathbb{N}}$,  the empirical measure
$$S^N_0:=\frac{1}{N} \sum^N_{j=1}\delta_{\varphi^{j}}$$ converges weakly in probability to some measure $\mu_0$ as $N\rightarrow \infty$.
\item There exist constants $C>0$ and $\p> 2$ such that
\[
\sup_N\left(\sum_{x\in \Lambda}\mathbb{E}\left[\abs{\varphi_x^{1}}^{\p}\right]\right)< C.
\]
\end{enumerate}
\end{assumption}
If there exists a family of random vectors $\{(\varphi^1, \dots, \varphi^N)\}_{N \in \mathbb{N}}$  and a measure $\mu_0$ satisfying Assumption~\ref{initial-value},  we call $\mu_0$ a measure determined by $\{(\varphi^1, \dots, \varphi^N)\}_{N \in \mathbb{N}}$. We denote by $\mathscr{P}_{\mathcal S, \p}$ the set of all measures that can be determined by some $\{(\varphi^1, \dots, \varphi^N)\}_{N \in \mathbb{N}}$, where the parameter $\p$ refers to the constant in Assumption~\ref{initial-value}$(3)$.

In the following, we show that the large $N$ limit of the system \eqref{the-new-model} with initial data given by Assumption \ref{initial-value} is \eqref{mean-field-limit}. More precisely, we have
\begin{theorem}\label{Main-Theorem}
Suppose that the initial data  $\{(\varphi^1,\cdots,\varphi^N)\}_{N\in \mathbb{N}}$ of the system \eqref{the-new-model} and $\mu_0$ satisfy Assumption \ref{initial-value}, then  the empirical measure  $S^N$ of  the solutions to the system \eqref{the-new-model} (viewed as a  $(\mathscr{P}_2(\SSF),\WF)$-valued random variable) converges in law to $S$, and the limit $S$ is the unique martingale solution to the mean-field SDE \eqref{mean-field-limit} with initial distribution $\mu_0$.   Furthermore, for the martingale solution, we have
\begin{equation}\label{Cov-Psi}
    \begin{aligned}
 \int \abs{v_x}^2\,S(t)(\d v)= 1,\quad \ \mathbb{P}\text{-\rm{a.s.}}, \quad \forall\, x\in \Lambda, \, 0\le t\le T,
    \end{aligned}
\end{equation}
and
\begin{equation}\label{POC}
    \begin{aligned}
         &  \lim_{N\rightarrow \infty} \mathbb{E}\left[\WFT(S^N,S) \right]=0.
    \end{aligned}
\end{equation}
\end{theorem}
\begin{remark}
By the uniqueness of the martingale solution to the mean-field SDE \eqref{mean-field-limit}, the limit  $S$ of the sequence $\{S^N\}_{N\in \mathbb{N}}$ is $\mathbb{P}\text{-a.s.}$ constant in  $(\mathscr{P}_2(\SSF),\WF)$. 
\end{remark}
\begin{remark}\label{typical-examples}
By Proposition~\ref{mu0assum} below, a typical example satisfying Assumption \ref{initial-value} is given by taking $\varphi$ as a random vector with distribution $\mu_{\Lambda,N,\kappa}$, and $\mu_0$ as the field $\mu^{\L}$ as defined in Theorem~\ref{Aru-cor-im} (see Proposition \ref{mu0assum} below for more details).  For any $N \in \mathbb{N}$, another example satisfying Assumption \ref{initial-value} can be obtained by setting $\varphi_x^i = 1$ for all $x \in \Lambda$ and $i = 1, \ldots, N$, and taking $\mu_0 = \delta_{1}$.
\end{remark}

To compare the system \eqref{the-new-model} with the mean-field SDE  \eqref{mean-field-limit} directly,  for any fixed $i\in \mathbb{N}$, we consider the following mean-field SDE:   
\begin{equation}\label{Psi-i-limit}
    \left\{\begin{aligned}
         \d\Psi^i_x(t)= &2\kappa\sum_{\N}\Big(  \Psi_y^i(t)-\mathbb{E}[ \Psi_x^i(t) \Psi_y^i(t)] \Psi_x^i(t)\Big)\,\d t-\frac{1}{2}\Psi_x^i(t)\,\d t+\d W^i_x(t),\quad 0\le t\le T,\\
\Psi_x^i(0)=& \psi_x^i, \quad x\in \Lambda,
\end{aligned}\right.
\end{equation}
where $\{ W_x^i \}$ is the same as in \eqref{the-new-model}.

\begin{theorem}\label{strong-solution}
If the initial data $\psi$ of the mean-field SDE \eqref{mean-field-limit}  satisfies $\psi \laweq \mu_0 \in \mathscr{P}_{\mathcal S,\p}$, then there exists a unique probabilistically strong solution to equation \eqref{mean-field-limit} starting from $\psi$. 
Furthermore, suppose that the initial data $\{(\varphi^1,\cdots,\varphi^N)\}_{N\in \mathbb{N}}$ of the system \eqref{the-new-model} and $\mu_0$ satisfy Assumption \ref{initial-value}, $\psi^i\laweq \mu_0$, and that 
\begin{equation}\label{initial-data-i}
    \begin{aligned}
         & \lim_{N\rightarrow \infty} \sum_{x\in \Lambda} \mathbb{E}\left[|\varphi^i_x-\psi^i_x|^2\right]=0.
    \end{aligned}
\end{equation} Then, we have
\begin{equation}\label{Convergence-2}
    \begin{aligned}
         & \lim_{N\rightarrow \infty}\sum_{x\in \Lambda}\mathbb{E}\left[ \sup_{0\le t\le T} \abs{\Phi_x^i(t)-\Psi_x^i(t)}^2 \right]=0,
    \end{aligned}
\end{equation}
where $\Phi^i$ and $\Psi^i$ denote the solutions to \eqref{the-new-model} and \eqref{Psi-i-limit} with initial data $\varphi^i$ and $\psi^i$, respectively. 

If the initial data $\{(\varphi^1,\cdots,\varphi^N)\}_{N\in \mathbb{N}}$ of the system \eqref{the-new-model} and  $\mu_0$ satisfy Assumption \ref{initial-value} with $\p=8$, $\psi^i\laweq \mu_0$, in addition, the random variables $\{\psi^i\}_{i\in \mathbb{N}}$ are independent and identically distributed, and the sequence $\{(\varphi^i,\psi^i)\}^N_{i=1}$ is symmetric, then we have
\begin{equation}\label{rate}
    \begin{aligned}
         & \sup_{0\le t\le T}\sum_{x\in \Lambda}\mathbb{E}\left[  \abs{\Phi_x^i(t)-\Psi_x^i(t)}^2 \right]\lesssim  \sum_{x\in \Lambda}\mathbb{E}\left[\abs{\varphi^i_x-\psi^i_x}^2\right]+ \frac{1}{N},
    \end{aligned}
\end{equation}
where the implicit constant depends on $\abs{\Lambda}$, but is independent of $N$. 
\end{theorem}
\begin{remark}\label{Phi-Psi}
For any $N \in \mathbb{N}$,  a typical example that satisfies the assumptions in Theorem \ref{strong-solution} is given by setting $\varphi_x^i = 1$ and $\psi_x^i=1$ for all $x \in \Lambda$ and $i = 1, \ldots, N$.  For small $\kappa$, by taking $\varphi \laweq \mu_{\Lambda, N, \kappa}$ and $\psi^i \laweq \mu^{\L}$ (where the field $\mu^{\L}$ is defined in Theorem~\ref{Aru-cor-im} below), we can construct $\{(\varphi^i,\psi^i)\}^N_{i=1}$ satisfying the assumptions in the second part of Theorem~\ref{strong-solution} by stationary coupling.
\end{remark}

We begin with the proof of Theorem \ref{Main-Theorem}.
\subsection{Tightness of $\{S^N\}_{N\in \mathbb{N}}$ in $(\mathscr{P}_2(\SSF),\WF)$}\label{sec3.3}
In this section, we prove that the sequence $\{ S^N \}_{N\in \mathbb{N}}$ is tight in $(\mathscr{P}_2(\SSF),\WF)$.  First, we present a proposition that provides bounds that are uniform in $N$ for the solutions to the SDEs \eqref{the-new-model}. These bounds are crucial for the proofs of Theorems \ref{Main-Theorem} and \ref{strong-solution}.

\begin{proposition}\label{Lp-estimate}
Suppose that $\{ (\Phi^i_x), i=1,\cdots,N\}$ is a solution to the system \eqref{the-new-model}. Then, for any $i=1,\cdots,N$ and any $p\ge 2$,  we have
\begin{equation}\label{sup-1-p}
    \begin{aligned}
         & \sum_{x\in \Lambda}  \mathbb{E}\left[\sup_{0\le t\le T}\abs{\Phi_x^i(t)}^p\right]\le   C(d,p,\kappa,T,\abs{\Lambda}) \left( \sum_{x\in \Lambda} \mathbb{E}\left[\abs{\varphi_x^i}^{p}\right] + 1\right),
    \end{aligned}
\end{equation}
where $C(d,p,\kappa,T,\abs{\Lambda})$ is a constant  depending on $d$, $p$, $\kappa$, $T$, $\abs{\Lambda}$, but independent of $N$.
\end{proposition}
\begin{proof}

Applying It\^o's  formula to $\abs{\Phi_x^i(t)}^p$, we obtain that for any $0\le t\le T$,   
\begin{equation}\label{Ito-Lp}
    \begin{aligned}
        \d\abs{\Phi_x^i(t)}^p=  & 2\kappa p\abs{\Phi_x^i(t)}^{p-2} \left[ \sum_{\N}\left( \Phi_y^i(t)-\frac{\sum^N_{k=1} \Phi_x^k(t) \Phi_y^k(t)}{N} \Phi_x^i(t)\right) \Phi_x^i(t)\right]\,\d t\\
&-\frac{p(N+p-2)}{2N}\abs{\Phi_x^i(t)}^p\,\d t + \frac{p(p-1)}{2}\abs{\Phi_x^i(t)}^{p-2}\,\d t \\
&+p\abs{\Phi_x^i(t)}^{p-2}\Phi_x^i(t)\,\d W_x^i-p\abs{\Phi_x^i(t)}^p\frac{\sum^{N}_{k=1}\Phi_x^k(t)\,\d W_x^k}{N}.
    \end{aligned}
\end{equation}
Next, we estimate each term on the right-hand side of \eqref{Ito-Lp} separately. By Young's inequality,
\begin{equation*}\label{D-1}
    \begin{aligned}
         &  \mathbb{E}\left[ \sup_{0\le t\le T}\int_0^t \sum_{\N}\abs{\Phi^i_x(s)}^{p-2}\Phi_x^i(s)\Phi_y^i(s)\,\d s \right]\\
\leq&\,  2d \,\mathbb{E}\left[ \int^T_0 \abs{\Phi_x^i(s)}^p\,\d s \right]+  \sum_{\N}\mathbb{E}\left[ \int^T_0 \abs{\Phi_y^i(s)}^p\,\d s\right].
    \end{aligned}
\end{equation*}
By H\"older's inequality, we have
\begin{equation}\label{sphere-holder}
    \begin{aligned}
         & \abs{\frac{1}{N}{\sum^N_{k=1} \Phi_x^k(s) \Phi_y^k(s)}} \le \left( \frac{1}{N}\sum^N_{k=1} \abs{\Phi_x^k(s)}^2  \right)^\frac{1}{2} \left( \frac{1}{N}\sum^N_{k=1} \abs{\Phi_y^k(s)}^2  \right)^\frac{1}{2}= 1,\ \ \forall\, x,\,y\in \Lambda,\ 0\le s\le t\le T,
    \end{aligned}
\end{equation}
which implies that
\begin{equation*}\label{D-2}
    \begin{aligned}
         & \mathbb{E}\left[\sup_{0\le t \le T} \int^t_0 \sum_{\N}\frac{\sum^N_{k=1} \Phi_x^k(s) \Phi_y^k(s)}{N} \abs{\Phi_x^i(s)}^p \,\d s\right]
\le 2d\,\mathbb{E}\left[ \int^T_0 \abs{\Phi_x^i(s)}^p\,\d s \right].
    \end{aligned}
\end{equation*}
By Young's inequality, we obtain
\begin{equation*}\label{D-3}
    \begin{aligned}
         & \mathbb{E}\left[\sup_{0\le t\le T}\int^t_0 \abs{\Phi_x^i(s)}^{p-2}\,\d s  \right]\le \mathbb{E}\left[ \int^T_0 \abs{\Phi_x^i(s)}^p\,\d s \right] +T.
    \end{aligned}
\end{equation*}
For the martingale terms, by the Burkholder-Davis-Gundy inequality, we have
\begin{equation*}\label{M-1}
    \begin{aligned}
          &p\, \mathbb{E}\left[ \abs{\sup_{0\le t\le T}\int^t_0 \Phi_x^i(s)\abs{\Phi_x^i(s)}^{p-2}\, \d W_x^i(s) }\right]\le C\,p\, \mathbb{E}\left[ \left(\int^T_0 \abs{\Phi_x^i(s)}^{2p-2}\,\d s \right)^{\frac{1}{2}} \right]\\
\le &\frac{1}{4}\, \mathbb{E}\left[\sup_{0\le t\le T} \abs{\Phi_x^i(t)}^{p}\right] +  C^2p^2 \mathbb{E}\left[ \int^T_0 \abs{\Phi_x^i(s)}^p\,\d s \right]+C^2p^2 T,
    \end{aligned}
\end{equation*} 
and use $\Phi\in \mathcal S$ to derive 
\begin{equation*}\label{M-2}
    \begin{aligned}
         & p\,\mathbb{E}\left[ \abs{\sup_{0\le t\le T}\int^t_0 \abs{\Phi_x^i(s)}^{p} \frac{1}{N}\sum^N_{k=1}\Phi_x^k(s)\,\d W_x^k(s)\,\d s }\right]\le Cp\, \mathbb{E}\left[  \left( \frac{1}{N}\int^T_0 \abs{\Phi_x^i(s)}^{2p} \,\d s \right)^\frac{1}{2}  \right]\\
\le & \frac{1}{4}\, \mathbb{E}\left[\sup_{0\le t\le T} \abs{\Phi_x^i(t)}^{p}\right] +  C^2p^2 \mathbb{E}\left[ \int^T_0 \abs{\Phi_x^i(s)}^p\,\d s \right].
    \end{aligned}
\end{equation*} 
Combining \eqref{Ito-Lp},  we have
\begin{equation}\label{fix-x-estimate}
    \begin{aligned}
         \mathbb{E}\left[ \sup_{0\le s\le T}\abs{\Phi_x^i(s)}^{p} \right] \le& C(d,p,\kappa) \int^T_0 
          \mathbb{E}\left[ \sup_{0\le s\le t}\abs{\Phi_x^i(s)}^{p} \right]\,\d t +2\mathbb{E}\left[\abs{\varphi_x^i}^p\right]+C(p,T)\\
&+C(p,\kappa) \sum_{\N}\int^T_0 
          \mathbb{E}\left[ \sup_{0\le s\le t}\abs{\Phi_y^i(s)}^{p} \right]\,\d t,
    \end{aligned}
\end{equation}
where the constants $C(d, p, \kappa)$, $C(p, T)$ and $C(p,\kappa)$  are both independent of $\L$. Summing over $x$, it follows that
\begin{equation*}
    \begin{aligned}
         & \sum_{x\in \Lambda} \mathbb{E}\left[ \sup_{0\le s\le T}\abs{\Phi_x^i(s)}^{p} \right]\le C'(d,p,\kappa)\int^T_0 \sum_{x\in\Lambda} 
          \mathbb{E}\left[ \sup_{0\le s\le t}\abs{\Phi_x^i(s)}^{p} \right]\,\d t+2\sum_{x\in \Lambda}\mathbb{E}\left[\abs{\varphi_x^i}^p\right]+C(p,T)\abs{\Lambda}.
    \end{aligned}
\end{equation*}
Applying Gronwall's inequality yields \eqref{sup-1-p}.
\end{proof}

Using Proposition \ref{Lp-estimate}, we can prove the following lemma.

\begin{lemma}\label{tight-P-2}
       The sequence $\{ S^N \}_{N\in \mathbb{N}}$ is tight in $(\mathscr{P}_2(\SSF),\WF)$.
\end{lemma}
\begin{proof}
{\bf Step $1$:} 
Recall that $\SSF = C([0,T]; \mathbb{R}^{\abs{\Lambda}})$ denotes the space of all continuous functions on $[0,T]$ with values in $\mathbb{R}^{\abs{\Lambda}}$, and for any $w_1, w_2 \in \SSF$, 
$$\dF(w_1, w_2) = \sup_{0 \le t \le T} \abs{w_1(t) - w_2(t)}.$$
 We prove that the sequence $\{(\Phi_x^1) \}_{N\in \mathbb{N}}$ is tight in $(\SSF,\dF)$. 

By \cite[Theorem $4.2$, Theorem $4.3$, Theorem $2.6$]{Ikeda1989}, it suffices to prove 
\begin{enumerate}[(i)]
    \item there exist constants $M>0$ and $\gamma>0$ such that
        $$ \sup_{N\in \mathbb{N}}\mathbb{E}\left[\abs{\varphi^1}^\gamma \right]\le M;
        $$
   \item there exist constants $\alpha_1>0$, $\alpha_2>0$, and $M>0$ such that for $s,t\in [0,T]$,
        $$\sup_{N\in \mathbb{N}}\mathbb{E}\left[ \abs{\Phi^1(t)-\Phi^1(s)}^{\alpha_1} \right]\le M \abs{t-s}^{1+\alpha_2}.
        $$
\end{enumerate}
Condition (i) is guaranteed by Assumption \ref{initial-value} (3). It remains to prove condition (ii). By \eqref{the-new-model}, it follows that for any $0\le s\le t\le T$, 
\begin{equation}\label{Phi1-x-p}
    \begin{aligned}
         &\mathbb{E}\left[ \abs{\Phi^1(t)-\Phi^1(s)}^p\right]\\
\lesssim& \sum_{x\in \Lambda}\mathbb{E}\left[ \abs{\Phi_x^1(t)-\Phi_x^1(s)}^p \right]\\
\lesssim& \sum_{x\in \Lambda}\mathbb{E}\left[ \abs{\int^t_s 2\kappa  \sum_{\N}  \left( \Phi_y^1 - \frac{\sum^N_{k=1}\Phi_x^k\Phi_y^k}{N} \Phi_x^1 \right)\,\d r  }^p \right]+\sum_{x\in \Lambda}\mathbb{E}\left[ \abs{\int^t_s \frac{N-1}{2N}\Phi_x^1\,\d r}^p \right]\\
&+\sum_{x\in \Lambda}\mathbb{E}\left[ \abs{\int^t_s \d W_x^1(r)}^p \right] + \sum_{x\in \Lambda}\mathbb{E}\left[\abs{\int^t_s\frac{\sum_{k=1}^N \Phi_x^k\,\d W_x^k(r)}{N}\Phi_x^1 }^p \right]\\
:=&J_1+J_2+J_3+J_4.
    \end{aligned}
\end{equation}
For $p=\p>2$, with $\p$ given in Assumption \ref{initial-value}, we will estimate each term on the right-hand side of \eqref{Phi1-x-p}  separately. By \eqref{sphere-holder}, Jensen's inequality, and Proposition \ref{Lp-estimate}, we have
\begin{equation*}
    \begin{aligned}
          J_1\lesssim & \abs{t-s}^{p-1}\int^t_s \sum_{x\in\Lambda}\sum_{\N}\mathbb{E}\left[ \abs{\Phi_y^1}^p \right]\,\d r+ \abs{t-s}^{p-1}\int^t_s \sum_{x\in\Lambda}\sum_{\N}\mathbb{E}\left[ \abs{\Phi_x^1}^p \right]\,\d r
\lesssim  \abs{t-s}^p.
    \end{aligned}
\end{equation*}
Similarly, by Jensen's inequality and Proposition \ref{Lp-estimate}, we have that $ {J_2} \lesssim \abs{t-s}^p$. By the Burkholder-Davis-Gundy inequality, it is evident that $J_3\lesssim \abs{t-s}^\frac{p}{2}$. Similarly, using the Burkholder-Davis-Gundy inequality again, we find that $J_4\lesssim \abs{t-s}^{\frac{p}{2}}$.  Together, the estimates for $J_1$, $J_2$, $J_3$, and $J_4$ imply condition (ii).

{\bf Step $2$:} Note that $\mathscr{P}(\SSF)$ denotes the set of all probability measures on $\SSF$, and  for any $\mu_1,\mu_2\in \mathscr{P}(\SSF)$, 
$$
\dLF(\mu_1,\mu_2)=\sup_{\norm{f}_{BL}\le 1} \abs{\int_{\SSF}f\mu_1(\d w) - \int_{\SSF}f\mu_2(\d w) },
$$
where $\norm{f}_{BL}:=\sup_{w\in \SSF} \abs{f(w)} + \sup_{w_1\neq w_2\in \SSF} \frac{\abs{f(w_1)-f(w_2)}}{\dF(w_1,w_2)}$. The metric $\dLF$ is the dual-Lipschitz metric on $\mathscr{P}(\SSF)$. We prove that the sequence $\{S^N \}_{N\in \mathbb{N}}$ is tight in $(\mathscr{P}({\SSF}),\dLF)$.  

Since the sequence $\{(\Phi_x^1) \}_{N\in \mathbb{N}}$ is tight in $(\SSF,\dF)$,  for any $\varepsilon > 0$, there exists a compact set $\mathscr{U}_\varepsilon$ in $\SSF$  such that 
\begin{equation*}
    \begin{aligned}
         & \mathbb{P}\left( (\Phi_x^1)\notin \mathscr{U}_\varepsilon\right)\le \varepsilon^2.
    \end{aligned}
\end{equation*}
By the symmetry of the law of $((\Phi_x^1),\cdots,(\Phi_x^N))$, we have
\begin{equation*}
    \begin{aligned}
         & \mathbb{E}\left[ S^N(\mathscr{U}^c_\varepsilon) \right]=\frac{1}{N}\sum^N_{j=1}\mathbb{P}\left( (\Phi_x^j) \notin \mathscr{U}_\varepsilon\right)\le \varepsilon^2.
    \end{aligned}
\end{equation*}
Let $\Theta_\varepsilon:=\cap^\infty_{m=1} \{ \Gamma\in \mathscr{P}(\SSF): \Gamma(\mathscr{U}_{\varepsilon 2^{-m}}) \ge 1-\varepsilon2^{-m}\} $. It is not difficult to see that $\Theta_\varepsilon$ is a compact set in $(\mathscr{P}(\SSF),\dLF)$. In fact, for any $\varepsilon'>0$, there exists a $M\in \mathbb{N}$ such that $\varepsilon 2^{-M}<\varepsilon'$. By the definition of $\Theta_\varepsilon$, for any $\Gamma\in \Theta_\varepsilon$, we have $\Gamma(\mathscr{U}_{\varepsilon 2^{-M}})\ge 1-\varepsilon 2^{-M}>1-\varepsilon'$. By \cite[Theorem $2.6$]{Ikeda1989},  we obtain that $\Theta_\varepsilon$ is relatively compact in $(\mathscr{P}(\SSF),\dLF)$.  We only need to further prove that $\Theta_\varepsilon$ is a closed subset of $(\mathscr{P}(\SSF),\dLF)$. If $\Gamma_n\in \Theta_\varepsilon$ and $\Gamma_n\wcon \Gamma$, by \cite[Proposition $2.4$ (iii)]{Ikeda1989} and $\mathscr{U}_{\varepsilon 2^{-m}}$ is also a closed set, we obtain that $\Gamma\in \Theta_\varepsilon$.

By the definition of $\Theta_\varepsilon$, we have
\begin{equation}\label{compact-S}
    \begin{aligned}
         \mathbb{P}\left( S^N\notin \Theta_\varepsilon \right)\le& \sum^{+\infty}_{m=1} \mathbb{P}\left( S^N(\mathscr{U}^c_{\varepsilon 2^{-m}}) > \varepsilon 2^{-m}\right)\\
\le& \sum^{+\infty}_{m=1} (\varepsilon 2^{-m})^{-1} \mathbb{E}\left[ S^N(\mathscr{U}^c_{\varepsilon 2^{-m}}) \right]\\
\le& \sum^{+\infty}_{m=1} \varepsilon 2^{-m}=\varepsilon.
    \end{aligned} 
\end{equation}
Thus, the sequence $\{ S^N \}_{N\in \mathbb{N}}$ is tight in $(\mathscr{P}(\SSF),\dLF)$.

{\bf Step $3$:} Let us recall that $\mathscr{P}_2(\SSF)$ denotes the Wasserstein space of order $2$, equipped with the $2$-Wasserstein distance
$$
\WF(\mu_1,\mu_2)=\left(\inf_{\tilde{\pi}\in \bm{\pi}(\mu_1,\mu_2)} \int_{\SSF\times \SSF} \dF(w_1,w_2)^2 \,\tilde{\pi}(\d w_1\d w_2) \right)^{\frac{1}{2}},\quad \mu_1,\mu_2\in \mathscr{P}_2(\SSF),
$$
where $\bm{\pi}(\mu_1,\mu_2)$ denotes the set of all probability measures $\tilde{\pi} \in \mathscr{P}(\SSF\times \SSF)$ whose marginals are $\mu_1$ and $\mu_2$. We prove that the sequence $\{S^N \}_{N\in \mathbb{N}}$ is tight in  $(\mathscr{P}_2(\SSF),\WF)$. 

Since $\{ S^N \}_{N\in \mathbb{N}}$ is tight in $(\mathscr{P}(\SSF),\dLF)$ by {\bf Step $2$}, and $\{ S^N \}_{N\in \mathbb{N}}\subset \mathscr{P}_2(\SSF),\ \mathbb{P}\text{-a.s.}$, by  Proposition \ref{Lp-estimate}, for every $\varepsilon>0$, we can choose a set $\mathcal{N}_\varepsilon \subset \mathscr{P}_2(\SSF)$ that is relatively compact in $(\mathscr{P}(\SSF),\dLF)$ such that 
\begin{equation}\label{N-e}
    \begin{aligned}
         & \mathbb{P}\left( S^N\notin \mathcal{N}_\varepsilon \right)\le \varepsilon.
    \end{aligned}
\end{equation}
In fact, due to $\{ S^N \}_{N\in \mathbb{N}}\subset \mathscr{P}_2(\SSF),\ \mathbb{P}\text{-a.s.}$, we have $$\mathbb{P}(S^N\in \Theta_\varepsilon)=\mathbb{P}(S^N\in \Theta_\varepsilon\cap \mathscr{P}_2(\SSF)).$$
Let $\mathcal{N}_\varepsilon:=\Theta_\varepsilon\cap \mathscr{P}_2(\SSF)=\cap^\infty_{m=1} \{ \Gamma\in\mathscr{P}_2(\SSF): \Gamma(\mathscr{U}_{\varepsilon 2^{-m}}) \ge 1-\varepsilon2^{-m}\}$. By following a similar approach to that used in {\bf Step $2$} for establishing the relative compactness of $\Theta_\varepsilon$ in $(\mathscr{P}(\SSF), \dLF)$ and the inequality $\mathbb{P}(S^N \notin \Theta_\varepsilon) \leq \varepsilon$ in \eqref{compact-S}, we can prove the desired conclusion.

For $p=\p>2$, with $\p$ given in Assumption \ref{initial-value}, let $a_m:=m^{\frac{1}{p-2}}2^{\frac{m}{p-2}}$, $b_m:=\frac{\varepsilon m}{C_p+1}$ with \begin{equation*}
    \begin{aligned}
         & C_p:=\sup_{N\in \mathbb{N}}\mathbb{E}\left[ \sup_{0\le t\le T} \abs{(\Phi_x^{1}(t))}^p \right] \lesssim \sup_{N\in \mathbb{N}}\sum_{x\in \Lambda}  \mathbb{E}\left[\sup_{0\le t\le T}\abs{\Phi_x^1(t)}^p\right]\lesssim 1.
    \end{aligned}
\end{equation*}
Note that $a_m,b_m\rightarrow \infty$ as $m\rightarrow \infty$.  Let
\[
H_\varepsilon:= \bigcap^{+\infty}_{m=1} \left\{  \nu\in \mathscr{P}_2(\SSF): \int \left(\sup_{0\le t\le T} \abs{w_t}\right)^2 \mathds{1}_{\left\{\sup_{0\le t\le T} \abs{w_t}\,\ge a_m\right\}}\nu(\d  w) <\frac{1}{b_m}   \right\},
\]
where $\mathds{1}_{\{\sup_{0\le t\le T} \abs{w_t}\,\ge a_m\}}$ is the indicator function of the set $\{w:\sup_{0\le t\le T} \abs{w_t}\,\ge a_m\}$.
By Proposition \ref{Lp-estimate}, we have
\begin{equation}\label{H-e}
    \begin{aligned}
         & \mathbb{P}\left( S^N\notin H_\varepsilon \right)\\
\le& \sum^{+\infty}_{m=1} \mathbb{P}\left( \frac{1}{N} \sum^N_{i=1} \left( \sup_{0\le t\le T} \abs{(\Phi_x^i)} \right)^2\mathds{1}_{\left\{ \sup_{0\le t\le T} \abs{(\Phi_x^i)}\,\ge a_m\right\}}\ge \frac{1}{b_m}\right)\\
\le& \sum^{+\infty}_{m=1} \frac{b_m}{N} \sum^N_{i=1} \mathbb{E}\left[ \left( \sup_{0\le t\le T} \abs{(\Phi_x^i)} \right)^2 \mathds{1}_{\left\{\sup_{0\le t\le T} \abs{(\Phi_x^i)} \ge a_m\right\}} \right]\\
\le& \sum^{+\infty}_{m=1} \frac{b_m}{a_m^{p-2}N}\sum^N_{i=1} \mathbb{E}\left[ \sup_{0\le t\le T} \abs{(\Phi_x^{i})}^p \right]\\
\le& \varepsilon,
    \end{aligned}
\end{equation}
where the third inequality follows from H\"older's inequality and Markov's inequality.  Combining \eqref{N-e} and \eqref{H-e}, we obtain
\[
\mathbb{P}\left( S^N\notin \mathcal{N}_\varepsilon\cap H_\varepsilon \right)\le 2\varepsilon.
\] 
By \cite[ Corollary $5.6$]{Carmona2018}, $\mathcal{N}_\varepsilon\cap H_\varepsilon$ is relatively compact in  $(\mathscr{P}_2(\SSF),\WF)$, which implies that  $\{ S^N \}_{N\in \mathbb{N}}$ is tight in $(\mathscr{P}_2(\SSF),\WF)$.

\end{proof}

\subsection{Proof of propagation of chaos}\label{sec3.4}
In this section, we give the proof of Theorem \ref{Main-Theorem}.  

Since $\{S^N\}_{N\in \mathbb{N}}$ is tight in $(\mathscr{P}_2(\SSF),\WF)$, we can extract a subsequence that converges in law to an accumulation point $S$, which is a $(\mathscr{P}_2(\SSF),\WF)$-valued random variable. First, we prove an important equality for $S$, namely $\int \abs{v_x}^2 S(t)(\d v)=1,\ \mathbb{P}\text{-a.s.}$, which shows that the most difficult term $\mathbb{E}[\Psi_x(t)\Psi_y(t)] $ in the mean-field SDE \eqref{mean-field-limit} is actually bounded  (see Lemma \ref{sphere-SN-S} below). 
Second, we show that for almost every $\omega \in \Omega$, $S(\omega)$ satisfies condition (ii) of Definition \ref{martingale-solution}.  The proof consists of two lemmas. Lemma \ref{Drift-Estimate} shows that $\mathscr{M}_\theta(t, w, S(\omega))$ is a continuous square-integrable martingale. Similarly, Lemma \ref{QV-Estimate} shows that the quadratic variation of $\mathscr{M}_\theta(t, w, S(\omega))$ is $\abs{\theta}^2 t$.

On the other hand, by using the bound $\abs{\mathbb{E}[\Psi_x(t)\Psi_y(t)]}\le 1 $, we establish the pathwise uniqueness of solutions for both the SDE \eqref{fix-mu} (given below) and the mean-field SDE \eqref{mean-field-limit}.  Combining this with the Yamada-Watanabe theorem (see \cite[Lemma 2.1]{Huang2021}), we obtain that the limit  $S$ of the sequence $\{S^N\}_{N\in \mathbb{N}}$ is $\mathbb{P}\text{-a.s.}$ constant in $(\mathscr{P}_2(\SSF),\WF)$. Finally, combining uniform in $N$ bounds from \eqref{sup-1-p}, we deduce that $S$ satisfies the condition (i) of Definition \ref{martingale-solution}, and  $\lim_{N\rightarrow \infty} \mathbb{E}\left[\WFT(S^N,S) \right]=0$.\\

\noindent{\bf Proof of Theorem \ref{Main-Theorem}}
By \cite[Definition $1.43$, Theorem $1.44$]{Pardoux2014} and Lemma \ref{tight-P-2},  we obtain that for the sequence $\{S^N \}_{N\in \mathbb{N}}$ (viewed as a family of $(\mathscr{P}_2(\SSF),\WF)$-valued random variables), there exists a subsequence, still denoted as $\{S^N\}_{N\in \mathbb{N}}$, such that $S^N$ converges in law; that is, there exists a $(\mathscr{P}_2(\SSF),\WF)$-valued random variable $S$ such that $S^N\lawcon S$, as $N\rightarrow \infty$.  In the following, we prove that $S$ is a martingale solution to \eqref{mean-field-limit} in the sense of Definition \ref{martingale-solution}. To this end, by Skorokhod's Theorem, we construct a probability space $(\hat \Omega,{\hat{\mathscr{F}}},\hat{\mathbb{P}})$, and random variables $\hat{S}^N$ and $\hat{S}$ such that $\hat S^N\laweq S^N$, $\hat{S}\laweq S$, and $\hat S^N\rightarrow\hat{S}, \,\mathbb{\hat P}\text{-a.s.}$ in $(\mathscr{P}_2(\SSF),\WF)$.

\begin{lemma}\label{sphere-SN-S}
    For any $x\in \Lambda$ and $0\le t\le T$, 
\begin{equation}\label{SN=1-L}
    \begin{aligned}
         & \int \abs{v_x}^2 S^N(t)(\d v)=1, \quad \mathbb{P}\text{-\rm a.s.},
    \end{aligned}
\end{equation}

\begin{equation}\label{S=1-L}
    \begin{aligned}
         &  \int \abs{v_x}^2 S(t)(\d v)=1,\quad \mathbb{P}\text{-\rm a.s.}.
    \end{aligned}
\end{equation}
Moreover, for any $p\ge 1$, $x,y\in \Lambda$, and $0\le t\le T$,
\begin{equation}\label{vxvy-alpha}
    \begin{aligned}
 & \mathbb{E}\left[ \abs{\int v_x v_yS^N(t)(\d v)-\int v_x v_yS(t)(\d v)}^{p} \right]\rightarrow 0.
    \end{aligned}
\end{equation}
\end{lemma}
\begin{proof}
 The first equation follows from the fact that, for any $x \in \Lambda$, $0\le t\le T$,  we have ${1}/{N}\sum_{i=1}^{N} \abs{\Phi_x^i(t)}^2 = 1$. We now turn to the proof of the second equation \eqref{S=1-L}. Note that for any $0\le t\le T$, $w\in \SSF$, $\abs{w_{t,x}}^2\le \sup_{0\le t\le T}\abs{w_t}^2 $  and $\hat S^N\rightarrow \hat{S}, \mathbb{\hat P}\,\text{-a.s.}$ in $(\mathscr{P}_2(\SSF),\WF)$. By \cite[Theorem $7.12$ (iv)]{Villani2003}, we obtain that for any $0\le t\le T$,
\begin{equation}\label{con-hatP}
    \begin{aligned}
         & \int \abs{v_x}^2\,\hat S^N(t)(\d v)\rightarrow \int \abs{v_x}^2\,\hat S(t)(\d v),\quad \mathbb{\hat P}\text{-a.s.}.
    \end{aligned}
\end{equation}
By \eqref{SN=1-L} and $\hat S^N\laweq S^N$, for any $0\le t\le T$, we have
\begin{equation}\label{hatSN=1}
    \begin{aligned}
         &  \mathbb{P}\left(\int \abs{v_x}^2 S^N(t)(\d v)=1\right)=\mathbb{\hat P}\left(\int \abs{v_x}^2 \hat S^N(t)(\d v)=1\right)=1.
    \end{aligned}
\end{equation}
Combining \eqref{con-hatP} and \eqref{hatSN=1} with $\hat S \laweq S$, for any $0\le t\le T$, we obtain 
\begin{equation*}
    \begin{aligned}
         &  \mathbb{P}\left(\int \abs{v_x}^2 S(t)(\d v)=1\right)=\mathbb{\hat P}\left(\int \abs{v_x}^2 \hat S(t)(\d v)=1\right)=1.
    \end{aligned}
\end{equation*}
By \eqref{SN=1-L}, \eqref{S=1-L}, and  H\"older's inequality,  it is evident that for any $0\le t\le T$,
\begin{equation}\label{vxvy-1}
    \begin{aligned}
         & \abs{\int v_x v_yS^N(t)(\d v)}\le 1,\quad \abs{\int v_x v_yS(t)(\d v)}\le 1, \ \mathbb{P}\text{-a.s.}.
    \end{aligned}
\end{equation}
Since for any $0\le t\le T$, $w\in \SSF$, $\abs{w_{t,x} w_{t,y}}\le \sup_{0\le t\le T} \abs{w_t}^2$  and $\hat S^N\rightarrow \hat{S}, \mathbb{\hat P}\,\text{-a.s.}$ in $(\mathscr{P}_2(\SSF),\WF)$, by \cite[Theorem $7.12$ (iv)]{Villani2003}, we obtain that for any $0\le t\le T$,
\begin{equation*}
    \begin{aligned}
         & \int v_x v_y\hat S^N(t)(\d v)\rightarrow\int v_x v_y\hat S(t)(\d v), \ \mathbb{\hat P}\text{-a.s.}.
    \end{aligned}
\end{equation*}
By \eqref{vxvy-1} and the dominated convergence theorem,
for any $p\ge 1$, $0\le t\le T$,
\begin{equation*}
    \begin{aligned}
&\mathbb{E}\left[ \abs{\int v_x v_yS^N(t)(\d v)-\int v_x v_yS(t)(\d v)}^{p} \right]\\
=& \mathbb{\hat E}\left[ \abs{\int v_x v_y\hat S^N(t)(\d v)-\int v_x v_y\hat S(t)(\d v)}^{p} \right]\rightarrow 0.
    \end{aligned}
\end{equation*}
\end{proof}

In the following, we aim to prove that for almost every $\omega\in \Omega$, 
\begin{equation}\label{martingale-1}
    \begin{aligned}
         & \mathscr{M}_\theta(t,w,{S(\omega)})= \langle w_t,\theta \rangle - \langle w_0,\theta  \rangle-\int^t_0 \langle{\mathbf{D}(w_s, S(s)(\omega))},\theta\rangle\,\d s,\ \theta\in \mathbb{R}^{\abs{\Lambda}}, \,0\le t\le T,
    \end{aligned}
\end{equation}
is a continuous $(\mathscr F_t)$-martingale under $S(\omega)$. To this end,  for any $(w,\mu,\theta)\in \SSF\times \mathscr{P}_2(\SSF)\times \mathbb{R}^{\abs{\Lambda}}$, $0\le t\le T$, we define 
\[
G^N(t,w,\mu):= G^{(1)}(t,w) + G^{N,(2)}(t,w,\mu),
\]
where 
\begin{equation*}
    \begin{aligned}
          G^{(1)}(t,w):=&\langle w_t,\theta \rangle - \langle w_0,\theta  \rangle,\\
          G^{N,(2)}(t,w,\mu):=&-\int^t_0 \langle\mathbf{D}^N( w_s, \mu_s),\theta \rangle\,\d s.
    \end{aligned}
\end{equation*}
Similarly, for any $0\le t\le T$, we define
\[
G(t,w,\mu):= G^{(1)}(t,w) + G^{(2)}(t,w,\mu),
\]
where 
\begin{equation*}
    \begin{aligned}
          G^{(2)}(t,w,\mu):=&-\int^t_0 \langle\mathbf{D}(w_s, \mu_s),\theta \rangle\,\d s.
    \end{aligned}
\end{equation*}
For any $m\in \mathbb{N}$, $g_1,\cdots,g_m\in C_c(\mathbb{R}^{\abs{\Lambda}};\mathbb{R})$, $\mu\in \mathscr{P}_2(\SSF)$, $0\le s< t\le T$, and $0\le s_1<\cdots<s_m\le s$, we define 
\begin{equation}\label{Pi-N}
    \begin{aligned}
         & \Pi^N(\mu):=\int \left( G^N(t,w,\mu)- G^N(s,w,\mu) \right)g_1(w_{s_1})\cdots g_m(w_{s_m})\mu(\d w),
    \end{aligned}
\end{equation}
and
\begin{equation*}
    \begin{aligned}
         & \Pi(\mu):=\int\left( G(t,w,\mu) - G(s,w,\mu)\right)g_1(w_{s_1})\cdots g_m(w_{s_m})\mu(\d w).
    \end{aligned}
\end{equation*}
Once we have shown that $\mathbb{E}[\abs{\Pi(S(\cdot))}]=0$, it follows that, for almost every $\omega$,  $\mathscr M_\theta(t,w,S(\omega))$ is a continuous $(\mathscr F_t)$-martingale under $S(\omega)$.

Our next goal is to show that $\mathbb{E}\left[\abs{\Pi(S)}\right] = 0$. This will be guaranteed by the following lemma.
\begin{lemma}\label{Drift-Estimate}
 It holds that
\begin{equation}\label{lem-e1}
    \begin{aligned}
       &  \mathbb{E}\left[\abs{\Pi^N(S^N)-\Pi(S)}\right]\rightarrow 0\quad \text{as}\ N\rightarrow \infty,
    \end{aligned}
\end{equation}
and 
\begin{equation}\label{lem-e2}
    \begin{aligned}
         & \mathbb{E}\left[\abs{\Pi^N(S^N)}^2\right]\rightarrow 0 \quad \text{as}\ N\rightarrow \infty.
    \end{aligned}
\end{equation}
\end{lemma}

\begin{proof}
Note that 
\begin{equation*}
    \begin{aligned}
         &  \mathbb{E}\left[\abs{\Pi^N(S^N)-\Pi(S)}\right]\le \mathbb{E}\left[ \abs{\Pi^N(S^N)-\Pi(S^N)} \right]+\mathbb{E}\left[ \abs{\Pi(S^N)-\Pi(S)} \right].
    \end{aligned}
\end{equation*}
To prove \eqref{lem-e1}, it suffices to prove 
\begin{equation*}
    \begin{aligned}
         & I_1:=\mathbb{E}\left[ \abs{\Pi^N(S^N)-\Pi(S^N)} \right]\rightarrow 0, \quad \text{as}\ N\rightarrow \infty,
    \end{aligned}
\end{equation*}
and
\begin{equation}\label{I2conv}
    \begin{aligned}
         & I_2:=\mathbb{E}\left[ \abs{\Pi(S^N)-\Pi(S)} \right]\rightarrow 0, \quad \text{as}\ N\rightarrow \infty.
    \end{aligned}
\end{equation}

By the definitions of $\Pi^N$ and $\Pi$, we obtain
\begin{equation*}
    \begin{aligned}
         I_1
\lesssim&\sup_{0\le t\le T}\left( \frac{1}{N}\sum^N_{i=1} \mathbb{E}\left[ \abs{G^{(2)}(t,\Phi^i,S^N)-G^{N,(2)}(t,\Phi^i,S^N)} \right] \right)\\
\lesssim& \sup_{0\le t\le T} \left( \frac{1}{N}\sum^N_{i=1}\int^t_0 \frac{1}{N}\mathbb{E}\left[ \abs{\langle \Phi^i,\theta  \rangle} \right]\,\d s \right)\lesssim\frac{1}{N},
    \end{aligned}
\end{equation*}
where the last inequality follows from the Cauchy-Schwarz inequality and Proposition \ref{Lp-estimate}.

By the definition of $\Pi$, we have
\begin{equation*}
    \begin{aligned}
         & I_2\le J(t) +J(s), \quad 0\le s< t\le T,
    \end{aligned}
\end{equation*}
where for $0\le t\le T$,
\begin{equation*}
    \begin{aligned}
         & J(t):=\mathbb{E}\left[ \abs{\int G(t,w,S^N)g_1(w_{s_1}) \cdots g_m(w_{s_m})S^N(\d w)- \int G(t,w,S)g_1(w_{s_1}) \cdots g_m(w_{s_m})S(\d w)}\right].\\
    \end{aligned}
\end{equation*}
To prove \eqref{lem-e1}, it suffices to prove that for any $0\le t\le T$, $J(t)\rightarrow 0$ as $N\rightarrow \infty$. The proof is divided into two steps.

{\bf Step $1$:} We prove that for any $0\le t\le T$,
\begin{equation}\label{step1}
    \begin{aligned}
         & \mathbb{E}\left[\left|\int G^{(1)}(t,w)g_1(w_{s_1})\cdots g_m(w_{s_m}) S^N(\d w)-\int G^{(1)}(t,w)g_1(w_{s_1})\cdots g_m(w_{s_m}) S(\d w)\right|\right] \rightarrow 0.
    \end{aligned}
\end{equation}  
\noindent Since  $\hat S^N\rightarrow\hat{S}, \mathbb{\hat P}\text{-a.s.}$ in $(\mathscr{P}_2(\SSF),\WF)$, and for any $0\le t\le T$, $w\mapsto G^{(1)}(t,w) \cdot g_1(w_{s_1})\cdots g_m(w_{s_m})$ is a continuous function on $(\SSF,\dF)$, 
$$\abs{G^{(1)}(t,w) \cdot g_1(w_{s_1})\cdots g_m(w_{s_m})}\lesssim \sup_{0\le t\le T} \abs{w_t}^2+1,$$
by \cite[Theorem $7.12$ (iv)]{Villani2003}, we derive that for any $0\le t\le T$,
\begin{equation}\label{as-G-1}
    \begin{aligned}
      \int G^{(1)}(t,w)g_1(w_{s_1})\cdots g_m(w_{s_m}) \hat S^N(\d w)\rightarrow \int G^{(1)}(t,w)g_1(w_{s_1})\cdots g_m(w_{s_m}) \hat S(\d w), \quad \mathbb{\hat P}\text{-a.s.}.
    \end{aligned}
\end{equation}   
By Proposition \ref{Lp-estimate}, we obtain that for any $1<p\le \p$, with $\p$ given in Assumption \ref{initial-value}, $0\le t\le T$,
\begin{equation*}\label{as-G1-ui}
    \begin{aligned}
         &   \mathbb{\hat E}\left[\abs{\int G^{(1)}(t,w)g_1(w_{s_1})\cdots g_m(w_{s_m}) \hat S^N(\d w)}^{p}\right]\\
=&\mathbb{ E}\left[\abs{\int G^{(1)}(t,w)g_1(w_{s_1})\cdots g_m(w_{s_m})  S^N(\d w)}^{p}\right]<\infty.
    \end{aligned}
\end{equation*}
By uniform integrability, together with \eqref{as-G-1} and the facts that $\hat S^N\laweq S^N$ and $\hat S\laweq S$, we obtain \eqref{step1}.

{\bf Step $2$:} We prove that for any $0\le t\le T$,
\begin{equation*}
        \begin{aligned}
             & \mathbb{E}\left[\left|\int G^{(2)}(t,w,S^N)g_1(w_{s_1})\cdots g_m(w_{s_m}) S^N(\d w)-\int G^{(2)}(t,w,S)g_1(w_{s_1})\cdots g_m(w_{s_m}) S(\d w)\right|\right] \rightarrow 0.
        \end{aligned}
\end{equation*}
It suffices to prove that for $0\le t\le T$, $J_1(t)\rightarrow 0$ and $J_2(t)\rightarrow 0$ as $N\rightarrow \infty$, where 
\begin{equation*}\label{D-1-1}
    \begin{aligned}
         & J_1(t):=\mathbb{E}\left[ \abs{\int\int^t_0 \langle \mathbf{D}(w_s,S^N(s))-\mathbf{D}(w_s,S(s)),\theta\rangle \,\d s\, g_1(w_{s_1})\cdots g_m(w_{s_m}) S^N(\d w)}^2\right],
    \end{aligned}
\end{equation*}
and 
\begin{equation*}\label{D-2-1}
    \begin{aligned}
       J_2(t):=\mathbb{E}\bigg[\bigg| &\int\int^t_0 \langle \mathbf{D}(w_s,S(s)),\theta\rangle \,\d s\, g_1(w_{s_1})\cdots g_m(w_{s_m}) S^N(\d w)\\
&- \int\int^t_0 \langle \mathbf{D}(w_s,S(s)),\theta\rangle \,\d s\, g_1(w_{s_1})\cdots g_m(w_{s_m}) S(\d w)\bigg|\bigg].
    \end{aligned}
\end{equation*}
Note that for $0\le t\le T$,
\begin{equation}\label{D-bounded}
    \begin{aligned}
   J_1(t)\lesssim&\frac{1}{N}\sum^N_{i=1}\mathbb{E}\left[ \int^t_0 \sum_{x\in \Lambda} \abs{\mathbf{D}_x(\Phi^i(s),S^N(s))- \mathbf{D}_x(\Phi^i(s),S(s))}^2\,\d s\right]\\
\lesssim& \frac{1}{N}\sum^N_{i=1} \int^t_0 \left[\sum_{x\in \Lambda}  \mathbb{E}\left[\abs{\Phi_x^i}^{2p}\right]\right]^\frac{1}{p} \left[\sum_{x\in \Lambda}\mathbb{E}\left[\sum_{\N}\abs{\int v_xv_y  S^N(s)(\d v)-\int v_xv_y  S(s)(\d v)}^{2q}\right]\right]^\frac{1}{q}\,\d s, 
    \end{aligned}
\end{equation}
where $1/p+1/q=1$ with  $p=\frac{\p}{2}>1$, $\p$ given in Assumption \ref{initial-value}.
By Proposition \ref{Lp-estimate}, \eqref{vxvy-alpha}, and the dominated convergence theorem, we obtain that for any $0\le t\le T$,  $J_1(t)\rightarrow 0$ as $N\rightarrow \infty$.

Since $\hat S\laweq S$ and by \eqref{vxvy-1}, there exists a $\mathbb{\hat P}$-null set $\mathcal{N}_1$ such that  for any $0\le t\le T$, $\hat\omega\in \hat \Omega\backslash \mathcal N_1$,
\begin{equation*}
    \begin{aligned}
         & \abs{\int v_x v_y \hat S(\hat \omega)(t)(\d v)}\le 1.
    \end{aligned}
\end{equation*} 
Therefore, for any  $\hat\omega\in \hat \Omega\backslash \mathcal N_1$, $0\le t\le T$, the function
$w \mapsto\int^t_0 \langle \mathbf{D}(w_s,\hat S(\hat \omega)(s)), \theta\rangle\,\d s g_1(w_{s_1})\cdots g_m(w_{s_m})$ is continuous on $(\SSF,\dF)$, and
\begin{equation*}
    \begin{aligned}
         & \abs{\int^t_0 \langle \mathbf{D}(w_s,\hat S(\hat \omega)(s)), \theta\rangle\,\d s g_1(w_{s_1})\cdots g_m(w_{s_m})}\le C\cdot \left[\sup_{0\le t\le T}\abs{w_t}^2+1\right]. 
    \end{aligned}
\end{equation*}
Since $\hat S^N\rightarrow \hat S,\, \mathbb{\hat P}\text{-a.s.}$ in $(\mathscr{P}_2(\SSF),\WF)$,  there exists a $\mathbb{\hat P}$-null set $\mathcal{N}_2$ such that for any $\hat \omega\in \hat \Omega\backslash \mathcal N_2$, $\hat S^N(\hat \omega)\rightarrow\hat{S}(\hat \omega)$ in $(\mathscr{P}_2(\SSF),\WF)$. By \cite[Theorem $7.12$ (iv)]{Villani2003},  we obtain that for any $\hat \omega\in \hat\Omega\backslash (\mathcal N_1\cup \mathcal N_2)$, $0\le t\le T$,
\begin{equation*}
    \begin{aligned}
         & \int\int^t_0 \langle \mathbf{D}(w_s,\hat S(\hat \omega)(s)),\theta\rangle \,\d s\, g_1(w_{s_1})\cdots g_m(w_{s_m}) \hat S^N(\hat \omega)(\d w)\\
&\rightarrow  \int\int^t_0 \langle \mathbf{D}(w_s,\hat S(\hat\omega)(s)),\theta\rangle \,\d s\, g_1(w_{s_1})\cdots g_m(w_{s_m}) \hat S(\hat\omega)(\d w).  
    \end{aligned}
\end{equation*}
By Proposition \ref{Lp-estimate} and \eqref{vxvy-1}, we obtain that for any $1<p\le \p$, with $\p$ given in Assumption \ref{initial-value}, $0\le t\le T$,
\begin{equation*}
    \begin{aligned}
         & \mathbb{E}\left[  \abs{\int\int^t_0 \langle \mathbf{D}(w_s,S(s)),\theta\rangle \,\d s\, g_1(w_{s_1})\cdots g_m(w_{s_m})  S^N(\d w)}^{p}\right]<\infty.
    \end{aligned}
\end{equation*}
Therefore, for any $0\le t\le T$, $J_2(t)\rightarrow 0$ as $N\rightarrow \infty$. This completes {\bf Step $2$} and establishes \eqref{lem-e1}.

Finally, we prove  \eqref{lem-e2}.  We define $\mathscr{M}^{N,i}_t
:=\hat{\mathscr{M}}^{N,i}_t+ \tilde{\mathscr{M}}^{N,i}_t$, where for $0\le t\le T$,
\begin{equation}\label{Mhat}
    \begin{aligned}
         & \hat{\mathscr{M}}^{N,i}_t:=\sum_{x\in \Lambda}\int^t_0 \theta_x\,\d W_x^i(s),
    \end{aligned}
\end{equation} 
and
\begin{equation}\label{Mtilde}
    \begin{aligned}
         & \tilde{\mathscr{M}}^{N,i}_t:=-\sum_{x\in \Lambda}\int^t_0 \theta_x \Phi_x^i(s)\cdot\frac{\sum^N_{k=1} \Phi_x^k(s)\,\d W_x^k(s)}{N}.
    \end{aligned}
\end{equation}
Since $\Phi$ satisfies \eqref{the-new-model}, by the definition of $\Pi^N(\mu)$ in \eqref{Pi-N}, we have
\begin{equation*}
    \begin{aligned}
         & \Pi^N(S^N)=\frac{1}{N}\sum^N_{i=1} \left( \mathscr{M}^{N,i}_t-\mathscr{M}^{N,i}_s \right)g_1(\Phi^i(s_1))\cdots g_m(\Phi^i(s_m)).
    \end{aligned}
\end{equation*}
Note that $g_1,\cdots,g_m$ are bounded. Thus,  we have
\begin{equation*}
    \begin{aligned}
         \mathbb{E}\left[ \abs{\Pi^N(S^N)}^2\right]\lesssim & \sup_{0\le t\le T} \mathbb{E}\left[ \left(\frac{1}{N}\sum^N_{i=1} \mathscr{M}^{N,i}_t \right)^2\right]\\
\lesssim&\sup_{0\le t\le T} \mathbb{E}\left[ \left(\frac{1}{N}\sum^N_{i=1} \hat{\mathscr{M}}^{N,i}_t \right)^2\right]+\sup_{0\le t\le T} \mathbb{E}\left[ \left(\frac{1}{N}\sum^N_{i=1} \tilde{\mathscr{M}}^{N,i}_t \right)^2\right]\\
\lesssim&\sup_{0\le t\le T}\frac{1}{N^2}\sum^N_{i,j=1}\mathbb{E}\left[ \hat{\mathscr{M}}^{N,i}_t \hat{\mathscr{M}}^{N,j}_t \right]+\sup_{0\le t\le T}\frac{1}{N}\sum^N_{i=1}\mathbb{E}\left[\left( \tilde{\mathscr{M}}^{N,i}_t \right)^2\right]\\
\lesssim& \frac{1}{N},
    \end{aligned}
\end{equation*}
where the last inequality follows from
\begin{equation}\label{hat-M}
    \begin{aligned}
         \mathbb{E}\left[ \hat{\mathscr{M}}^{N,i}_t \hat{\mathscr{M}}^{N,j}_t\right]=\delta_{ij}\abs{\theta}^2 t, \quad 0\le t\le T,
    \end{aligned}
\end{equation}
and
\begin{equation}\label{tilde-M}
    \begin{aligned}
         \mathbb{E}\left[ \left(\tilde{\mathscr{M}}^{N,i}_t\right)^2 \right]=\mathbb{E}\left[\int^t_0 \frac{1}{N}\sum_{x\in \Lambda} \theta_x^2\,  \abs{\Phi_x^i(s)}^2  \,\d s\right]\mathrel{\overset{\text{\eqref{sup-1-p}}}{\lesssim}}\frac{1}{N}, \quad  0\le t\le T.
    \end{aligned}
\end{equation}
Therefore, we have \eqref{lem-e2}. This completes the proof.
\end{proof}

In the following, we aim to prove that for almost every $\omega\in \Omega$, the quadratic variation process of $\mathscr{M}_\theta(t,w,{S(\omega)})$ defined in \eqref{martingale-1} is given by $\abs{\theta}^2 t$. To this end, we introduce the following notation.  For any $(w,\mu,\theta)\in \SSF\times \mathscr{P}_2(\SSF)\times \mathbb{R}^{\abs{\Lambda}}$, $0\le t\le T$, we define
\[
\mathbf G^N(t,w,\mu):= \mathbf G^{N,(1)}(t,w,\mu) + \mathbf G^{N,(2)}(t,w,\mu),
\]
where 
\begin{equation*}
    \begin{aligned}
          \mathbf G^{N,(1)}(t,w,\mu):=&(G^N(t,w,\mu))^2,\\
         \mathbf G^{N,(2)}(t,w,\mu):=&-\abs{\theta}^2 t +\int^t_0 \frac{\sum_{x\in \Lambda} \abs{\theta_x}^2 \abs{w_{s,x}}^2}{N}\,\d s.
    \end{aligned}
\end{equation*}
Similarly, we define
\begin{equation*}
    \begin{aligned}
         & \mathbf{G}(t,w,\mu):= \mathbf G^{(1)}(t,w,\mu) + \mathbf G^{(2)}(t),
    \end{aligned}
\end{equation*}
where 
\[
\mathbf G^{(1)}(t,w,\mu):=(G(t,w,\mu))^2,\quad\mathbf G^{(2)}(t):=-\abs{\theta}^2 t.
\]
For any $m\in \mathbb{N}$,   $g_1,\cdots,g_m\in C_c(\mathbb{R}^{\abs{\Lambda}};\mathbb{R})$, $\mu\in \mathscr{P}_2(\SSF)$, $0\le s< t\le T$, and $0\le s_1<\cdots<s_m\le s$, we define 
\begin{equation*}\label{BPi-N}
    \begin{aligned}
         & \mathbf \Pi^N(\mu):=\int \left( \mathbf G^N(t,w,\mu)- \mathbf G^N(s,w,\mu) \right)g_1(w_{s_1})\cdots g_m(w_{s_m})\mu(\d w),
    \end{aligned}
\end{equation*}
and
\begin{equation*}
    \begin{aligned}
         & \mathbf \Pi(\mu):=\int\left( \mathbf G(t,w,\mu) - \mathbf G(s,w,\mu)\right)g_1(w_{s_1})\cdots g_m(w_{s_m})\mu(\d w).
    \end{aligned}
\end{equation*}

The proof of Lemma \ref{QV-Estimate} is similar to that of Lemma \ref{Drift-Estimate}. To prove that, for almost every $\omega\in \Omega$, the quadratic variation process of $\mathscr{M}_\theta(t,w,S(\omega))$ defined in \eqref{martingale-1} is $\abs{\theta}^2 t$, it suffices to show that  $\mathbb{E}[\abs{\mathbf \Pi(S)}]=0$.

\begin{lemma}\label{QV-Estimate}
    It holds that
\begin{equation}\label{lem-1-Pi}
    \begin{aligned}
         & \mathbb{E}\left[\abs{\mathbf \Pi^N(S^N)-\mathbf \Pi(S)}\right]\rightarrow 0\quad \text{as}\ N\rightarrow \infty,
    \end{aligned}
\end{equation} 
and 
\begin{equation}\label{lem-2-Pi}
    \begin{aligned}
         & \mathbb{E}\left[\abs{\mathbf \Pi^N(S^N)}^2\right]\rightarrow 0 \quad \text{as}\ N\rightarrow \infty.
    \end{aligned}
\end{equation}
\end{lemma}
\begin{proof}
Note that 
\begin{equation*}
    \begin{aligned}
         & \mathbb{E}\left[\abs{\mathbf{\Pi}^N(S^N)-\mathbf{\Pi}(S)}\right]\le J_1 +J_2,
    \end{aligned}
\end{equation*}
where $J_1:=\mathbb{E}\left[ \abs{\mathbf{\Pi}^N(S^N)-\mathbf{\Pi}(S^N)} \right]$ and $J_2:=\mathbb{E}\left[ \abs{\mathbf{\Pi}(S^N)-\mathbf{\Pi}(S)} \right]$. To prove \eqref{lem-1-Pi}, it suffices to show that both $J_1$ and $J_2$ converge to zero as $N \to \infty$.

By the definitions of $\mathbf \Pi^N$ and $\mathbf \Pi$, we have 

\begin{equation}\label{I1+J1}
    \begin{aligned} 
J_1\lesssim& \sum^2_{j=1} \sup_{0\le t\le T}\left( \frac{1}{N}\sum^N_{i=1} \mathbb{E}\left[ \abs{\mathbf G^{N,(j)}(t,\Phi^i,S^N)-\mathbf G^{(j)}(t,\Phi^i,S^N)} \right] \right):=\sum^2_{j=1}  \sup_{0\le t\le T}I^j_t.
    \end{aligned}
\end{equation}

By the definitions of $\mathbf G^{N,(1)}$ and $\mathbf G^{(1)}$, we obtain that for any $0\le t\le T$,
\begin{equation}\label{I1N}
    \begin{aligned}
          I^1_t=&\frac{1}{N}\sum^N_{i=1}\mathbb{E}\left[ \abs{(G^N(t,\Phi^i,S^N) )^2-(G(t,\Phi^i,S^N))^2} \right]\\
\lesssim&\frac{1}{N}\sum^N_{i=1} \mathbb{E}\left[ (G^N-G)^2(t,\Phi^i,S^N) \right]^\frac{1}{2}\cdot \mathbb{E}\left[\left(\abs{G^N}^2+\abs{G}^2 \right)(t,\Phi^i,S^N) \right]^\frac{1}{2}\\
\lesssim& \frac{1}{N}\sum^N_{i=1} \left( \int^T_0\frac{1}{N^2}\abs{\theta}^2 \mathbb{E}\left[\abs{\Phi^i(s)}^2\right]\,\d s  \right)^\frac{1}{2}\\
&\quad\cdot \left(\abs{\theta}^2\mathbb{E}\left[\sup_{0\le t\le T}\abs{\Phi^i(t)}^2\right] +\int_0^T\abs{\theta}^2 \mathbb{E}\left[\left(\abs{\mathbf D^N}^2+\abs{\mathbf D}^2\right)(\Phi^i(s),S^N(s)) \right] \,\d s\right)^{\frac{1}{2}}\\
\lesssim& \frac{1}{N}.
    \end{aligned}
\end{equation}
where the last inequality follows from the definitions of $\mathbf D^N$, $\mathbf D$, \eqref{sphere-holder} and Proposition \ref{Lp-estimate}.   By the definitions of $\mathbf G^{N,(2)}$ and $\mathbf G^{(2)}$,  and the fact that $\Phi\in \mathcal S$, it follows that for any $0\le t\le T$,
\begin{equation}\label{J1N}
    \begin{aligned}
         & I^2_t= \frac{1}{N}\sum^N_{i=1}\mathbb{E}\left[\int^t_0 \frac{\sum_{x\in \Lambda} \abs{\theta_x}^2\abs{\Phi_x^i(s)}^2}{N} \,\d s\right]\lesssim \frac{1}{N}.\\
    \end{aligned}
\end{equation}
Combining \eqref{I1N}, \eqref{J1N}, and \eqref{I1+J1}, we obtain $J_1\rightarrow 0$ as $N\rightarrow \infty$.

Next, we prove that $J_2\rightarrow 0$ as $N\rightarrow \infty$. Similar to \eqref{I2conv} in the proof of Lemma \ref{Drift-Estimate}, we only need to prove that for any $0\le t\le T$,
\begin{equation}\label{BG1}
    \begin{aligned}
          \mathbb{E}\bigg[&\bigg|\int \mathbf{G}^{(1)}(t,w,S^N)g_1(w_{s_1})\cdots g_m(w_{s_m}) S^N(\d w)\\ 
&-\int \mathbf{G}^{(1)}(t,w,S)g_1(w_{s_1})\cdots g_m(w_{s_m}) S(\d w)\bigg|\bigg] \rightarrow 0,
    \end{aligned}
\end{equation}  
and 
\begin{equation}\label{G2-con}
    \begin{aligned}
         \mathbb{E}\bigg[\bigg| &\int \mathbf{G}^{(2)}(t)g_1(w_{s_1})\cdots g_m(w_{s_m}) S^N(\d w)\\
         &-\int \mathbf{G}^{(2)}(t)g_1(w_{s_1})\cdots g_m(w_{s_{m}}) S(\d w)\bigg|\bigg] \rightarrow 0.
    \end{aligned}
\end{equation}
To obtain \eqref{BG1}, it suffices to prove that for any $0\le t\le T$,
\begin{equation*}\label{D-1-2}
    \begin{aligned}
         & J_{21}(t):=\mathbb{E}\left[ \abs{\int\int^t_0\Big[ [G(t,w,S^N)]^2 -[G(t,w,S)]^2 \Big]\,\d s\, g_1(w_{s_1})\cdots g_m(w_{s_m}) S^N(\d w) }\right]\rightarrow 0,
    \end{aligned}
\end{equation*}
and
\begin{equation*}\label{D-2-2}
    \begin{aligned}
       J_{22}(t):=\mathbb{E}\bigg[ \bigg|    &\int [G(t,w,S)]^2\, g_1(w_{s_1})\cdots g_m(w_{s_m}) S^N(dw)\\
&-  \int[G(t,w,S)]^2\, g_1(w_{s_1})\cdots g_m(w_{s_m}) S(\d w)\bigg|\bigg]\rightarrow 0.
    \end{aligned}
\end{equation*}
 By Proposition \ref{Lp-estimate}, \eqref{vxvy-1} and \eqref{D-bounded}, we obtain that for any $0\le t\le T$,
\begin{equation*}\label{J_21}
    \begin{aligned}
         J_{21}(t)\lesssim&\frac{1}{N}\sum^N_{i=1} \mathbb{E}\left[ \abs{G(t,\Phi^i,S^N)+G(t,\Phi^i,S)}^2 \right]^\frac{1}{2}\cdot \mathbb{E}\left[ \abs{G(t,\Phi^i,S^N)-G(t,\Phi^i,S)}^2 \right]^\frac{1}{2}\\
\lesssim& \frac{1}{N}\sum^N_{i=1} \left[ \mathbb{E}\int^t_0 \sum_{x\in \Lambda} \abs{\mathbf{D}_x(\Phi^i(s),S^N(s))- \mathbf{D}_x(\Phi^i(s),S(s))}^2\,\d s\right]^\frac{1}{2}\rightarrow 0.
    \end{aligned}
\end{equation*}
Similar to the proof that $J_2\rightarrow 0$ as $N\rightarrow 0$ in Lemma \ref{Drift-Estimate}, by verifying that for any $0\le t\le T$, $w\mapsto [G(t,w,\hat S)]^2$ is a continuous function on  $(\SSF,\dF)$ and $[G(t,w,\hat S)]^2\le C\cdot (\sup_{0\le t\le T}\abs{w_t}^2 +1)$, we obtain that for any $0\le t\le T$,

\begin{equation*}
    \begin{aligned}
       & \int [G(t,w,\hat S)]^2\, g_1(w_{s_1})\cdots g_m(w_{s_m}) \hat S^N(\d w)\rightarrow  \int[G(t,w,\hat S)]^2\, g_1(w_{s_1})\cdots g_m(w_{s_m}) \hat S(\d w), \quad \mathbb{\hat P}\text{-a.s.},
    \end{aligned}
\end{equation*}
and for $p=\frac{\p}{2}>1$, with $\p$ given in Assumption \ref{initial-value},
\begin{equation*}
    \begin{aligned}
         & \mathbb{E}\left[  \abs{\int [G(t,w,S)]^2\, g_1(w_{s_1})\cdots g_m(w_{s_m}) S^N(\d w)}^{p} \right]<\infty.
    \end{aligned}
\end{equation*}
Moreover, we obtain that for any $0\le t\le T$, $J_{22}(t)\rightarrow 0$ as $N\rightarrow \infty$ since $S^N\laweq \hat S^N$ and  $S\laweq \hat S$. Combining $J_{22}(t)\rightarrow 0$ and $J_{21}(t)\rightarrow 0$ as $N\rightarrow \infty$, we obtain \eqref{BG1}. Since for any $0\le t\le T$, $w\mapsto \mathbf{G}^{(2)}(t) g_1(w_{s_1})\cdots g_m(w_{s_m})$ is a bounded continuous function on  $(\SSF,\dF)$,  $\hat S^N\rightarrow \hat S,\ \mathbb{\hat P}\text{-a.s.}$ in $(\mathscr{P}_2(\SSF),\WF)$, and $S^N\laweq \hat S^N$,  $S\laweq \hat S$, we obtain  \eqref{G2-con}.  This completes the proof of \eqref{lem-1-Pi}.

Finally, we prove \eqref{lem-2-Pi}. By the definitions of $\mathbf \Pi^N$ and $S^N$, we have
\begin{equation*}
    \begin{aligned}
          \mathbf \Pi^N(S^N)=& \frac{1}{N}\sum^N_{i=1} \Bigg[ 2\int^t_s {\mathscr{M}}^{N,i}_r \,\d{\mathscr{M}}^{N,i}_r 
\Bigg] g_1(\Phi^i(s_1))\cdots g_m(\Phi^i(s_m)).\\
    \end{aligned}
\end{equation*}
Recall the definitions of $\hat{\mathscr{M}}^{N,i}$ and $\tilde{\mathscr{M}}^{N,i}$ given in \eqref{Mhat} and \eqref{Mtilde}.

Note that $g_1,\cdots,g_m$ are bounded. Thus, we have
\begin{equation*}
    \begin{aligned}
        \mathbb{E}\left[ \abs{\mathbf \Pi^N(S^N) }^2\right]\lesssim& \sup_{0\le t\le T} \mathbb{E}\left[ \left(\frac{1}{N}\sum^N_{i=1} \int^t_0 {\mathscr{M}}^{N,i}_r \,\d{\mathscr{M}}^{N,i}_r \right)^2\right]\\
\le & \frac{\abs{\theta}^2}{N^2}\sum_{i=1}^N \int_0^T \mathbb{E}\left[\abs{\mathscr{M}^{N,i}_r}^2  \right]\,\d r\\
&+ \frac{3\abs{\theta}^2}{N^3}\int^T_0 \sum_{x\in \Lambda}\mathbb{E}\left[\abs{\sum_{i=1}^N \Phi_x^i(r) \mathscr{M}_r^{N,i}} \abs{\sum_{j=1}^N \Phi_x^j(r) \mathscr{M}_r^{N,j}}\right]\, \d r\\
&:=H_1 + H_2,
    \end{aligned}
\end{equation*}
where the second inequality follows from It\^o's isometry and
\begin{equation*}
    \begin{aligned}
        \d\,\langle\hat{\mathscr{M}}^{N,i}, \hat{\mathscr{M}}^{N,j}\rangle_t=\delta_{ij}\abs{\theta}^2 \d t,\quad 0\le t\le T,
    \end{aligned}
\end{equation*}
and
\begin{equation*}
    \begin{aligned}
     \d\,\langle\mathscr{\hat M}^{N,i}, \mathscr{\tilde M}^{N,j}\rangle_t=  \d\,\langle\mathscr{\tilde M}^{N,i}, \mathscr{\tilde M}^{N,j}\rangle_t=\frac{1}{N}\sum_{x\in \Lambda} \abs{\theta_x}^2 \Phi^i_x(t) \Phi_x^j(t) \,\d t, \quad 0\le t\le T.
    \end{aligned}
\end{equation*}
By \eqref{hat-M} and \eqref{tilde-M}, we have 
\begin{equation*}
    \begin{aligned}
         & \mathbb{E}\left[ \abs{\mathscr{M}^{N,i}_r}^2  \right]\lesssim \mathbb{E}\left[ \abs{\mathscr{\hat M}^{N,i}_r}^2  \right] +\mathbb{E}\left[ \abs{\mathscr{\tilde M}^{N,i}_r}^2  \right]\lesssim 1,\quad 0\le r\le T.
    \end{aligned}
\end{equation*}
Therefore, $H_1\lesssim 1/N$.  By H\"older's inequality, we obtain 
\begin{equation*}
    \begin{aligned}
         & \abs{\sum^N_{i=1} \Phi^i_x(r) \mathscr{M}^{N,i}_r}^2\le \sum^N_{i=1} \abs{\Phi^i_x(r)}^2 \cdot\sum^N_{i=1} \abs{\mathscr{M}^{N,i}_r}^2=N\cdot\sum^N_{i=1} \abs{\mathscr{M}^{N,i}_r}^2,\quad 0\le r\le T.
    \end{aligned}
\end{equation*}
Therefore, for $0\le r\le T$,
\begin{equation*}
    \begin{aligned}
         & H_2\lesssim\frac{1}{N^3}\cdot N\cdot \sum^N_{i=1} \mathbb{E}\left[\abs{\mathscr{M}^{N,i}_r}^2\right]\lesssim \frac{1}{N}. 
    \end{aligned}
\end{equation*}
This completes the proof of \eqref{lem-2-Pi}.
\end{proof}

Since $S^N_0$ converges in probability to $\mu_0$ , we have $S_0=\mu_0,\ \mathbb{P}\text{-a.s.}$. By Lemma \ref{Drift-Estimate} and Lemma \ref{QV-Estimate}, we have
\begin{equation*}
    \begin{aligned}
         & \quad \Pi(S)=0,\ \mathbb{P}\text{-a.s.}, \quad\mathbf \Pi(S)=0,\  \mathbb{P}\text{-a.s.}.
    \end{aligned}
\end{equation*}
Since  $C_c(\mathbb{R}^{\abs{\Lambda}};\mathbb{R})$ is separable, there exists a countable dense subset  $\mathscr{C}_0\subset C_c(\mathbb{R}^{\abs{\Lambda}};\mathbb{R})$.  Define 
\begin{equation*}
    \begin{aligned}
         & \widetilde{\Omega}:=\{\widetilde{\omega}\in {\Omega}:S(\widetilde{\omega})\in \mathscr{P}_2(\SSF),\  \  \Pi(S)(\widetilde{\omega})=0,\, 
 \text{and}\ \mathbf \Pi(S)(\widetilde{\omega})=0, \ \forall\, g_1,\cdots,g_m\in \mathscr{C}_0  \},
    \end{aligned}
\end{equation*}
Then, $\mathbb{P}(\widetilde{\Omega})=1$, and for all $\widetilde{\omega}\in \widetilde{\Omega}$, $0\le s< t\le T$, $m\in \mathbb{N}$, $0\le s_1<\cdots < s_m\le s$, $g_1,\cdots,g_m \in C_c(\mathbb{R}^{\abs{\Lambda}};\mathbb{R})$,
\begin{equation*}\label{Martingale-1}
    \begin{aligned}
         & \int\left( G(t,w,S(\widetilde{\omega})) - G(s,w,S(\widetilde{\omega}))\right)g_1(w_{s_1})\cdots g_m(w_{s_m})S(\widetilde{\omega})(\d w)=0,
    \end{aligned}
\end{equation*}
and 
\begin{equation*}\label{Cov-1}
    \begin{aligned}
         & \int\left( \mathbf G(t,w,S(\widetilde{\omega})) - \mathbf G(s,w,S(\widetilde{\omega}))\right)g_1(w_{s_1})\cdots g_m(w_{s_m})S(\widetilde{\omega})(\d w)=0,
    \end{aligned}
\end{equation*}
which imply that $\{\mathscr{M}_\theta(t,w,S(\widetilde{\omega}))\}_{0\le t\le T}$ is an $S(\widetilde{\omega})$-martingale and $\langle \mathscr{M}_\theta\rangle(t,w,S(\widetilde{\omega}))=\abs{\theta}^2 t$. Therefore, $S(\widetilde{\omega})$, $\widetilde{\omega}\in \widetilde{\Omega}$ satisfies Definition \ref{martingale-solution} $(ii)$.  By \cite[Theorem $7.12$ (iii)]{Villani2003} and Proposition \ref{Lp-estimate}, we have
\begin{equation*}\label{W2-con}
    \begin{aligned}
         & \int \big(\sup_{0\le t\le T} \abs{w_t}\big)^2\hat S^N(\d w)\rightarrow \int \big(\sup_{0\le t\le T} \abs{w_t}\big)^2\hat S(\d w),\ \mathbb{\hat P}\text{-a.s},
    \end{aligned}
\end{equation*}
and $\{\int (\sup_{0\le t\le T} \abs{w_t})^2S^N(\d w)\}_{N\in \mathbb{N}}$ is uniformly integrable, which implies
\begin{equation}\label{W2-S-bounded}
    \begin{aligned}
         & \mathbb{E}\left[\int \big(\sup_{0\le t\le T} \abs{w_t}\big)^2  S(\d w)\right]<\infty.
    \end{aligned}
\end{equation}
By \eqref{vxvy-1} and \eqref{W2-S-bounded}, we obtain that $S(\widetilde{\omega})$, $\widetilde{\omega}\in \widetilde{\Omega}$ satisfies Definition \ref{martingale-solution} $(i)$.

In the following, we introduce two sets of probability measures.  First, we define
\begin{equation}\label{Pone-R}
    \begin{aligned}
         &\mathscr{P}_{\O}(\mathbb{R}^{\abs{\Lambda}}):=\left\{ \mu\in \mathscr{P}(\mathbb{R}^{\abs{\Lambda}}) \Big| \int \abs{v_x}^2 \d\mu(v)\le 1,\ \text{for all}\ x\in \Lambda\right\}.
    \end{aligned}
\end{equation}
Furthermore, we set
\begin{equation}\label{Pone-Phi-F}
    \begin{aligned}
         &\mathscr{P}_{\O,\Psi}(\SSF):=\left\{ \mathbf{M}\in \mathscr{P}(\SSF) \Big|\ \mathbf{M}\circ \pi_t^{-1}\in \mathscr{P}_{\O}(\mathbb{R}^{\abs{\Lambda}}) ,\ \text{for all}\, 0\le t\le T\right\}.
    \end{aligned}
\end{equation}
By Theorem \ref{Main-Theorem}, it is clear that the martingale solution to the mean-field SDE \eqref{mean-field-limit} with initial distribution $\mu_0\in \mathscr{P}_{\mathcal S,\p}$ belongs to $\mathscr{P}_{1,\Psi}(\SSF)$.

Let $S$ denote a martingale solution to the mean-field SDE \eqref{mean-field-limit}.  To establish the uniqueness of the martingale solution as well as the existence and uniqueness of a probabilistically strong solution to \eqref{mean-field-limit} via the Yamada-Watanabe theorem (\cite[Lemma $2.1$]{Huang2021}), we introduce 
the following linear SDE :
\begin{equation}\label{fix-mu}
    \left\{\begin{aligned}
         d \widetilde \Psi_x(t)= &2\kappa\sum_{\N}\Big(  \widetilde \Psi_y(t)-\bm{h}_{x,y}(t)\widetilde {\Psi}_x(t)\Big)\,\d t-\frac{1}{2}\widetilde {\Psi}_x(t)\,\d t+\d\widetilde{W}_x(t),\quad  0\le t\le T,\\
\widetilde \Psi_x(0)=& \widetilde \psi_x, \quad x\in \Lambda, 
\end{aligned}\right.
\end{equation}
where $\widetilde{\psi} \laweq S(0)$, $\bm{h}_{x,y}(t):=\int v_x v_y \, S(t)(\d v)$, with $S(t)$ denoting the law at time $t$ of the solution $\Psi$ to \eqref{mean-field-limit}.

We show that the pathwise uniqueness holds for both the SDE \eqref{fix-mu} and the mean-field SDE \eqref{mean-field-limit}. 

\begin{lemma}\label{fixed-mu-le}
For any $S\in \mathscr{P}_{\O,\Psi}(\SSF)$, the pathwise uniqueness holds for the SDE \eqref{fix-mu}.
\end{lemma}
\begin{proof}
Since $S\in \mathscr{P}_{\O, \Psi}(\SSF)$, we have $\int \abs{v_x}^2S(t)(\d v)\le 1$, for all $x\in \Lambda$, $0\le t\le T$. By H\"older's inequality,  we obtain that for any $x,y\in \Lambda$, $0\le t\le T$, $\abs{\bm{h}_{x,y}(t)}\le 1$. Therefore, the SDE \eqref{fix-mu} satisfies the Lipschitz condition, and the pathwise uniqueness holds.  
\end{proof}

\begin{lemma}\label{not-fix-mu}
The pathwise uniqueness holds in $\mathscr{P}_{\O,\Psi}(\SSF)$ for the mean-field SDE \eqref{mean-field-limit}.
\end{lemma}

\begin{proof}
Let $\Psi^{(1)}$ and $\Psi^{(2)}$ be any two weak solutions to the mean-field SDE~\eqref{mean-field-limit}, defined on the same probability space and driven by the same Brownian motion, such that the laws of $\Psi^{(1)}$ and $\Psi^{(2)}$ belong to $\mathscr{P}_{\O,\Psi}(\SSF)$. 
By applying Newton-Leibniz formula to $|\Psi_x^{(1)}-\Psi_x^{(2)}|^2(t)$, we obtain that for any $0\le t\le T$, 
\begin{equation}\label{U-0}
    \begin{aligned}
         &\abs{\Psi_x^{(1)}-\Psi_x^{(2)}}^2(t)\\
\le & 4\kappa\sum_{\N} \int^t_0 \left[\Psi_y^{(1)}-\Psi_y^{(2)}\right](s)\left[\Psi_x^{(1)}-\Psi_x^{(2)}\right](s)\,\d s \\
&- 4\kappa\sum_{\N} \int^t_0 \left[ \mathbb{E}\left[ \Psi_x^{(1)}(s)\Psi_y^{(1)}(s) \right]\Psi_x^{(1)}(s) - \mathbb{E}\left[ \Psi_x^{(2)}(s)\Psi_y^{(2)}(s) \right]\Psi_x^{(2)}(s) \right]\left[ \Psi_x^{(1)}-\Psi_x^{(2)} \right](s)\,\d s\\
:=& 4\kappa U_1(t)-4\kappa\sum_{\N}U_2(t).
    \end{aligned}
\end{equation}
By Young's inequality, we have
\begin{equation}\label{U-1}
    \begin{aligned}
\abs{\mathbb{E}[U_1(t)]}\le& \frac{1}{2}\sum_{\N} \int^t_0 \mathbb{E}\left[ \abs{\Psi_y^{(1)}-\Psi_y^{(2)}}^2(s) \right]\,\d s+ d\int^t_0 \mathbb{E}\left[ \abs{\Psi_x^{(1)}-\Psi_x^{(2)}}^2(s) \right]\,\d s,\quad 0\le t\le T.
    \end{aligned}
\end{equation}
Since the laws of $\Psi^{(1)}$ and  $\Psi^{(2)}$  both belong to $\mathscr{P}_{1,\Psi}(\SSF)$,  we obtain that for any $x,y\in \Lambda$, $i=1,2$, and $0\le t\le T$,
\begin{equation}\label{Cov-Psi-r}
    \begin{aligned}
         & \mathbb{E}\left[\abs{\Psi_x^{(i)}(t)}^2\right]\le 1,\quad \abs{\mathbb{E}\left[ \Psi_x^{(i)}(t)\Psi_y^{(i)}(t) \right]}\le 1.
    \end{aligned}
\end{equation}
Using \eqref{Cov-Psi-r}, together with H\"older's and Young's inequalities,  we obtain that for any $0\le t\le T$,
\begin{equation}\label{U-2}
    \begin{aligned}
          \abs{\mathbb{E}\left[U_2(t)\right]}\le& \int^t_0 \abs{\mathbb{E}\left[ \Psi_x^{(1)}(s)\Psi_y^{(1)}(s) \right]} \mathbb{E}\left[\abs{\Psi_x^{(1)}-\Psi_x^{(2)}}^2(s)\right]\,\d s\\
&+\abs{\int^t_0 \mathbb{E}\left[ \left[\Psi_x^{(1)}(s)- \Psi_x^{(2)}(s)\right]\Psi_y^{(1)}(s)\right]\mathbb{E}\left[ \left[\Psi_x^{(1)}(s)- \Psi_x^{(2)}(s)\right]\Psi_x^{(2)}(s) \right]\,\d s}\\
&+ \abs{\int^t_0 \mathbb{E}\left[ \left[\Psi_y^{(1)}(s)- \Psi_y^{(2)}(s)\right]\Psi_x^{(2)}(s)\right]\mathbb{E}\left[ \left[\Psi_x^{(1)}(s)- \Psi_x^{(2)}(s)\right]\Psi_x^{(2)}(s) \right]\,\d s}\\
\lesssim& \int^t_0  \mathbb{E}\left[\abs{\Psi_x^{(1)}-\Psi_x^{(2)} }^2(s)\right]\,\d s   +  \int^t_0  \mathbb{E}\left[\abs{ \Psi_y^{(1)}-\Psi_y^{(2)} }^2(s)\right]\,\d s.
    \end{aligned}
\end{equation}
Combining \eqref{U-0}, \eqref{U-1} and \eqref{U-2}, it follows that for any $0\le t\le T$,
\begin{equation*}
    \begin{aligned}
           &\sum_{x\in \Lambda}\mathbb{E}\left[
\abs{\Psi_x^{(1)}-\Psi_x^{(2)}}^2(t)\right]\lesssim \int^t_0  \sum_{x\in \Lambda}\mathbb{E}\left[\abs{ \Psi_x^{(1)}-\Psi_x^{(2)} }^2(s)\right]\,\d s. 
    \end{aligned}
\end{equation*}
By Gronwall's inequality, we see that for any $0\le t\le T$,
\begin{equation}\label{P-U}
    \begin{aligned}
         & \sum_{x\in \Lambda}\mathbb{E}\left[\abs{\Psi_x^{(1)}-\Psi_x^{(2)}}^2(t)\right]=0.
    \end{aligned}
\end{equation}
By \eqref{P-U} and $\Psi^{(i)}(\cdot)\in C([0,T];\mathbb{R}^{\abs{\Lambda}})$ a.s., $i=1,2$, we obtain that $\Psi^{(1)}(t)=\Psi^{(2)}(t)$ for all $0\le t\le T$\  $\text{a.s.}$.
\end{proof}

\begin{remark}
It is straightforward to verify that the bound $1$ in the definition of the measure set $\mathscr{P}_{\O}(\mathbb{R}^{\abs{\Lambda}})$ can be replaced by any constant $C \geq 1$, and both Lemma~\ref{fixed-mu-le} and Lemma~\ref{not-fix-mu} remain valid. In the present work, we choose $C=1$, which suffices for our purposes, since the solution $S$ constructed in Theorem~\ref{Main-Theorem} satisfies $\int |v_x|^2\, S(t)(\mathrm{d}v) = 1$ for all $x\in \Lambda$, $0\le t\le T$.
\end{remark}

By the existence of a martingale solution to \eqref{mean-field-limit} and the martingale representation theorem, we obtain the existence of a weak solution to the mean-field SDE \eqref{mean-field-limit}. Combining Lemma \ref{fixed-mu-le} and Lemma \ref{not-fix-mu}, and applying the Yamada--Watanabe theorem (\cite[Lemma 2.1]{Huang2021}), we obtain the uniqueness of the weak solution to \eqref{mean-field-limit}. Moreover, it follows that the martingale solution to the mean-field SDE \eqref{mean-field-limit} is unique. More precisely, for any two martingale solutions $S_1, S_2 \in \mathscr{P}_{\O, \Psi}(\SSF)$ to the mean-field SDE~\eqref{mean-field-limit}, if $S_1(0) = S_2(0)$, then $S_1 = S_2$.
 Therefore, $S(\widetilde{\omega})$ is independent of $\widetilde{\omega}$, and we still denote it by $S$.
 We obtain that the entire sequence $\{S^N\}_{N\in \mathbb{N}}$, viewed as a family of $(\mathscr{P}_2(\SSF),\WF)$-valued random variables, converges in law to a $\mathbb{P}\text{-a.s.}$ constant random variable $S$ in $(\mathscr{P}_2(\SSF),\WF)$. By \cite[Lemma $5.7$]{Kallenberg2021}, we obtain that 
\begin{equation*}
    \begin{aligned}
         &  \WF(S^N,S) \rightarrow 0\quad \text{in probability}.
    \end{aligned}
\end{equation*}

To prove \eqref{POC}, we only need to verify that $\{ \WFT(S^N,S)\}_{N\in \mathbb{N}}$ is uniformly integrable. To this end, we will show that for $p=\frac{\p}{2}>1$, with $\p$ given in Assumption \ref{initial-value},
\begin{equation*}
    \begin{aligned}
         & \mathbb{E}\left[ \abs{\WFT(S^N,S)}^{p} \right]\lesssim 1.
    \end{aligned}
\end{equation*}
By Proposition \ref{Lp-estimate}, we  have
\begin{equation*}
    \begin{aligned}
         & \mathbb{E}\left[\left(\int \big(\sup_{0\le t\le T} \abs{w_t}\big)^2 S^N(\d w)\right)^{p}\right]<\infty.
    \end{aligned}
\end{equation*}
Note that 
\begin{equation*}
    \begin{aligned}
         & \mathbb{E}\left[ \abs{\WFT(S^N,S)}^{p} \right]\lesssim \mathbb{E}\left[ \left(\int \big(\sup_{0\le t\le T} \abs{w_t}\big)^2 S^N(\d w)\right)^{p}+ \left(\int \big(\sup_{0\le t\le T} \abs{w_t}\big)^2 S(\d w)\right)^{p}\right]<\infty.
    \end{aligned}
\end{equation*}
This completes the proof. \qed

\subsection{Convergence  rate}\label{sec3.5} In this section, we present the proof of Theorem \ref{strong-solution}.

First, we establish the existence and uniqueness of a probabilistically strong solution to \eqref{mean-field-limit}. Second, by using \eqref{vxvy-alpha} and directly comparing the two dynamics \eqref{the-new-model} and \eqref{Psi-i-limit}, we obtain the convergence stated in \eqref{Convergence-2}.  Finally, in order to obtain a convergence rate, we use a classical coupling argument and the following fact: for $p=1$ or $2$, if $U_1, \ldots, U_N$ are independent, identically distributed, mean-zero random variables with finite $2p$-th moment, then
$$
\mathbb{E}\left[\left| \frac{1}{N} \sum_{i=1}^N U_i \right|^{2p}\right] \lesssim \frac{1}{N^{p}},
$$
where the implicit constant is independent of $N$.

{\bf Step $1$:}
The existence of a martingale solution to \eqref{mean-field-limit}, together with the martingale representation theorem, implies the existence of a weak solution to \eqref{mean-field-limit}. Combining Lemma \ref{fixed-mu-le}, Lemma \ref{not-fix-mu}, and applying the Yamada-Watanabe theorem (\cite[Lemma 2.1]{Huang2021}), we obtain that there exists a unique probabilistically strong solution to 
the mean-field SDE \eqref{mean-field-limit}. 

{\bf Step $2$:} We prove  \eqref{Convergence-2}. By applying It\^o's  formula to $\abs{\Phi_x^i(t)-\Psi_x^i(t)}^2$, we obtain that for any $0\le t\le T$,
\begin{equation}\label{Ito-Phi-Psi}
    \begin{aligned}
         &\d\abs{\Phi_x^i(t)-\Psi_x^i(t)}^2\\
=& 4\kappa\sum_{\N}\left[ \Phi^i_y(t)-\Psi_y^i(t)\right][\Phi_x^i(t)-\Psi_x^i(t)]\,\d t\\
&-4\kappa \sum_{\N} \left[ \frac{\sum^N_{k=1} \Phi_x^k(t)\Phi_y^k(t)}{N}\Phi_x^i(t)-\mathbb{E}[\Psi_x^i(t)\Psi_y^i(t)] \Psi_x^i(t) \right][\Phi_x^i(t)-\Psi_x^i(t)]\,\d t \\
&+\left[\frac{\abs{\Phi_x^i(t)}^2}{N}-2\left[ \frac{N-1}{2N}\Phi_x^i(t)-\frac{1}{2}\Psi_x^i(t)\right] [\Phi_x^i(t)-\Psi_x^i(t)]\right]\,\d t\\
&-2 \frac{\sum^N_{k=1}\Phi_x^k(t)\,\d W^k_x(t) }{N} \Phi_x^i(t)[\Phi_x^i(t)-\Psi_x^i(t)]\\
:=&A_1(t)\,\d t -A_2(t)\,\d t -A_3(t)\,\d t-\d A_4(t).
    \end{aligned}
\end{equation}
Next, we estimate each term. By Young's inequality, we have
\begin{equation}\label{A-1}
    \begin{aligned}
          &\mathbb{E}\left[ \sup_{0\le t\le T} \int^t_0 \abs{A_1(s)}\,\d s \right] \\
\le& 2\kappa \sum_{\N} \mathbb{E}\int^T_0 \abs{\Phi^i_y(s)-\Psi_y^i(s)}^2\,\d s + 4\kappa d \mathbb{E}\int^T_0 \abs{\Phi_x^i(s)-\Psi_x^i(s)}^2\,\d s.
    \end{aligned}
\end{equation}
Similar to the proof of Proposition \ref{Lp-estimate}, by using \eqref{Cov-Psi}, we obtain that for any $p\ge 2$,
\begin{equation}\label{Psi-sup-1-p}
    \begin{aligned}
         & \sum_{x\in \Lambda}  \mathbb{E}\left[\sup_{0\le t\le T}\abs{\Psi_x(t)}^p\right]\le   C(d,p,\kappa,T,\abs{\Lambda})  \left( \sum_{x\in \Lambda} \mathbb{E}\left[\abs{\psi_x}^{p}\right] + 1\right).
    \end{aligned}
\end{equation}
By \eqref{vxvy-1}, H\"older's inequality,  we have
\begin{equation}\label{A-2}
    \begin{aligned}
         &\mathbb{E}\left[ \sup_{0\le t\le T} \int^t_0 \abs{A_2(s)}\,\d s \right] \\
\lesssim& \mathbb{E}\int^T_0 \abs{\Phi_x^i(t)-\Psi_x^i(t)}^2\,\d t\\
&+ \sum_{\N}\int^T_0 \mathbb{E}\left[\abs{\Phi^i_x(t)}^{2p}\right]^{\frac{1}{p}}\mathbb{E}\left[\abs{\frac{\sum^N_{k=1} \Phi_x^i(t)\Phi_y^i(t) }{N}-\mathbb{E}[\Psi_x^i(t)\Psi_y^i(t)] }^{2q}\right]^\frac{1}{q}\,\d t,
    \end{aligned}
\end{equation}
where $1/p+1/q=1$ with $p=\frac{\p}{2}>1$, $\p$ given in Assumption \ref{initial-value}.  By Proposition \ref{Lp-estimate} and \eqref{Psi-sup-1-p}, we have
\begin{equation}\label{A-3}
    \begin{aligned}
           &\mathbb{E}\left[ \sup_{0\le t\le T} \int^t_0 \abs{A_3(s)}\,\d s \right]\lesssim \mathbb{E}\int^T_0 \abs{\Phi_x^i(t)-\Psi_x^i(t)}^2\,\d t +\frac{1}{N}.\\
    \end{aligned}
\end{equation}
By the Burkholder-Davis-Gundy inequality and Young's inequality, we have
\begin{equation}\label{A-4}
    \begin{aligned}
         & \mathbb{E}\left[\abs{\sup_{0\le t\le T} A_4(t)}\right]\lesssim  \frac{1}{N}\,\mathbb{E}\left[\sup_{0\le t\le T}\abs{\Phi_x^i(t)}^2\right]+\mathbb{E}\int^T_0 \abs{\Phi_x^i(t)-\Psi_x^i(t)}^2\,\d t. 
    \end{aligned}
\end{equation}
Together, \eqref{Ito-Phi-Psi}, \eqref{A-1}, \eqref{A-2}, \eqref{A-3} and \eqref{A-4} imply for any $0\le t\le T$,
\begin{equation*}
    \begin{aligned}
         & \sum_{x\in \Lambda}\mathbb{E}\left[ \sup_{0\le s\le t} \abs{\Phi_x^i(s)-\Psi_x^i(s)}^2 \right]\\
\lesssim& \sum_{x\in \Lambda}\mathbb{E}\left[ \abs{\varphi_x^i-\psi_x^i}^2 \right] +  \frac{1}{N}+ \sum_{x\in \Lambda}\sum_{\N}\int^T_0 \mathbb{E}\left[\abs{\frac{\sum^N_{k=1} \Phi_x^i(t)\Phi_y^i(t) }{N}-\mathbb{E}[\Psi_x^i(t)\Psi_y^i(t)] }^{2q}\right]^\frac{1}{q}\,\d t \\
&+\int^t_0 \sum_{x\in \Lambda}\mathbb{E}\left[ \sup_{0\le r\le s} \abs{\Phi_x^i(r)-\Psi_x^i(r)}^2 \right]\,\d s,
    \end{aligned}
\end{equation*}
where the implicit constant depends on $\mathbb{E}[|\varphi^1|^{\p}]$.  Note that 
$(1/N)\sum^N_{k=1} \Phi_x^i(t)\Phi_y^i(t)=\int v_x v_y S^N(t)(\d v)$ and $\mathbb{E}[\Psi^i_x(t)\Psi_y^i(t)]=\int v_x v_y S(t)(\d v)$. By Gronwall's inequality,  \eqref{vxvy-alpha} and the dominated convergence theorem, we derive \eqref{Convergence-2}. This completes the proof of the first part of Theorem \ref{strong-solution}.

For the second part of Theorem \ref{strong-solution}, under the additional hypothesis that
$\{ (\varphi^i,\psi^i) \}^{N}_{i=1}$ is symmetric, we can conclude that for any $0\le t\le T$, the laws of $(\Phi^i-\Psi^i)(t)$ and $(\Phi^j-\Psi^j)(t)$, where $i\neq j$, are identical.  Therefore, in view of 
\begin{equation}\label{emp-com}
    \begin{aligned}
         & \mathbb{E}\left[ \abs{(\Phi^i-\Psi^i)(t)}^2 \right]=\frac{1}{N}\sum_{i=1}^N \mathbb{E}\left[ \abs{(\Phi^i-\Psi^i)(t)}^2 \right], \quad 0\le t\le T,
    \end{aligned}
\end{equation}
we only need to estimate the right-hand side of \eqref{emp-com}.  We also use the notation in \eqref{Ito-Phi-Psi} and estimate each term separately.

{\bf Step $3$:} By direct calculation, we derive that for any $0\le t\le T$,     
\begin{equation}\label{A1}
    \begin{aligned}
         & \mathbb{E}\left[\sum_{x\in \Lambda} \frac{1}{N}\sum^N_{i=1}\int^t_0 \abs{A_1(s)}\,\d s\right]\lesssim \int_0^t \mathbb{E}\left[\sum_{x\in \Lambda} \frac{1}{N}\sum^N_{i=1}\abs{\Phi_x^i-\Psi_x^i}^2(s)\right]\,\d s,
    \end{aligned}
\end{equation}
and
\begin{equation}\label{A3}
    \begin{aligned}
         & \mathbb{E}\left[\sum_{x\in \Lambda} \frac{1}{N}\sum^N_{i=1}\int^t_0 \abs{A_3(s)}\,\d s\right]\lesssim \frac{1}{N}+ \int^t_0 \mathbb{E}\left[\sum_{x\in \Lambda} \frac{1}{N}\sum^N_{i=1}\abs{\Phi_x^i-\Psi_x^i}^2(s)\right]\,\d s.
    \end{aligned}
\end{equation}
For the martingale part $A_4(t)$, we obtain that for any $0\le t\le T$,
\begin{equation}\label{A4}
    \begin{aligned}
        \sum_{x\in \Lambda} \mathbb{E}\left[\abs{\frac{1}{N}\sum^N_{i=1} A_4(t)}\right]\lesssim& \sum_{x\in \Lambda}\mathbb{E}\left[\abs{\int^t_0 \frac{1}{N}\sum^N_{k=1} \Phi_x^k(s)\,\d W_x^k(s)\cdot \frac{1}{N}\sum^N_{i=1}  \Phi_x^i(s)[\Phi_x^i(s)-\Psi_x^i(s)]}\right]\\
\lesssim&\sum_{x\in \Lambda}  \mathbb{E}\left[\frac{1}{\sqrt{N}}\cdot\abs{ \int^t_0  \Big[ \frac{1}{N}\sum^N_{i=1}  \Phi_x^i(s)[\Phi_x^i(s)-\Psi_x^i(s)] \Big]^2\,\d s}^\frac{1}{2} \right]\\
\lesssim& \frac{1}{N} + \int^t_0 \mathbb{E}\left[\sum_{x\in \Lambda} \frac{1}{N}\sum^N_{i=1}\abs{\Phi_x^i-\Psi_x^i}^2(s)\right]\,\d s,
    \end{aligned}
\end{equation}
where the second inequality follows from the Burkholder-Davis-Gundy inequality, and the third inequality follows from Young's inequality and $\Phi\in \mathcal S$. Unlike {\bf Step $2$}, we need to further decompose $A_2(t)$, for simplicity, we omit the time variable,
\begin{equation}\label{A2-decompose}
    \begin{aligned}
         & \left[\frac{1}{N}\sum^N_{k=1} \Phi_x^k\Phi_y^k\cdot \Phi_x^i- \mathbb{E}\left[ \Psi_x^i \Psi_y^i\right]\Psi_x^i\right]\cdot\left[ \Phi_x^i-\Psi_x^i \right]\\
=&\frac{1}{N}\sum^N_{k=1}\left[ \Psi_x^k\Psi_y^k -\mathbb{E}[\Psi_x^k\Psi_y^k ]\right]\Psi_x^i\cdot\left[ \Phi_x^i-\Psi_x^i \right]+\left[\frac{1}{N}\sum^N_{k=1} \Phi_x^k\Phi^k_y\right]\cdot\left[ \Phi_x^i-\Psi_x^i \right]^2\\
&+\frac{1}{N}\sum^N_{k=1} \Phi_x^k[\Phi_y^k-\Psi_y^k]\cdot\Psi_x^i\cdot\left[ \Phi_x^i-\Psi_x^i \right]+ \frac{1}{N}\sum^N_{k=1} \Psi_y^k[\Phi_x^k-\Psi_x^k]\cdot\Psi_x^i\cdot\left[ \Phi_x^i-\Psi_x^i \right]\\
:=& A_{21}+ A_{22}+A_{23}+A_{24}.
    \end{aligned}
\end{equation}
In the following, we estimate each term on the right-hand side of \eqref{A2-decompose} separately. Since for $x,y\in \Lambda$, $0\le t\le T$,  $\abs{{1}/{N}\sum^N_{k=1}\Phi_x^k(t)\Phi_y^k(t)}\le 1$, we obtain that for any $0\le t\le T$, 
\begin{equation}\label{A22}
    \begin{aligned}
         & \mathbb{E}\left[\sum_{x\in \Lambda} \sum_{\N}\frac{1}{N}\sum^N_{i=1}\int^t_0 \abs{A_{22}(s)}\,\d s\right]\lesssim \int^t_0 \mathbb{E}\left[\sum_{x\in \Lambda} \frac{1}{N}\sum^N_{i=1}\abs{\Phi_x^i-\Psi_x^i}^2(s)\right]\,\d s.
    \end{aligned}
\end{equation}
Note that for any $x\in \Lambda$, $0\le t\le T$, and $i\in \mathbb{N}$,   $(1/N)\cdot\sum^N_{k=1}{\abs{\Phi_x^k(t)}^2}=1$ and $\mathbb{E}[|\Psi_x^i(t)|^2]=1$,  by the Cauchy-Schwarz inequality, \eqref{Psi-sup-1-p} and Young's inequality, we have 
\begin{equation}\label{A23-1}
    \begin{aligned}
         \frac{1}{N}\sum^{N}_{i=1}\abs{A_{23}}\le& \left[\frac{1}{N}\sum^N_{k=1} \abs{\Phi^k_y-\Psi^k_y}^2\right]^\frac{1}{2} \left[\frac{1}{N}\sum^N_{i=1}\abs{\Psi^i_x}^2 \right]^\frac{1}{2} \left[\frac{1}{N}\sum^N_{i=1} \abs{\Phi^i_x-\Psi^i_x}^2\right]^\frac{1}{2}\\
\lesssim&  \left[\frac{1}{N}\sum^N_{k=1} \abs{\Phi^k_y-\Psi^k_y}^2\right]+\left[\frac{1}{N}\sum^N_{i=1}\left(\abs{\Psi^i_x}^2-\mathbb{E}\left[\abs{\Psi^i_x}^2\right]\right) \right]^2 \left[\frac{1}{N}\sum^N_{i=1} \abs{\Phi^i_x-\Psi^i_x}^2\right]\\
&+ \left[\frac{1}{N}\sum^N_{i=1} \abs{\Phi^i_x-\Psi^i_x}^2\right].
    \end{aligned}
\end{equation}
Since, for fixed time, $\{ \Psi^i \}_{i\in \mathbb{N}}$ are independent and identically distributed random vectors, define $U^i_x:= \abs{\Psi_x^i}^2-\mathbb{E}[|\Psi_x^i|^2]$. Then, we have
\begin{equation}\label{A23-2}
    \begin{aligned}
         & \mathbb{E}\left[ \abs{\frac{1}{N}\sum^N_{i=1} U^i_x }^4\right]=\frac{1}{N^3} \mathbb{E}\left[\abs{U_x^i}^4  \right]+\frac{3N(N-1)}{N^4} \mathbb{E}\left[\abs{U^i_x}^2  \right]\cdot \mathbb{E}\left[\abs{U^i_x}^2  \right]\lesssim \frac{1}{N^2},
    \end{aligned}
\end{equation}
where the last inequality follows from \eqref{Psi-sup-1-p} with $p=8$, together with Assumption~\ref{initial-value}(3) for $\p=8$, which ensures that $\mathbb{E}[|\psi|^8]<\infty$.
Similarly, by \eqref{sup-1-p} and \eqref{Psi-sup-1-p} with $p=4$, we have 
\begin{equation}\label{A23-3}
    \begin{aligned}
         & \mathbb{E}\left[\abs{\frac{1}{N}\sum^N_{i=1} \abs{\Phi^i_x-\Psi^i_x}^2}^2 \right]\lesssim 1. 
    \end{aligned}
\end{equation}
Together, \eqref{A23-1}, \eqref{A23-2} and \eqref{A23-3} imply for any $0\le t\le T$,
\begin{equation}\label{A23}
    \begin{aligned}
         & \mathbb{E}\left[\sum_{x\in \Lambda} \sum_{\N}  \frac{1}{N}\sum^{N}_{i=1}\int^t_0 \abs{A_{23}}(s)\,\d s \right]\lesssim \int^t_0 \mathbb{E}\left[\sum_{x\in \Lambda} \frac{1}{N}\sum^N_{i=1} \abs{\Phi^i_x-\Psi^i_x}^2(s)\right]\,\d s +\frac{1}{N}.
    \end{aligned}
\end{equation}

Since $\Psi_y^k[\Phi_x^k-\Psi^k_x]$ and $\Psi_x^i[\Phi_x^i-\Psi^i_x]$ in $A_{24}$ have analogous structures, we can apply the same argument as that for $\Psi_x^i[\Phi_x^i-\Psi^i_x]$ in $A_{23}$, and obtain that for any $0\le t\le T$,
\begin{equation}\label{A24}
    \begin{aligned}
         & \mathbb{E}\left[\sum_{x\in \Lambda} \sum_{\N}  \frac{1}{N}\sum^{N}_{i=1}\int^t_0 \abs{A_{24}}(s)\,\d s \right]\lesssim \int^t_0 \mathbb{E}\left[\sum_{x\in \Lambda} \frac{1}{N}\sum^N_{i=1} \abs{\Phi^i_x-\Psi^i_x}^2(s)\right]\,\d s +\frac{1}{N}.
    \end{aligned}
\end{equation} 
Define $U^k_{xy} := \Psi_x^k \Psi_y^k - \mathbb{E}[\Psi_x^k \Psi_y^k]$. Applying the same argument as that for $\Psi_x^i[\Phi_x^i-\Psi^i_x]$ in $A_{23}$, and noting that $\mathbb{E}[| ({1}/{N}) \sum_{k=1}^N U^k_{xy} |^2] \lesssim {1}/{N}$, it follows that for any $0\le t\le T$,
\begin{equation}\label{A21}
    \begin{aligned}
         & \mathbb{E}\left[\sum_{x\in \Lambda} \sum_{\N}  \frac{1}{N}\sum^{N}_{i=1}\int^t_0 \abs{A_{21}}(s)\,\d s \right]
\lesssim \int^t_0 \mathbb{E}\left[\sum_{x\in \Lambda} \frac{1}{N}\sum^N_{i=1} \abs{\Phi^i_x-\Psi^i_x}^2(s)\right]\,\d s +\frac{1}{N}.
    \end{aligned}
\end{equation}
Combining \eqref{A1}-\eqref{A4}, \eqref{A22}, \eqref{A23}-\eqref{A21}, and applying Gronwall's inequality, we obtain \eqref{rate}. This completes the proof.    
\qed

\section{Large $N$ limit of dynamics on $\mathbb{Z}^d$}\label{sec4}
In this section, we will prove Theorem \ref{simple-1} and Theorem \ref{simple-2} for the lattice $\mathbb{Z}^d$.  The main ideas and the structure are the same as in Section \ref{sec3}; we only need to extend the proofs of Theorem~\ref{Main-Theorem} and Theorem~\ref{strong-solution} to the weighted $\ell^p$ space.

First, we introduce new notation. In this section, we denote by $v=(v_x)$ a vector indexed by $x\in \mathbb{Z}^d$.  We define 
\begin{equation*}
    \begin{aligned}
         & \ell^p:=\Big\{ v:\mathbb{Z}^d\rightarrow \mathbb{R}, \sum_{x\in \mathbb{Z}^d}\abs{v_x}^p<\infty \Big\},\quad p\ge 2. 
    \end{aligned}
\end{equation*}
For $p\ge 2$ and $a>1$, we define a configuration space $\mathcal Q:=\ell^p_a(\mathbb{Z}^d;\mathbb{S}^{N-1}(\sqrt N))\subset(\ell^p_a(\mathbb{Z}^d;\mathbb{R}))^{N}$, where
\begin{equation}\label{lpa}
    \begin{aligned}
         & \ell^p_a:=\ell^p_a(\mathbb{Z}^d;\mathbb{R}):=\Big\{ v:\mathbb{Z}^d\rightarrow \mathbb{R}, \sum_{x\in \mathbb{Z}^d}\frac{1}{a^{\abs{x}}}\cdot\abs{v_x}^p<\infty \Big\} 
    \end{aligned}
\end{equation}
is a weighted $\ell^p$ space with exponential weights. The norm on $\ell_a^p$ is defined by
\begin{equation}\label{wlp}
    \begin{aligned}
         & \norm{v}_p:=\left(\sum_{x\in \mathbb{Z}^d} \frac{1}{a^{\abs{x}}}\abs{v_x}^p\right)^{\frac{1}{p}},
    \end{aligned}
\end{equation}
for all $v\in \ell_a^p$.
In particular, when $p=2$, we set $\mathbb{H} := \ell^2_a$. Note that $\mathcal Q$ depends on $N$; however, we omit this dependence for notational simplicity. 
Furthermore, we define $C([0,T];\mathcal{Q})$ as the space of all continuous functions from $[0,T]$ to $\mathcal{Q}$.

The standard inner product on $\ell^2$ is defined by
\begin{equation*}
    \begin{aligned}
         & \langle v^1, v^2 \rangle = \sum_{x\in \mathbb{Z}^d} v^1_x \cdot v^2_x,
    \end{aligned}
\end{equation*}
for all $v^1, v^2 \in \ell^2$.
The norm induced by the inner product is given by
\begin{equation*}
    \begin{aligned}
         & \norm{v} := \sqrt{\langle v, v \rangle} = \Big( \sum_{x\in \mathbb{Z}^d} |v_x|^2 \Big)^{1/2},
    \end{aligned}
\end{equation*}
for all $v \in \ell^2$. We say that a vector $v = (v_x)$, indexed by $\mathbb{Z}^d$, is a compactly supported vector if there exists a $\L\in \mathbb{N}$ such that $v_x = 0$ for all $x \in \mathbb{Z}^d\setminus\Lambda_\L$. We denote the set of all such vectors by $\ell_c$.

\subsection{Langevin dynamics} In the infinite volume case, the Langevin dynamics for the spin $O(N)$ model is given by the following SDEs: 
\begin{equation}\label{the-new-model-Z}
    \left\{\begin{aligned}
         \d\Phi^i_x(t)= &2\kappa\sum_{\N}\left( \Phi_y^i(t)-\frac{\sum^N_{k=1} \Phi_x^k(t) \Phi_y^k(t)}{N} \Phi_x^i(t)\right)\,\d t-\frac{N-1}{2N} \Phi_x^i(t)\,\d t\\
&+\,\d W_x^i(t)-\frac{\sum^N_{k=1} \Phi_x^k(t)\,\d W_x^k(t) }{N}\Phi_x^i(t), \quad 0\le t\le T,\\
\Phi_x^i(0)=&\varphi_x^i,\quad x\in \mathbb{Z}^d,\quad i=1,\cdots,N,
\end{aligned}\right.
\end{equation}
where, for any $x\in \mathbb{Z}^d$, $\bm{W}_x:=(W_x^1,\cdots,W^N_x)$ is an $N$-dimensional standard Brownian motion on the probability space $(\Omega,\mathscr{F},\mathbb{P})$, and for any $x\neq y\in \mathbb{Z}^d$, $\bm{W}_x$ and $\bm{W}_y$ are independent.  Since, for any $x \in \mathbb{Z}^d$, $\Phi_x \in \mathbb{S}^{N-1}(\sqrt{N})$, the existence and uniqueness of the solution to \eqref{the-new-model-Z} can be established by combining the proofs of \cite[Proposition 3.4]{Shen2023} and \cite[Proposition 3.6]{Shen2024}, we omit the details. 

\begin{proposition}\label{u-e-t-z}
    For fixed $N,\L\in \mathbb{N}$, $\kappa>0$, and $T>0$, and given any initial data $\Phi(0)\in \mathcal{Q}$, there exists a unique probabilistically strong solution $\Phi=(\Phi_x)\in C([0,T];\mathcal{Q})$ to \eqref{the-new-model-Z}.
\end{proposition}

As discussed in the introduction, the system \eqref{the-new-model-Z} formally converges to the following mean-field SDE:
\begin{equation}\label{mean-field-limit-Z}
    \left\{\begin{aligned}
         \d\Psi_x(t)= &2\kappa\sum_{\N}\Big(  \Psi_y(t)-\mathbb{E}[ \Psi_x(t) \Psi_y(t)] \Psi_x(t)\Big)\,\d t-\frac{1}{2}\Psi_x(t)\,\d t+\d W_x(t),\quad 0\le t\le T,\\
\Psi_x(0)=& \psi_x, \quad x\in \mathbb{Z}^d,
\end{aligned}\right.
\end{equation}
where $W_x$ is a $1$-dimensional standard Brownian motion on the probability space $(\Omega,\mathscr{F},\mathbb{P})$, for $x\neq y$, the processes $W_x$ and $W_y$ are independent, and $(\psi_x)$ satisfies specific assumptions. Similar to the case of a finite lattice $\Lambda$, we construct a martingale solution to \eqref{mean-field-limit-Z} using the solution to the Langevin dynamics \eqref{the-new-model-Z}.

We first write the Langevin dynamics \eqref{the-new-model-Z} in terms of the following empirical measure:
\begin{equation*}
    \begin{aligned}
         & S^N(t):= \frac{1}{N}\sum^N_{j=1}\delta_{\Phi^j(t)},\ \  0\le t\le T,
    \end{aligned}
\end{equation*}
where $\Phi^j(t)=(\Phi_x^j(t))$ is a vector indexed by $x\in \mathbb{Z}^d$. For any $x\in \mathbb{Z}^d$, we also use \eqref{vector-Lambda} to define $\mathbf{D}_x^N(v,\nu)$, and  \eqref{vector-M} to define $\d\mathcal{M}_x^{N,i}$.  Then, \eqref{the-new-model-Z} can be rewritten as 
\begin{equation}\label{the-new-model-rw-2-Z}
    \left\{\begin{aligned}
         & \d\Phi_x^i(t)=\mathbf{D}^N_x(\Phi^i(t),S^N(t))\,\d t + \d\mathcal{M}_x^{N,i}(t),\quad 0\le t\le T,\\
& \Phi_x^i(0)=\varphi_x^i,\quad x\in \mathbb{Z}^d, \ i=1,\cdots, N.
\end{aligned}\right.
\end{equation}
Similarly, for any $x\in \mathbb{Z}^d$, we also use \eqref{vector-Lambda-limit} to define $\mathbf{D}_x(v,\nu)$. Then, \eqref{mean-field-limit-Z} can be expressed as
\begin{equation}\label{mean-field-limit-rw-2-Z}
    \left\{\begin{aligned}
         & \d\Psi_x(t)=\mathbf{D}_x(\Psi(t),S(t))\,\d t + \d W_x(t),\quad 0\le t\le T,\\
& \Psi_x(0)=\psi_x,\quad x\in \mathbb{Z}^d,
\end{aligned}\right.
\end{equation}
where  $S(t)$ denotes the law of $\Psi(t)$. 

We then formulate the martingale problem to the mean-field SDE \eqref{mean-field-limit-rw-2-Z}. To this end, let $\Omega := C([0, T]; \mathbb{H}) =: \SSI$ be the space of all continuous functions on $[0,T]$ with values in $\mathbb{H}$. For any $w_1,w_2\in \SSI$, define $\dI(w_1,w_2):=\sup_{0\le t\le T}\norm{w_1(t)-w_2(t)}_2$. Since $[0,T]$ is a compact subset of $\mathbb{R}$ and $\mathbb{H}$ is a Polish space, it follows from \cite[Theorem 4.19]{Kechris1995} that $(\SSI,\dI)$ is also a Polish space.  Let $\mathscr{P}(\SSI)$ denote the set of all probability measures on $\SSI$, we use $\dLI$ to denote the dual-Lipschitz distance on $\mathscr{P}(\SSI)$.  Let $\mathscr{P}_2(\SSI)$ denote the Wasserstein space of order $2$, equipped with the $2$-Wasserstein distance $\WI$. Similarly, let $\mathscr{P}(\mathbb{H})$ denote the set of all probability measures on $\mathbb{H}$, and let $\mathscr{P}_2(\mathbb{H})$ denote the Wasserstein space of order $2$, equipped with the $2$-Wasserstein distance $\mathbf{W}_{2,\mathbb{H}}$. We use $w$ to denote a generic path in $\SSI$, and define the coordinate process by $\pi_t(w) = w_t$, $t\ge 0$. Let $\mathscr{F}_t:=\sigma\{ \pi_s: s\le t \}$, $t\ge 0$ denote  the filtration generated by the coordinate process. 

We now provide the definition of the martingale problem associated with the mean-field SDE \eqref{mean-field-limit-rw-2-Z}.

\begin{definition}\label{martingale-solution-Z}
    Let $\mu_0\in \mathscr{P}(\mathbb{H})$. A probability measure $\mathbf P\in \mathscr{P}(\SSI)$ is called a martingale solution to the mean-field SDE \eqref{mean-field-limit-rw-2-Z} with initial distribution $\mu_0$ if the following conditions are satisfied: 
\begin{enumerate}[(i)]
    \item $\mathbf P\circ \pi_0^{-1}=\mu_0$ and
    \begin{equation*}
        \begin{aligned}
             & \mathbf{P}\left( w\in \SSI:\int^T_0 \norm{\mathbf D(w_s,\mu_s)}_2\,\d s  <\infty\right)=1,
        \end{aligned}
    \end{equation*}
where $\mu_s:=\mathbf P\circ \pi_s^{-1}$, $\mathbf{D}(w_s,\mu_s)=(\mathbf{D}_x(w_s,\mu_s))$.  
\item For any $\theta\in \ell_c$, the process 
\begin{equation*}
    \begin{aligned}
         & \mathscr{M}_\theta(t,w,\mu):= \langle w_t,\theta \rangle - \langle w_0,\theta  \rangle-\int^t_0 \langle\mathbf{D}(w_s, \mu_s),\theta\rangle\,\d s,\quad 0\le t\le T,
    \end{aligned}
\end{equation*}
is a continuous square-integrable $(\mathscr{F}_t)$-martingale under $\mathbf{P}$, with
quadratic variation process given by 
\begin{equation*}
    \begin{aligned}
         & \langle \mathscr{M}_\theta\rangle(t,w,\mu)= \norm{\theta}^2t ,\quad 0\le t\le T,
    \end{aligned}
\end{equation*} 
where $\langle\cdot, \cdot\rangle$ denotes the $\ell^2$ inner product and $\| \cdot \|$ is the corresponding norm. 
\end{enumerate}
\end{definition}

We proceed to state the assumptions imposed on the initial data.

\begin{assumption}\label{initial-value-Z}
\begin{enumerate}[(1)]
    \item For any $N\in \mathbb{N}$,  the law of $$\Phi(0)=(\varphi^1,\cdots,\varphi^N)\in \mathcal Q $$ is symmetric on $\underbrace{  \mathbb{H}\times\cdots\times \mathbb{H} }_{\text{$N$ in total}}$\,, where $\varphi^i=(\varphi^i_x)$.      
    \item For the family of initial data $\{(\varphi^1,\cdots,\varphi^N) \}_{N\in \mathbb{N}}$, the empirical measure
\[
S^N_0:=\frac{1}{N} \sum^N_{j=1}\delta_{\varphi^{j}}
\] 
converges weakly in probability to some measure $\mu_0$ as $N\rightarrow \infty$.  

\item There exist constants $C>0$ and $\p>2$ such that
\[
\sup_N \mathbb{E}\left[\norm{\varphi^{1}}_{\p}^{\p}\right]< C.
\]
\end{enumerate}
\end{assumption}

If there exists a family of random vectors $\{(\varphi^1, \dots, \varphi^N)\}_{N \in \mathbb{N}}$  and a measure $\mu_0$ satisfying Assumption~\ref{initial-value-Z},  we call $\mu_0$ a measure determined by $\{(\varphi^1, \dots, \varphi^N)\}_{N \in \mathbb{N}}$. We denote by $\mathscr{P}_{\mathcal Q, \p}$ the set of all measures that can be determined by some $\{(\varphi^1, \dots, \varphi^N)\}_{N \in \mathbb{N}}$, where the parameter $\p$ refers to the constant in Assumption~\ref{initial-value-Z}$(3)$.

In the following, we show that the large $N$ limit of the system \eqref{the-new-model-Z} with initial data given by Assumption \ref{initial-value-Z} is \eqref{mean-field-limit-Z}. More precisely, we have
\begin{theorem}\label{Main-Theorem-Z}
Suppose that the initial data $\{(\varphi^1,\cdots,\varphi^N)\}_{N\in \mathbb{N}}$  of the system \eqref{the-new-model-Z} and $\mu_0$ satisfy Assumption \ref{initial-value-Z}, then the empirical measure $S^N$ of the solutions to the system \eqref{the-new-model-Z} (viewed as a $(\mathscr{P}_2(\SSI),\WI)$-valued random variable)  converges in law to $S$, and the limit $S$ is the unique martingale solution to the mean-field SDE \eqref{mean-field-limit-Z} with initial distribution $\mu_0$.   Furthermore, for the martingale solution, we have
\begin{equation}\label{Cov-Psi-Z}
    \begin{aligned}
\int \abs{v_x}^2\,S(t)(\d v)= 1,\quad \ \mathbb{P}\text{-\rm a.s.}, \quad \forall\, x\in \mathbb{Z}^d, \ 0\le t\le T,
    \end{aligned}
\end{equation}
and
\begin{equation*}\label{POC-Z}
    \begin{aligned}
         & \lim_{N\rightarrow \infty} \mathbb{E}\left[\WIT(S^N,S) \right]=0.
    \end{aligned}
\end{equation*}
\end{theorem}

\begin{remark}
By the uniqueness of the martingale solution to the mean-field SDE \eqref{mean-field-limit-Z}, the limit  $S$ of the sequence $\{S^N\}_{N\in \mathbb{N}}$ is $\mathbb{P}\text{-a.s.}$ constant in  $(\mathscr{P}_2(\SSI),\WI)$. 
\end{remark}

\begin{remark}\label{typical-examples-Z}
By Proposition~\ref{mu0assum-Z} below, a typical example satisfying Assumption \ref{initial-value-Z} is given by taking $\varphi$ as a random vector with distribution $\mu_{N,\kappa}$, and $\mu_0$ as the field $\bm{\mu}$ (see Proposition \ref{mu0assum-Z} below for more details).  For any $N \in \mathbb{N}$, another example  satisfying Assumption \ref{initial-value-Z} can be obtained by setting $\varphi_x^i = 1$ for all $x \in \mathbb{Z}^d$ and $i = 1, \ldots, N$, and taking $\mu_0 = \delta_{1}$.
\end{remark}

To compare the system \eqref{the-new-model-Z} with the mean-field SDE  \eqref{mean-field-limit-Z} directly,  for any fixed $i\in \mathbb{N}$, we consider the following mean-field SDE:   
\begin{equation}\label{Psi-i-limit-Z}
    \left\{\begin{aligned}
         \d\Psi^i_x(t)= &2\kappa\sum_{\N}\Big(  \Psi_y^i(t)-\mathbb{E}[ \Psi_x^i(t) \Psi_y^i(t)] \Psi_x^i(t)\Big)\,\d t-\frac{1}{2}\Psi_x^i(t)\,\d t+\d W^i_x(t),\quad 0\le t\le T,\\
\Psi_x^i(0)=& \psi_x^i, \quad x\in \mathbb{Z}^d,
\end{aligned}\right.
\end{equation}
where $\{ W_x^i \}$ is the same as in \eqref{the-new-model-Z}. 

\begin{theorem}\label{strong-solution-Z}
If the initial data $\psi$ of  the mean-field SDE \eqref{mean-field-limit-Z}   satisfies $\psi \laweq \mu_0 \in \mathscr{P}_{\mathcal Q,\p}$, then there exists a unique probabilistically strong solution to equation \eqref{mean-field-limit-Z} starting from $\psi$. Furthermore, suppose that the initial data $\{(\varphi^1,\cdots,\varphi^N)\}_{N\in \mathbb{N}}$ of the system \eqref{the-new-model-Z} and $\mu_0$ satisfy Assumption \ref{initial-value-Z}, $\psi^i \laweq \mu_0$, and that 
\begin{equation}\label{initial-value-i-Z}
    \begin{aligned}
         & \lim_{N\rightarrow \infty}\mathbb{E}\left[\norm{\varphi^i-\psi^i}_2^2\right]=0.
    \end{aligned}
\end{equation}
Then, we have
\begin{equation}\label{Convergence-2-Z}
    \begin{aligned}
         &  \lim_{N\rightarrow \infty}\mathbb{E}\left[ \sup_{0\le t\le T} \norm{\Phi^i(t)-\Psi^i(t)}_2^2 \right]=0,
    \end{aligned}
\end{equation}
where $\Phi^i$ and $\Psi^i$ denote the solutions to \eqref{the-new-model-Z} and \eqref{Psi-i-limit-Z} with initial data $\varphi^i$ and $\psi^i$, respectively.

If the initial data $\{(\varphi^1,\cdots,\varphi^N)\}_{N\in \mathbb{N}}$  of the system \eqref{the-new-model-Z} and $\mu_0$ satisfy Assumption \ref{initial-value-Z} with $\p=8$,  $\psi^i\laweq \mu_0$, in addition, the random variables $\{\psi^i\}_{i\in \mathbb{N}}$ are independent and identically distributed, and the sequence $\{(\varphi^i,\psi^i)\}^N_{i=1}$ is symmetric, then we have
\begin{equation}\label{rate-Z}
    \begin{aligned}
         & \sup_{0\le t\le T}\mathbb{E}\left[  \norm{\Phi^i(t)-\Psi^i(t)}_2^2 \right]\lesssim  \mathbb{E}\left[\norm{\varphi^i-\psi^i}_2^2\right]+\frac{1}{N}.
    \end{aligned}
\end{equation}
where the implicit constant is independent of $N$. 
\end{theorem}

\begin{remark}\label{Phi-Psi-Z}
For any $N \in \mathbb{N}$, a typical example 
that satisfies the assumptions in Theorem \ref{strong-solution-Z} is given by setting $\varphi_x^i = 1$ and $\psi_x^i=1$ for all $x \in \mathbb{Z}^d$ and $i = 1, \ldots, N$. Similar to the finite volume case, for small $\kappa$, by taking $\varphi \laweq \mu_{N, \kappa}$ and $\psi^i \laweq \bm{\mu}$ (see Proposition \ref{mu0assum-Z} below for more details), we can construct $\{(\varphi^i,\psi^i)\}^N_{i=1}$ satisfying the assumptions in the second part of Theorem~\ref{strong-solution-Z} by stationary coupling.
\end{remark}

\subsection{Propagation of chaos}
By carefully reviewing the proofs of Theorem \ref{Main-Theorem} and Theorem \ref{strong-solution}, we observe that only a few estimates and definitions require appropriate modification, and the corresponding results remain valid.

First, we provide bounds that are uniform in $N$ on the dynamics \eqref{the-new-model-Z}. 

\begin{proposition}\label{Lp-estimate-Zd}
Suppose that $\{ (\Phi^i_x), i=1,\cdots,N\}$ is the solution to the system \eqref{the-new-model-Z}.  Then, for any $i=1,\cdots,N$, and any $p\ge 2$,  we have
\begin{equation}\label{sup-1-p-Zd}
    \begin{aligned}
         &  \mathbb{E}\left[\sup_{0\le t\le T}\norm{\Phi^i(t)}_p^p\right]\le \mathbb{E}\left[ \sum_{x\in \mathbb{Z}^d}\sup_{0\le s\le T}\frac{1}{a^{\abs{x}}}\abs{\Phi_x^i(s)}^{p} \right] \le  C(d,p,\kappa,T) \left( \mathbb{E}\left[\norm{\varphi^i}_p^{p}\right] + 1\right),
    \end{aligned}
\end{equation}
where  $C(d,p,\kappa,T)$ is a constant depending on $d$, $p$, $\kappa$, $T$,  but independent of $N$.
\end{proposition}
\begin{proof}
   Recall the estimate \eqref{fix-x-estimate} in Proposition \ref{Lp-estimate}. For any $T\ge 0$, we have
\begin{equation}\label{Lp-w}
    \begin{aligned}
         \mathbb{E}\left[ \sup_{0\le s\le T}\abs{\Phi_x^i(s)}^{p} \right] \le& C(d,p,\kappa) \int^T_0 
          \mathbb{E}\left[ \sup_{0\le s\le t}\abs{\Phi_x^i(s)}^{p} \right]\,\d t +2\mathbb{E}\left[\abs{\varphi_x^i}^p\right]+C(p,T)\\
&+C(p,\kappa) \sum_{\N}\int^T_0 
          \mathbb{E}\left[ \sup_{0\le s\le t}\abs{\Phi_y^i(s)}^{p} \right]\,\d t,
    \end{aligned}
\end{equation}
where the constants $C(d, p, \kappa)$, $C(p, T)$ and $C(p,\kappa)$  are both independent of $\L$.  Multiplying both sides of \eqref{Lp-w} by $\frac{1}{a^{|x|}}$, summing over $x$, and using the estimates $\frac{1}{a^{|x|}} \leq a\, \frac{1}{a^{|y|}}$ for $x \sim y$ and $\sum_{x \in \mathbb{Z}^d} \frac{1}{a^{|x|}} < \infty$, together with Gronwall's inequality, we obtain \eqref{sup-1-p-Zd}.
\end{proof}
 
The proof of Lemma \ref{tight-P-2-infty} is similar to that of Lemma \ref{tight-P-2}.
\begin{lemma}\label{tight-P-2-infty}
       The sequence $\{ S^N \}_{N\in \mathbb{N}}$ is tight in $(\mathscr{P}_2(\SSI),\WI)$. 
\end{lemma}

\begin{proof}
Recall that $\SSI=C([0,T];\mathbb{H})$ denotes the space of all continuous functions on $[0,T]$ with values in $\mathbb{H}$, and for any $w_1,w_2\in \SSI$, 
$$
\dI(w_1,w_2)=\sup_{0\le t\le T} \norm{w_1(t)-w_2(t)}_2.
$$
   
We prove that the sequence $\{(\Phi_x^1) \}_{N\in \mathbb{N}}$ is tight in $(\SSI,\dI)$. By compact embedding $\ell^p_a\subset \ell^2_a$ for $p>2$,  it suffices to prove that 
\begin{enumerate}[(i)]
    \item for some $p>2 $, there exist a positive constant $M$ such that 
\begin{equation}\label{(i)}
    \begin{aligned}
         & \sup_{N\in \mathbb{N}}\mathbb{E}\left[\sup_{0\le s\le T}\norm{\Phi^1(s)}_p^p \right]< M;
    \end{aligned}
\end{equation}
\item  for some $p> 2$ and $\alpha<1/2$, there exists a positive constants $M$ such that for $s,t\in [0,T]$,
\begin{equation*}
    \begin{aligned}
     \sup_{N\in \mathbb{N}}\mathbb{E}\left[ \sup_{s\neq t\in[0,T]} \frac{\norm{\Phi^1(t)-\Phi^1(s)}_p}{\abs{t-s}^\alpha} \right]<M.
    \end{aligned}
\end{equation*}
\end{enumerate}
Condition (i) follows from Proposition \ref{Lp-estimate-Zd} and Assumption \ref{initial-value-Z} (3). We only need to prove condition (ii). According to Kolmogorov's criterion, it suffices to establish that for some $p>2$ and any $0\le s\le t\le T$,
\begin{equation}\label{lp-t}
    \begin{aligned}
         & \sup_{N\in \mathbb{N}}\mathbb{E}\left[ \norm{\Phi^1(t)-\Phi^1(s)}_p^p \right]\le M\left(\abs{t-s}^p +\abs{t-s}^{\frac{p}{2}}  \right).
    \end{aligned}
\end{equation}
By using an argument similar to  \eqref{Phi1-x-p}, we obtain that for some $p>2$ and any $0\le s\le t\le T$,
\begin{equation*}
    \begin{aligned}
         & \mathbb{E}\left[ \abs{\Phi_x^1(t)-\Phi_x^1(s)}^p\right]\\
\lesssim& \abs{t-s}^{p-1}\int^t_s \sum_{\N} \mathbb{E}\left[ \abs{\Phi^1_y(r)}^p \right]\,\d r +\abs{t-s}^{p-1}\int^t_s \mathbb{E}\left[ \abs{\Phi^1_x(r)}^p \right] \,\d r\\
+& \abs{t-s}^{\frac{p}{2}}+\abs{t-s}^{\frac{p}{2}-1}\int^t_s \mathbb{E}\left[ \abs{\Phi_x^1(r)}^p\right]\,\d r. 
    \end{aligned}
\end{equation*}
Hence, multiplying both sides by $\frac{1}{a^{\lvert x \rvert}}$, summing over $x$, and applying \eqref{(i)}, we obtain \eqref{lp-t}.

Note that $\mathscr{P}(\SSI)$ denotes the set of all probability measures on $\SSI$, and for any $\mu_1,\mu_2\in \mathscr{P}(\SSI)$, 
$$
\dLI(\mu_1,\mu_2)=\sup_{\norm{f}_{BL}\le 1} \abs{\int_{\SSI}f\mu_1(\d w) - \int_{\SSI}f\mu_2(\d w) },
$$
where $\norm{f}_{BL}:=\sup_{w\in \SSI} \abs{f(w)} + \sup_{w_1\neq w_2\in \SSI} \frac{\abs{f(w_1)-f(w_2)}}{\dI(w_1,w_2)}$. The metric $\dLI$ is the dual-Lipschitz metric on $\mathscr{P}(\SSI)$. Since $(\SSI, \dI)$ is a Polish space, the proof that $\{S^N\}_{N\in \mathbb{N}}$ is tight in $(\mathscr{P}(\SSI),\dLI)$ is the same as {\bf Step $2$} in the proof of Lemma \ref{tight-P-2}.

Let us recall that $\mathscr{P}_2(\SSI)$ denotes the Wasserstein space of order $2$, equipped with the $2$-Wasserstein distance
$$
\WI(\mu_1,\mu_2)=\left(\inf_{\tilde{\pi}\in \bm{\pi}(\mu_1,\mu_2)} \int_{\SSI\times \SSI} \dI(w_1,w_2)^2 \,\tilde{\pi}(\d w_1\d w_2) \right)^{\frac{1}{2}},\quad \mu_1,\mu_2\in \mathscr{P}_2(\SSI),
$$
where $\bm{\pi}(\mu_1,\mu_2)$ denotes the set of probability measures $\tilde{\pi}\in \mathscr{P}(\SSI\times \SSI)$ whose marginals are $\mu_1$ and $\mu_2$. Furthermore, considering Proposition \ref{Lp-estimate-Zd}, the proof that $\{S^N\}_{N\in \mathbb{N}}$ is tight in $(\mathscr{P}_2(\SSI),\WI)$ is analogous to {\bf Step $3$} in the proof of Lemma \ref{tight-P-2}. In fact, we can find a set $\mathcal{N}_\varepsilon \subset \mathscr{P}_2(\SSI)$ that is relatively compact in $(\mathscr{P}(\SSI),\dLI)$ and satisfies $\mathbb{P}\left(S^N \notin \mathcal N_\varepsilon\right)\le \varepsilon$. 
For $p=\p>2$, with $\p$ given in Assumption \ref{initial-value-Z}, let $$C_p:=\sup_{N\in \mathbb{N}}\mathbb{E}\left[\sup_{0\le t\le T} \norm{\Phi^1(t)}_p^p  \right]\lesssim 1,$$ and
\[
H_\varepsilon= \bigcap^{+\infty}_{m=1} \left\{  \nu\in \mathscr{P}_2(\SSI): \int \left(\sup_{0\le t\le T} \norm{w_t}_2\right)^2 \mathds{1}_{\left\{\sup_{0\le t\le T} \norm{w_t}_2\,\ge a_m\right\}}\nu(\d w) <\frac{1}{b_m}   \right\},
\]
where $a_m=m^{\frac{1}{p-2}}2^{\frac{m}{p-2}}$, $b_m=\frac{\varepsilon m}{C_p+1}$. Similarly to \eqref{H-e}, by Proposition \ref{Lp-estimate-Zd}, we have $\mathbb{P}(S^N\notin H_\varepsilon)\le \varepsilon$. By \cite[ Corollary $5.6$]{Carmona2018}, $\mathcal{N}_\varepsilon\cap H_\varepsilon $ is relatively compact in $(\mathscr{P}(\SSI),\WI)$. Hence, $\{ S^N \}_{N\in \mathbb{N}}$ is tight in  $(\mathscr{P}_2(\SSI),\WI)$.
\end{proof}

By Lemma \ref{tight-P-2-infty}, there exists a subsequence  $\{S^{N_k}\}_{k\in \mathbb{N}}$ such that $S^{N_k}\lawcon S$, as $k\rightarrow \infty$ ( for simplicity, we still denote the subsequence by $\{S^{N}\}_{N\in \mathbb{N}}$ ). By  Skorokhod's Theorem,  we can construct a probability space $(\hat \Omega,{\hat{\mathscr{F}}},\hat{\mathbb{P}})$ and random variables $\hat{S}^N$ and $\hat{S}$ such that $\hat S^N\laweq S^N$, $\hat{S}\laweq S$, and $\hat S^N\rightarrow\hat{S},\, \mathbb{\hat P}\text{-a.s.}$ in $(\mathscr{P}_2(\SSI),\WI)$, as $N\rightarrow \infty$. Similar to Lemma \ref{sphere-SN-S}, we present the following crucial estimates.

\begin{lemma}\label{vxvy-con}
    For any $x\in \mathbb{Z}^d$ and $0\le t\le T$, 
\begin{equation*}\label{SN=1}
    \begin{aligned}
         & \int \abs{v_x}^2 S^N(t)(\d v)=1, \quad \mathbb{P}\text{-\rm a.s.},
    \end{aligned}
\end{equation*}
\begin{equation}\label{S=1}
    \begin{aligned}
         &  \int \abs{v_x}^2 S(t)(\d v)=1,\quad \mathbb{P}\text{-\rm a.s.}.
    \end{aligned}
\end{equation}
Moreover, for any $p\ge 1$, $x,y \in \mathbb{Z}^d$, and $0\le t\le T$,
\begin{equation*}\label{vxvy-alpha-Z}
    \begin{aligned}
 & \mathbb{E}\left[ \abs{\int v_x v_yS^N(t)(\d v)-\int v_x v_yS(t)(\d v)}^{p} \right]\rightarrow 0.
    \end{aligned}
\end{equation*}
\end{lemma}
\begin{proof}
The proof follows the same line as that of Lemma \ref{sphere-SN-S}. The only difference is that,  for any $x,y\in \mathbb{Z}^d$, $w\in \SSI$, $\abs{w_{t,x}}^2\le a^{\abs{x}}\cdot\sup_{0\le t\le T}\norm{w_t}_2^2$, \, $\abs{w_{t,x} w_{t,y}}\le (a^{\abs{x}}\vee a^{\abs{y}})\cdot \sup_{0\le t\le T}\norm{w_t}_2^2 $.
\end{proof}

Thanks to Proposition~\ref{Lp-estimate-Zd} and Lemma~\ref{vxvy-con}, the argument for the existence of martingale solutions to the mean-field SDE~\eqref{mean-field-limit-Z} involves only minor technical differences from the finite volume case. Therefore, we omit the details.

We introduce two sets of probability measures. Specifically, we define
\begin{equation}\label{Pone-H}
    \begin{aligned}
         & \mathscr{P}_{\O}(\mathbb{H}):=\left\{ \mu\in \mathscr{P}(\mathbb{H}) \Big| \int \abs{v_x}^2 \d\mu(v)\le 1,\ \text{for all}\ x\in \mathbb{Z}^d\right\}.
    \end{aligned}
\end{equation}
We further define
\begin{equation}\label{Pone-Phi-I}
    \begin{aligned}
         &\mathscr{P}_{\O,\Psi}(\SSI):=\left\{ \mathbf{M}\in \mathscr{P}(\SSI) \Big|\ \mathbf{M}\circ \pi^{-1}_t\in \mathscr{P}_{\O}(\mathbb{H}),\ \text{for all}\  0\le t\le T\right\}.
    \end{aligned}
\end{equation}
By Theorem \ref{Main-Theorem-Z}, it is clear that the martingale solution to the mean-field SDE \eqref{mean-field-limit-Z} with initial distribution $\mu_0\in \mathscr{P}_{\mathcal Q,\p}$ belongs to $\mathscr{P}_{1,\Psi}(\SSI)$.

Let $S$ denote a martingale solution to the mean-field SDE \eqref{mean-field-limit-Z}.  To establish the uniqueness of the martingale solution as well as the existence and uniqueness of a probabilistically strong solution to the mean-field SDE \eqref{mean-field-limit-Z}, we introduce the following linear SDE :
\begin{equation}\label{fix-mu-Zd}
    \left\{\begin{aligned}
         d \widetilde \Psi_x(t)= &2\kappa\sum_{\N}\Big(  \widetilde \Psi_y(t)-\bm{h}_{x,y}(t)\widetilde {\Psi}_x(t)\Big)\,\d t-\frac{1}{2}\widetilde {\Psi}_x(t)\,\d t+\d\widetilde{W}_x(t),\quad 0\le t\le T,\\
\widetilde \Psi_x(0)=& \widetilde \psi_x, \quad x\in \mathbb{Z}^d, 
\end{aligned}\right.
\end{equation}
where $\widetilde{\psi}\laweq S(0)$, $\bm{h}_{x,y}(t):=\int v_x v_y S(t)(\d v)$, with $S(t)$ denoting the law at time $t$ of the solution $\Psi$ to \eqref{mean-field-limit-Z}.  

We show that the pathwise uniqueness holds for both the SDE \eqref{fix-mu-Zd} and the mean-field SDE \eqref{mean-field-limit-Z}. 
\begin{lemma}\label{fixed-mu-le-Zd}
For any $S\in \mathscr{P}_{\O,\Psi}(\SSI)$,  the pathwise uniqueness holds for the SDE \eqref{fix-mu-Zd}.
\end{lemma}
\begin{proof}
Since $S\in \mathscr{P}_{\O,\Psi}(\SSI)$, we have $\int \abs{v_x}^2 S(t)(\d v)\le 1$, for all $x\in \mathbb{Z}^d$, $0\le t\le T$. By  H\"older's inequality, for any $x,y\in \mathbb{Z}^d$ and $0\le t\le T$, we have $\abs{\bm{h}_{x,y}(t)}\le 1$. Therefore, the SDE \eqref{fix-mu-Zd} satisfies the Lipschitz condition, and the pathwise uniqueness holds.  
\end{proof}
\begin{lemma}\label{not-fix-mu-Zd}
The pathwise uniqueness holds in $\mathscr{P}_{\O,\Psi}(\SSI)$ for the mean-field SDE \eqref{mean-field-limit-Z}.
\end{lemma}
\begin{proof}
Let $\Psi^{(1)}$ and $\Psi^{(2)}$ be any two weak solutions to the mean-field SDE \eqref{mean-field-limit-Z}, defined on the same probability space and driven by the same Brownian motion, such that the laws of $\Psi^{(1)}$ and $\Psi^{(2)}$ belong to $\mathscr{P}_{\O, \Psi}(\SSI)$.  Similar to the proof of Lemma \ref{not-fix-mu},  we conclude that for any $0\le t\le T$,
\begin{equation}\label{unique-path-Z}
    \begin{aligned}
           &\mathbb{E}\left[
\abs{\Psi_x^{(1)}-\Psi_x^{(2)}}^2(t)\right]\\
\lesssim&  \sum_{\N}\int^t_0  \mathbb{E}\left[\abs{ \Psi_y^{(1)}-\Psi_y^{(2)} }^2(s)\right]\,\d s+\int^t_0  \mathbb{E}\left[\abs{ \Psi_x^{(1)}-\Psi_x^{(2)} }^2(s)\right]\,\d s. 
    \end{aligned}
\end{equation}
Multiplying both sides of \eqref{unique-path-Z} by $\frac{1}{a^{\abs{x}}}$, summing over $x$, using $\frac{1}{a^{\abs{x}}}\le a \frac{1}{a^{\abs{y}}}$ for $x\sim y$ and Gronwall's inequality, we obtain that for any $0\le t\le T$,
\begin{equation}\label{Z-path-un}
    \begin{aligned}
         & \mathbb{E}\left[\norm{\Psi^{(1)}-\Psi^{(2)}}_2^2(t)\right]=0.
    \end{aligned}
\end{equation}
By \eqref{Z-path-un} and $\Psi^{(i)}(\cdot)\in C([0,T];\mathbb{H})$ a.s., $i=1,2$, we obtain that $\Psi^{(1)}(t)=\Psi^{(2)}(t)$ for all $0\le t\le T$\  $\text{a.s.}$.
\end{proof}
\begin{remark}
Similar to the finite volume case, the bound $1$ in the definition of the measure set $\mathscr{P}_{\O}(\mathbb{H})$ can in fact be replaced by any constant $C \geq 1$.
\end{remark}

By combining the Yamada-Watanabe theorem from \cite{Rockner2008} with the proof of \cite[Lemma 2.1]{Huang2021}, it follows that the conclusion of the Yamada-Watanabe theorem in \cite[Lemma 2.1]{Huang2021} remains valid in the infinite volume case. Applying the argument for the uniqueness of the martingale solution to \eqref{mean-field-limit}, we deduce that the martingale solution to the mean-field SDE \eqref{mean-field-limit-Z} is unique.  More precisely, for any two martingale solutions $S_1, S_2 \in \mathscr{P}_{\O, \Psi}(\SSI)$ to the mean-field SDE~\eqref{mean-field-limit-Z}, if $S_1(0) = S_2(0)$, then $S_1 = S_2$.  Therefore, $S(\widetilde{\omega})$ is independent of $\widetilde{\omega}$, and we still denote it by $S$.

The proof of $\lim_{N\rightarrow \infty} \mathbb{E}\left[\WIT(S^N, S)\right] = 0$ in the infinite volume case is analogous to that in the finite volume case, except that $\sup_{0\le t \le T}\abs{w_t}^2$ is replaced by $\sup_{0\le t \le T}\norm{w_t}_2^2$.

This completes the proof of Theorem \ref{Main-Theorem-Z}. \qed

\subsection{Convergence  rate}  The proof of Theorem \ref{strong-solution-Z} follows the same line of argument as that of Theorem \ref{strong-solution}. Therefore, we omit the details and only present the necessary estimates to highlight the differences arising in the current setting.

As the conclusion of the Yamada-Watanabe theorem in \cite[Lemma 2.1]{Huang2021} holds in the infinite volume case, by an argument analogous to that in the finite volume case, we deduce that there exists a unique probabilistically strong solution to the mean-field SDE \eqref{mean-field-limit-Z}.

Similar to the proof of Proposition \ref{Lp-estimate-Zd}, by combining \eqref{S=1}, we can obtain that for any $p\ge 2$,
\begin{equation}\label{Psi-sup-1-p-Z}
    \begin{aligned}
         & \mathbb{E}\left[\sup_{0\le t\le T}\norm{\Psi(t)}_p^p\right]\le \mathbb{E}\left[ \sum_{x\in \mathbb{Z}^d}\sup_{0\le s\le t}\frac{1}{a^{\abs{x}}}\abs{\Psi_x(s)}^{p} \right]  \le C(d,p,\kappa,T) \left(\mathbb{E}\left[\norm{\psi}_p^{p}\right] + 1\right).
    \end{aligned}
\end{equation}

In the following, we provide two estimates to show the differences arising in the current setting.  The first concerns the quantity $A_2$, as introduced in {\bf Step 2} of the proof of Theorem~\ref{strong-solution}.

By {\bf Step $2$} in the proof of  Theorem \ref{strong-solution}, we obtain that for any $0\le t\le T$,
\begin{equation}\label{SN-S-Zd}
    \begin{aligned}
         & \mathbb{E}\left[ \sup_{0\le s\le t} \abs{\Phi_x^i(s)-\Psi_x^i(s)}^2 \right]\\
\lesssim& \sum_{\N}\int^t_0 \mathbb{E}\left[\abs{\frac{\sum^N_{k=1} \Phi_x^i(s)\Phi_y^i(s) }{N}-\mathbb{E}[\Psi_x^i(s)\Psi_y^i(s)] }^2 \abs{\Phi_x^i(s)}^2\right]\,\d s\\
&+\mathbb{E}\left[ \abs{\varphi_x^i-\psi_x^i}^2 \right] +\frac{\mathbb{E}\left[\sup_{0\le s\le t} \abs{\Phi_x^i(s)}^2\right]}{N}+\int^t_0 \mathbb{E}\left[ \sup_{0\le r\le s} \abs{\Phi_x^i(r)-\Psi_x^i(r)}^2 \right]\,\d s.
    \end{aligned}
\end{equation}  

By multiplying both sides of \eqref{SN-S-Zd} by $\frac{1}{a^{\abs{x}}}$, summing over $x$, we will consider the following estimate:
\begin{equation*}
    \begin{aligned}
         & \sum_{x\in \mathbb{Z}^d}\frac{1}{a^{\abs{x}}}\mathbb{E}\left[\sum_{\N} \abs{\frac{\sum^N_{k=1} \Phi_x^i(s)\Phi_y^i(s) }{N}-\mathbb{E}[\Psi_x^i(s)\Psi_y^i(s)] }^2 \abs{\Phi_x^i(s)}^2\right]\\
\lesssim& \mathbb{E}\left[\sum_{x\in \mathbb{Z}^d}\frac{1}{a^{\abs{x}}} \sum_{\N} \abs{\frac{\sum^N_{k=1} \Phi_x^i(s)\Phi_y^i(s) }{N}-\mathbb{E}[\Psi_x^i(s)\Psi_y^i(s)] }^{2q} \right]^{\frac{1}{q}}\cdot\mathbb{E}\left[\norm{\Phi^i(s)}_{2p}^{2p}  \right]^{\frac{1}{p}}, \quad 0\le s\le t\le T,
    \end{aligned}
\end{equation*}
where $1/p+1/q=1$ with $p={\p}/{2}$, $\p$ given in Assumption \ref{initial-value-Z}. Note that 
$$\frac{1}{N}\sum^N_{k=1} \Phi_x^i(t)\Phi_y^i(t)=\int v_x v_y S^N(t)(\d v),\quad\mathbb{E}[\Psi^i_x(t)\Psi_y^i(t)]=\int v_x v_y S(t)(\d v).$$ Combining Gronwall's inequality, \eqref{Psi-sup-1-p-Z}, Proposition \ref{Lp-estimate-Zd}, and the dominated convergence theorem, we  will complete the proof of \eqref{Convergence-2-Z} in Theorem \ref{strong-solution-Z}.

To prove \eqref{rate-Z} in Theorem~\ref{strong-solution-Z}, we proceed as in {\bf Step~3} of the proof of Theorem~\ref{strong-solution}. In particular, it is necessary to further decompose $A_2$ as $A_2 = \sum_{k=1}^{4} A_{2k}$. In the following, we focus on the treatment of the term $A_{23}$. Note that
\begin{equation}\label{A23-sec4}
    \begin{aligned}
         \frac{1}{N}\sum^{N}_{i=1}\abs{A_{23}}
\lesssim&  \left[\frac{1}{N}\sum^N_{k=1} \abs{\Phi^k_y-\Psi^k_y}^2\right]+\left[\frac{1}{N}\sum^N_{i=1}\left(\abs{\Psi^i_x}^2-\mathbb{E}\left[\abs{\Psi^i_x}^2\right]\right) \right]^2 \left[\frac{1}{N}\sum^N_{i=1} \abs{\Phi^i_x-\Psi^i_x}^2\right]\\
&+ \left[\frac{1}{N}\sum^N_{i=1} \abs{\Phi^i_x-\Psi^i_x}^2\right].
    \end{aligned}
\end{equation} 
Multiplying both sides of \eqref{A23-sec4} by $\frac{1}{a^{|x|}}$ and then summing over $x$, for the second term on the right-hand side we have
\begin{equation*}
    \begin{aligned}
         & \sum_{x\in \mathbb{Z}^d}\frac{1}{a^{\abs{x}}}\mathbb{E}\left[ \abs{\frac{1}{N}\sum^N_{i=1}U_x^i}^2 \frac{1}{N}\sum^N_{i=1}\abs{\Phi_x^i-\Psi_x^i}^2  \right]\\
\lesssim& \left(\sum_{x\in \mathbb{Z}^d}\frac{1}{a^{\abs{x}}}  \mathbb{E}\left[\abs{\frac{1}{N}\sum^N_{i=1}U_x^i}^4 \right]\right)^{\frac{1}{2}} \left(\sum_{x\in \mathbb{Z}^d}\frac{1}{a^{\abs{x}}}  \mathbb{E}\left[ \abs{\frac{1}{N}\sum^N_{i=1}\abs{\Phi_x^i-\Psi_x^i}^2 }^2\right]\right)^{\frac{1}{2}}\\
\lesssim& \frac{1}{N}\cdot \left( \frac{1}{N} \sum_{i=1}^N \left(\norm{\Phi^i}_4^4+ \norm{\Psi^i}_4^4 \right)\right)^{\frac{1}{2}}\lesssim \frac{1}{N},
    \end{aligned}
\end{equation*}
where $U^i_x=\abs{\Psi_x^i}^2-\mathbb{E}[|\Psi_x^i|^2]$. The adjustments for $A_{21}$, $A_{22}$ and $A_{24}$ are analogous to those for $A_{23}$. Finally, we will complete the proof of \eqref{rate-Z} in Theorem \ref{strong-solution-Z}.

\section{Stationary measure}\label{Invariant measure}
In this section, we focus on the stationary measures to the mean-field SDEs \eqref{mean-field-limit} and \eqref{mean-field-limit-Z}. First, we study the existence of stationary measures for these mean-field SDEs by reducing the problem to the discrete stochastic heat equation and verifying the corresponding self-consistency condition. 
Second, we investigate the large $N$ behavior of the invariant measures for \eqref{the-new-model} and \eqref{the-new-model-Z} by coupling the respective Langevin dynamics with the mean-field SDEs \eqref{mean-field-limit} and \eqref{mean-field-limit-Z}, and employing It\^o's calculus. Finally, we prove that there exists at most one stationary measure to the mean-field SDEs \eqref{mean-field-limit} and \eqref{mean-field-limit-Z} by directly comparing any two stationary solutions.

\subsection{Existence of the stationary measures}
In this section, we study the existence of stationary measures for the mean-field SDE \eqref{mean-field-limit-n-i}. For the finite volume case, we consider the stationary measure in $\mathscr{P}_{\mathrm{const}}(\mathbb{R}^{\abs{\Lambda}})$ defined in \eqref{Pconst-i} and verify the corresponding self-consistency condition (see \eqref{fixed-poit} below). For the infinite volume case, when $\kappa < \kappa_c$, we carry out an analysis similar to the finite volume case. For $\kappa \ge \kappa_c$, we follow the argument presented in \cite[Proposition 3.5]{Shen2023} to establish the existence of stationary measures  for the mean-field SDE \eqref{mean-field-limit-Z}.

We first prove the following Lemma \ref{un-N}, which provides uniform in $N$ $p$-th moment estimates for the spin $O(N)$ measure defined by \eqref{spin-O-N}.  To this end, we provide a proposition that establishes the existence of invariant measures for the Langevin dynamics \eqref{the-new-model} and, since the proof is similar to that of \cite[Lemma~3.5 ($M = \mathbb{S}^{N-1}$)]{Shen2024}, we omit it here.
\begin{proposition}\label{invariant-measures}
    The spin $O(N)$ measure \eqref{spin-O-N} is invariant under the finite volume dynamics  \eqref{the-new-model}. 
\end{proposition}
Then, we have the following lemma.
\begin{lemma}\label{un-N}
If $\kappa > 0$ and $p > 2$ satisfy 
\begin{equation}\label{p-mon}
    \begin{aligned}
         & \kappa<\frac{p}{16d(p-1)^{1-\frac{1}{p}}+8dp},
    \end{aligned}
\end{equation}
then the field $\mu_{\Lambda,N,\kappa}$ defined by \eqref{spin-O-N}, satisfies
\begin{equation*}
    \begin{aligned}
         & \sum_{x\in \Lambda}\int \abs{\Phi^i_x}^p \d \mu_{\Lambda,N,\kappa} (\Phi)\lesssim 1,
    \end{aligned}
\end{equation*}
where the implicit constant depends on $\abs{\Lambda}$, but is independent of $N$.
\end{lemma}
\begin{proof}
Let $\Phi$ be a stationary solution to \eqref{the-new-model} with $\Phi(0)\laweq \mu_{\Lambda,N,\kappa}$.  Since the spin  $O(N)$ model on $\Lambda$ (see \eqref{spin-O-N}) is an invariant measure to the SDEs \eqref{the-new-model},  
we only need to prove that if $\kappa<{(p-2\varepsilon})/{(16d(p-1)^{1-\frac{1}{p}}+8dp)}$ with any fixed $\varepsilon>0$, then
\begin{equation}\label{ieq:1}
    \begin{aligned}
         & \sup_N\sum_{x\in \Lambda}\mathbb{E}\left[ \abs{{\Phi}_x^i}^p \right]\le \frac{2C(\varepsilon,p)\abs{\Lambda}}{p-2\varepsilon-16\kappa d(p-1)^{1-\frac{1}{p}}-8\kappa d p},
    \end{aligned}
\end{equation} 
where $C(\varepsilon,p)$ is a constant that depends only on $\varepsilon$ and $p$.

By applying It\^o's formula to the stationary solution, taking expectations, and using the stationary property to eliminate the initial data, we obtain
\begin{equation*}\label{Phi-p-es}
    \begin{aligned}
      &\frac{p(N+p-2)}{2N} \mathbb{E}\left[\abs{{\Phi}_x^i}^p\right]\\
=& 2\kappa p \,\mathbb{E}\left[ \sum_{\N}\left( \Phi_y^i \Phi_x^i \abs{\Phi_x^i}^{p-2} -\frac{1}{N}\sum^N_{k=1} \Phi_x^k \Phi_y^k \abs{\Phi_x^i}^p\right) \right]+ \frac{p(p-1)}{2}\mathbb{E}\left[\abs{\Phi_x^i}^{p-2}\right]\\
\le& 2\kappa p\, \mathbb{E}\left[\sum_{\N}  \abs{\Phi_y^i} \abs{\Phi_x^i}^{p-1}\right] +2\kappa p\, \mathbb{E}\left[\sum_{\N} \abs{\Phi^i_x}^p\right]+ \frac{p(p-1)}{2}\mathbb{E}\left[\abs{\Phi_x^i}^{p-2}\right]\\
:=&I_1+I_2+II.
    \end{aligned}
\end{equation*}
Then, by Young's inequality, for any $\varepsilon > 0$, we have
\begin{equation}\label{IIes}
II \leq \varepsilon\mathbb{E}\left[\abs{\Phi_x^i}^p\right] + C(\varepsilon, p).
\end{equation}
We next estimate $I_1$. Observe that
\begin{equation}\label{I1es}
\begin{aligned}
I_1\leq 2\kappa (p-1)^{1-1/p} \sum_{\N} \mathbb{E}\left[ \abs{\Phi_x^i}^p + \abs{\Phi_y^i}^p \right],
\end{aligned}
\end{equation}
where the inequality follows from Young's inequality,
$$
ab \leq \frac{a^p}{p} + \frac{b^{p/(p-1)}}{p/(p-1)},
$$
with
$$
a = (p-1)^{\frac{p-1}{p^2}} \abs{\Phi_y^i}, \quad
b = \left(\frac{1}{p-1}\right)^{\frac{p-1}{p^2}} \abs{\Phi_x^i}^{p-1}.
$$
Combining \eqref{IIes} and \eqref{I1es}, we obtain
\begin{equation}\label{finite-es}
    \begin{aligned}
          \frac{p(N+p-2)}{2N}  \mathbb{E}[|\Phi_x^i|^p]
\le & \left(\varepsilon + 4\kappa d (p-1)^{1-1/p} + 4\kappa p d \right) \mathbb{E}[|\Phi_x^i|^p]\\
& + 2\kappa  (p-1)^{1-1/p}\sum_{\N}\mathbb{E}[|\Phi_y^i|^p] +C(\varepsilon, p).
    \end{aligned}
\end{equation}
Note that $\frac{p}{2} \leq \frac{p(N+p-2)}{2N}$. Summing over $x \in \Lambda$ yields
\begin{equation*}
\frac{p}{2} \sum_{x \in \Lambda} \mathbb{E}[|\Phi_x^i|^p]
\leq \left(\varepsilon + 8\kappa d (p-1)^{1-1/p} + 4\kappa p d \right) \sum_{x\in \Lambda}\mathbb{E}[|\Phi_x^i|^p] + C(\varepsilon, p) |\Lambda|,
\end{equation*}
which implies inequality~\eqref{ieq:1}.
\end{proof}

We now present the definition of a stationary measure for the mean-field SDE~\eqref{mean-field-limit}. 
\begin{definition}\label{sta-F}
A probability measure $\mu$ is called a stationary measure for the mean-field SDE~\eqref{mean-field-limit} if there exists a martingale solution $\mathbf{P}$ to the mean-field SDE \eqref{mean-field-limit} with initial distribution $\mu$, and $\mathbf{P}\circ \pi^{-1}_t=\mu$ for all $0\le t\le T$.
\end{definition}

\noindent A martingale solution to the mean-field SDE \eqref{mean-field-limit} is called a stationary solution if its initial distribution is a stationary measure.

The following theorem provides the existence of stationary measures for the mean-field SDE \eqref{mean-field-limit}. Define
\begin{equation*}\label{scrP}
    \begin{aligned}
         &\mathscr{P}_{\mathrm{const}}(\mathbb{R}^{\abs{\Lambda}}) := \left\{
            \mu \in \mathscr{P}(\mathbb{R}^{|\Lambda|})
            \;\middle|\;
            \text{there exists } u \in \mathbb{R} \text{ such that } 
            \int v_x v_y\, \mathrm{d}\mu(v) = u, \text{ for all } x \sim y
         \right\}.
    \end{aligned}
\end{equation*}
 Let $\mu^\L$ be a massive Gaussian free field scaled by ${1}/{(2\sqrt{\kappa})}$ with mass $m^2>0$, where $m^2$ is the unique positive solution to 
\begin{equation*}\label{iv-G}
    \begin{aligned}
         & G_{\Lambda,m^2}(x,x)=4\kappa.
    \end{aligned}
\end{equation*}
\begin{theorem}\label{Aru-cor-im}
The field $\mu^\L$ is the unique stationary measure in $\mathscr{P}_{\mathrm{const}}(\mathbb{R}^{\abs{\Lambda}})$ to the mean-field SDE \eqref{mean-field-limit}.
\end{theorem}
\begin{proof}
 Considering the following SDE:
\begin{equation}\label{rewrite-mean-limit}
    \begin{aligned}
        \d {\Psi}^u_x(t) = &2\kappa\sum_{\N} \left( \Psi^u_y(t)- u\Psi^u_x(t) \right)\,\d t -\frac{1}{2} \Psi^u_x(t)\,\d t+\d W_x(t)\\
=& (2\kappa\,\Delta_{\Lambda} -2\kappa \mathfrak{m}){\Psi}^u_x(t)\,  \d t +\,\d W_x(t), \quad  x\in \Lambda, \ 0\le t\le T,
    \end{aligned}
\end{equation}
where $\mathfrak{m}:=2d({u}-1)+1/(4\kappa)$.

If $\mathfrak{m}>0$, it is well known that there exists an invariant measure to \eqref{rewrite-mean-limit},  which is 
\begin{equation*}\label{mu-m}
    \begin{aligned}
         & N\left(0,\left[-2(2\kappa \Delta_{\Lambda} -2\kappa\mathfrak{m})\right]^{-1}\right)= \frac{1}{2\sqrt{\kappa}}N\left(0,\left[-\Delta_{\Lambda}+\mathfrak{m}\right]^{-1}\right):=\mu_{\mathfrak{m}},
    \end{aligned}
\end{equation*}
a massive Gaussian free field on $\Lambda$  scaled by $\frac{1}{2\sqrt{\kappa}}$ with mass $\mathfrak{m}$.   Since, for any massive Gaussian free field on $\Lambda$, the correlation function is invariant under translations and rotations by $\frac{\pi}{2}$, it follows that $\mu_{\mathfrak{m}}\in \mathscr{P}_{\mathrm{const}}(\mathbb{R}^{\abs{\Lambda}})$. 

Note that $u=\frac{1}{2d} \mathfrak{m}+1-\frac{1}{8\kappa d}$, it is easy to see that $\mu_{\mathfrak{m}} $ is a stationary measure to the mean-field SDE \eqref{mean-field-limit}   if and only if there exists a $\mathfrak{m}>0$ such that the following self-consistency condition holds:
\begin{equation}\label{fixed-poit}
    \begin{aligned}
         & \int v_{\bm{0}} v_{\bm{1}}\,\d \mu_{\mathfrak{m}}(v)= \frac{1}{4\kappa} G_{\Lambda,\mathfrak{m}}(\bm{0},\bm{1})=u=\frac{1}{2d} \mathfrak{m}+1-\frac{1}{8\kappa d},
    \end{aligned}
\end{equation}
where we choose 
\begin{equation*}
    \begin{aligned}
         & \Lambda=[0,\mathfrak{L})^d\cap \mathbb{Z}^d,\ \  \bm{0}:=(0,\cdots,0)\in \Lambda,\ \  \bm{1}:=(1,0,\cdots,0)\in \Lambda.
    \end{aligned}
\end{equation*}
In order to calculate the exact representation of $\int v_{\bm{0}} v_{\bm{1}}\,\d \mu_{\mathfrak{m}}(v)$ in terms of $\mathfrak{m}$.  We recall the standard Fourier basis on the discrete torus $\Lambda$, which is given by
\begin{equation}\label{elk}
    \begin{aligned}
         & e_k^{\L}(x) := \frac{1}{\L^{\frac{d}{2}}} \exp\left(2\pi i \frac{k\cdot x}{\L}\right),
    \end{aligned}
\end{equation}
where $x, k \in \Lambda$, and $k \cdot x := \sum_{j=1}^d k_j x_j$ denotes the standard inner product. The normalization factor $1/\L^{d/2}$ ensures that the family $\{e_k^{\L}\}_{k \in \Lambda}$ forms an orthonormal basis with respect to the inner product
\begin{equation}\label{inner-product-C}
    \begin{aligned}
         & \langle f, g \rangle_{\mathbb{C}} := \sum_{x \in \Lambda} f(x) \overline{g(x)}.
    \end{aligned}
\end{equation}
This basis allows any complex-valued function $f: \Lambda \to \mathbb{C}$ to be expanded as
\begin{equation*}
    \begin{aligned}
         & f(x) = \sum_{k \in \Lambda} \widehat{f}(k) e_k^{\L}(x),
    \end{aligned}
\end{equation*}
where the Fourier coefficient $\widehat{f}(k)$ is given by
\begin{equation*}
    \begin{aligned}
         & \widehat{f}(k) = \sum_{x \in \Lambda} f(x) \overline{e_k^{\L}(x)}.
    \end{aligned}
\end{equation*}
Note that for any $x\in \Lambda$, we have
\begin{equation}\label{spec-lambda}
    \begin{aligned}
        \Delta_{\Lambda} e^{\L}_k(x) & =\frac{1}{\L^{\frac{d}{2}}} e^{2\pi i \frac{k\cdot x}{\L}} \sum^d_{j=1} \left(  e^{2\pi i \frac{k_j}{\L}} +  e^{-2\pi i \frac{k_j}{\L}} -2\right)\\
& =2\sum_{j=1}^d \left( \cos\left(2\pi \frac{k_j}{\L}\right) -1\right) e^{\L}_k(x):=\lambda^{\L}_k e^{\L}_k(x).
    \end{aligned}
\end{equation}
By the definition of the Green's function $G_{\Lambda, \mathfrak{m}}(x,y)$, $(- \Delta_\Lambda + \mathfrak{m}) G_{\Lambda, \mathfrak{m}}(x, y) = \delta_{x,y}$, where $\delta_{x,y} = 1$ if $x = y$, and $\delta_{x,y} = 0$ otherwise. By direct calculation, we obtain that
\begin{equation}\label{F-correlation}
    \begin{aligned}
         & G_{\Lambda,\mathfrak{m}}(x, y) = \frac{1}{\mathfrak{L}^d} \sum_{k\in \Lambda} \frac{e^{2\pi i \frac{k \cdot (x-y)}{\mathfrak{L}}}} {-\lambda^\L_k +\mathfrak{m}}= \frac{1}{\mathfrak{L}^d} \sum_{k\in \Lambda} \frac{\cos\left(2\pi \frac{k\cdot (x-y)}{\L}\right)} {-\lambda^\L_k +\mathfrak{m}}, \quad x,\,y \in \Lambda.
    \end{aligned}
\end{equation}
In particular, we have
\begin{equation*}
    \begin{aligned}
         & G_{\Lambda,\mathfrak{m}}(\bm{0},\bm{1})=\frac{1}{\mathfrak{L}^d}\frac{1}{\mathfrak{m}} +\frac{1}{\mathfrak{L}^d}\sum_{k\in\Lambda\backslash \{\bm{0}\}} \frac{\cos\left( 2\pi\frac{ k_1}{\L}\right)}{-\lambda^\L_k +\mathfrak{m}}. 
    \end{aligned}
\end{equation*}
Note that 
\begin{equation*}
    \begin{aligned}
         & \lim_{\mathfrak{m}\rightarrow 0^+} G_{\Lambda,\mathfrak{m}}(\bm{0},\bm{1})=+\infty,\quad \lim_{\mathfrak{m}\rightarrow +\infty} G_{\Lambda,\mathfrak{m}}(\bm{0},\bm{1})=0,
    \end{aligned}
\end{equation*}
and $G_{\Lambda,\mathfrak{m}}(\bm{0},\bm{1})$ is continuous on $(0, +\infty)$ with respect to $\mathfrak{m}$. It is easy to see that equation \eqref{fixed-poit} admits at least one positive solution. Therefore, there exists a martingale solution $\mathbf{P}$ to the mean-field SDE~\eqref{mean-field-limit} such that $\mathbf{P}\circ\pi^{-1}_t=\mu_{\mathfrak{m}}$ for all $0\le t\le T$.

Let $\mathfrak{m}>0$ be any positive solution to equation \eqref{fixed-poit}. Since the martingale solution $\mathbf{P}$ starting from $\mu_{\mathfrak{m}}$ to the mean-field SDE~\eqref{mean-field-limit} satisfies $\mathbf{P}\circ \pi^{-1}_t=\mu_{\mathfrak{m}}$ for all $0\le t\le T$, applying It\^o's formula to $\abs{{\Psi_x}(t)}^2$, $x\in \Lambda$ and taking expectations, we have
\begin{equation*}
    \begin{aligned}
         & (8d\kappa{u}+1)\cdot(1-\mathbb{E}[\abs{{\Psi}_x}^2])=0.
    \end{aligned}
\end{equation*}
Since $\mathfrak{m}>0$, we obtain $8d\kappa {u} +1\neq 0$. Therefore,
$\mathbb{E}\left[ |\Psi_x|^2 \right]=1, x\in \Lambda$,  which implies that
\begin{equation*}\label{Green-kappa-1}
    \begin{aligned}
     & G_{\Lambda,\mathfrak{m}}(x,x)=\mathbb{E}\left[ 2\sqrt{\kappa}{\Psi}_x \cdot 2\sqrt{\kappa}{\Psi}_x \right]=4\kappa,
    \end{aligned}
\end{equation*}
or, using \eqref{F-correlation}, more explicitly,
\begin{equation}\label{Green-kappa-2}
    \begin{aligned}
     \frac{1}{\mathfrak{L}^d}\frac{1}{\mathfrak{m}} +\frac{1}{\mathfrak{L}^d}\sum_{k\in\Lambda\backslash \{\bm{0}\}} \frac{1}{-\lambda^\L_k +\mathfrak{m}}=4\kappa.
    \end{aligned}
\end{equation}
Note that
\begin{equation*}
    \begin{aligned}
         & \lim_{\mathfrak{m}\rightarrow 0^+} G_{\Lambda,\mathfrak{m}}(x,x)=+\infty,\quad \lim_{\mathfrak{m}\rightarrow +\infty} G_{\Lambda,\mathfrak{m}}(x,x)=0,
    \end{aligned}
\end{equation*}
and $G_{\Lambda,\mathfrak{m}}(x,x)$ is monotonically decreasing on $(0,+\infty)$ with respect to $\mathfrak{m}$. Therefore, there is a unique positive solution to \eqref{Green-kappa-2}. This completes the proof.
\end{proof}
Since the martingale solution $\mathbf{P}$ in the proof of Theorem~\ref{Aru-cor-im} belongs to $\mathscr{P}_{\O,\Psi}(\mathbb{R}^{|\Lambda|})$, by applying the Yamada-Watanabe argument as in the proof of Theorem~\ref{strong-solution}, we obtain the following corollary directly.
\begin{corollary}\label{muL-s}
Suppose the initial data $\psi \laweq \mu^\L$,  then there exists a unique probabilistically strong solution to the mean-field SDE \eqref{mean-field-limit} starting from $\psi$.    
\end{corollary}

In the following, we present a proposition that provides an example of $\varphi$ and $\mu_0$ satisfying Assumption~\ref{initial-value}. Recall the classical definition of Kac's chaos. Let $\X$ be a Polish space, and let $\mu \in \mathscr{P}(\X)$ be a probability measure on $\X$. Consider a sequence $\{\mu^N\}_{N \in \mathbb{N}}$ of symmetric probability measures on the product space $\X^N$. We say that the sequence $\{\mu^N\}_{N \in \mathbb{N}}$ is $\mu$-chaotic if for any $1\le k\le N$, the marginal $\mu^{k,N} := \mu^N \circ (\mathbb{V}^{(k)})^{-1}$ converges weakly to $\mu^{\otimes k}$ as $N \to \infty$, where $\mathbb{V}^{(k)} : \X^{N} \rightarrow \X^{k}$ denotes the projection onto the first $k$ coordinates. 
\begin{proposition}\label{mu0assum}
Suppose that $0 < \kappa < 1/(16d)$, let $\varphi \laweq \mu_{\Lambda, N, \kappa}$, and set $\mu_0 = \mu^\L$ to be the massive Gaussian free field defined in Theorem~\ref{Aru-cor-im}. Then  $\varphi$ and $\mu_0$ satisfy Assumption~\ref{initial-value}. In particular, the empirical measure of the system \eqref{the-new-model} with initial data $\varphi$ converges in law to the unique martingale solution to the mean-field SDE \eqref{mean-field-limit} with initial distribution $\mu_0$.
\end{proposition}
\begin{proof}

By \cite[Theorem 4.1]{aru2024}, the spin $O(N)$ measures are $\mu^\L$-chaotic, where $\mu^\L$ is the massive Gaussian free field defined in Theorem \ref{Aru-cor-im}. Since $\varphi=(\varphi^1,\cdots,\varphi^N)\laweq \mu_{\Lambda,N,\kappa}$, it is well known that $\{\mu_{\Lambda,N,\kappa}\}_{N\in \mathbb{N}}$ are $\mu^\L$-chaotic is equivalent to the statement that the empirical measure $\frac{1}{N}\sum^N_{i=1} \delta_{\varphi^i}$ converges weakly in probability to $\mu^\L$ as $N\rightarrow \infty$ (see \cite[Proposition $2.2$ \ i)]{Sznitman1991}). Therefore, setting $\mu_0=\mu^\L$, $\varphi$ and $\mu_0$ satisfy Assumptions \ref{initial-value} $(1)$ and $(2)$. Note that for $p\ge 2$, the right-hand side of \eqref{p-mon} is continuous and monotonically decreasing with respect to $p$. Therefore, by Lemma \ref{un-N}, if $0<\kappa< 1/(16d)$, it further follows that $\varphi$ satisfies Assumption \ref{initial-value} $(3)$ for some $\p>2$. We complete the proof by Theorem \ref{Main-Theorem}.
\end{proof}

We now turn to the study of the infinite volume case. Recalling the definition of $\mathcal{Q}$ from Section~\ref{sec4}.  In the following proposition, we consider the periodic extension of  ${\mu}_{\Lambda_\L,N,\kappa}$,  which is still denoted by ${\mu}_{\Lambda_\L,N,\kappa}$ for notational simplicity.   Then, we have
\begin{proposition}\label{infinite-measure}
  The fields $\{\mu_{\Lambda_{\L},N,\kappa}\}_{\L \in \mathbb{N}}$ in $\mathcal{Q}$ form a tight set, and every tight limit of $\{\mu_{\Lambda_{\L},N,\kappa}\}_{\L \in \mathbb{N}}$ is an invariant measure for the infinite volume dynamics  \eqref{the-new-model-Z}. Moreover, if $0<\kappa< (N-2)/(16dN)$, then the invariant measure for \eqref{the-new-model-Z} is unique.
\end{proposition}
\begin{proof}
    Since the sphere $\mathbb{S}^{N-1}(\sqrt{N})$ is a compact set, it follows that the fields $\{\mu_{\Lambda_{\L},N,\kappa}\}_{\L \in \mathbb{N}}$ in $\mathcal{Q}$ form a tight set. By the same argument as in the proof of \cite[Proposition 3.5]{Shen2023}, we obtain that every tight limit of $\{\mu_{\Lambda_{\L},N,\kappa}\}_{\L \in \mathbb{N}}$ is an invariant measure for \eqref{the-new-model-Z}. Similar to the argument of \cite[Corollary $4.13$]{Shen2024}, by directly verifying the Bakry-\'Emery criterion, the invariant measure for the infinite volume dynamics  \eqref{the-new-model-Z} is unique.
\end{proof}
\begin{remark}
    The constant $(N-2)/(16dN)$ obtained in Proposition \ref{infinite-measure} is not optimal; indeed, using the method of \cite{Bauerschmidt2019}, it is possible to achieve a sharper constant. Nevertheless, such an improvement does not enhance $\kappa$ such that the diagram in Corollary \ref{commutes} commutes.
\end{remark}

By the first part of Proposition \ref{infinite-measure}, there exists an invariant measure to \eqref{the-new-model-Z},  denoted by $\mu_{N,\kappa}$.   Similar to the proof of Lemma~\ref{un-N}, we multiply both sides of~\eqref{finite-es} by $\frac{1}{a^{|x|}}$, sum over $x$, and use the estimates $\frac{1}{a^{|x|}} \leq a\,\frac{1}{a^{|y|}}$ for $x \sim y$ and $\sum_{x\in \mathbb{Z}^d} \frac{1}{a^{\abs{x}}}<\infty$. We then obtain the following result.
\begin{lemma}\label{un-N-Z}
If $\kappa > 0$ and $p > 2$ satisfy 
\begin{equation}\label{p-mon-Z}
    \begin{aligned}
         & \kappa<\frac{p}{8(a+1)d(p-1)^{1-\frac{1}{p}}+8dp},
    \end{aligned}
\end{equation}
then the field $\mu_{N,\kappa}$ satisfies
\begin{equation*}
    \begin{aligned}
         & \sum_{x\in \mathbb{Z}^d} \int \frac{1}{a^{\abs{x}}}\abs{\Phi^{i}_x}^p\,\d \mu_{N,\kappa}( \Phi)\lesssim 1,
    \end{aligned}
\end{equation*} 
where the implicit constant is independent of $N$. 
\end{lemma}

Similar to the finite volume case,  we can establish the existence of stationary measures for the mean-field SDE \eqref{mean-field-limit-Z} by verifying the self-consistency condition.  We present the definition of a stationary measure for the mean-field SDE~\eqref{mean-field-limit-Z}.  
\begin{definition}
A probability measure $\mu$ is called a stationary measure for the mean-field SDE~\eqref{mean-field-limit-Z} if there exists a martingale solution $\mathbf{P}$ to the mean-field SDE \eqref{mean-field-limit-Z} with initial distribution $\mu$, and $\mathbf{P}\circ \pi^{-1}_t=\mu$ for all $0\le t\le T$.
\end{definition}
\noindent A martingale solution to the mean-field SDE \eqref{mean-field-limit-Z} is called a stationary solution if its initial distribution is a stationary measure.

We define
\begin{equation*}
    \begin{aligned}
         & \mathscr{P}_{\rm const}(\mathbb{H}) := \left\{
            \mu \in \mathscr{P}(\mathbb{H})
            \;\middle|\;
            \text{there exists } u \in \mathbb{R} \text{ such that } 
            \int v_x v_y\, \mathrm{d}\mu(v) = u, \text{ for all } x \sim y
         \right\}.
    \end{aligned}
\end{equation*}
Furthermore, we define
$$\kappa_c:=\frac{1}{4}G_{\mathbb{Z}^d,0}(\bm{0},\bm{0}),$$ 
where $G_{\mathbb{Z}^d,0}$ denotes the Green function of the Gaussian free field on $\mathbb{Z}^d$.  Let $\muc$ be a  massive Gaussian free field on $\mathbb{Z}^d$ scaled by ${1}/{(2\sqrt{\kappa})}$ with the mass $m^2$, which is the solution to 
\begin{equation*}
    \begin{aligned}
         & G_{\mathbb{Z}^d,m^2}(x,x)=4\kappa,\quad \text{for all\ } x\in \mathbb{Z}^d.
    \end{aligned}
\end{equation*}
\begin{theorem}\label{Aru-cor-im-Z}
Suppose that $0<\kappa<\kappa_{c}$. Then, the field $\muc$ is the unique stationary measure in $\mathscr{P}_{\mathrm{const}}(\mathbb{H})$ to the mean-field SDE \eqref{mean-field-limit-Z}.
\end{theorem}
\begin{proof}
Following the approach used for the finite volume case, we consider the following SDE:
\begin{equation}\label{Zd-Phiu}
    \begin{aligned}
        \d {\Psi}^u_x(t) = &2\kappa\sum_{\N} \left( \Psi^u_y(t)- u\Psi^u_x(t) \right)\,\d t -\frac{1}{2} \Psi^u_x(t)\,\d t+\d W_x(t)\\
=& (2\kappa\,\Delta_{\mathbb{Z}^d} -2\kappa \mathfrak{m}){\Psi}^u_x(t)\,\d t  +\,\d W_x(t), \quad  x\in \mathbb{Z}^d, \ 0\le t\le T,
    \end{aligned}
\end{equation}
where $\mathfrak{m}:=2d({u}-1)+1/(4\kappa)$.  

If $\mathfrak{m}>0$, it is well known that there exists an invariant measure to \eqref{Zd-Phiu},  which is 
\begin{equation*}
    \begin{aligned}
         & N\left(0,\left[-2(2\kappa \Delta_{\mathbb{Z}^d} -2\kappa\mathfrak{m})\right]^{-1}\right)= \frac{1}{2\sqrt{\kappa}}N\left(0,\left[-\Delta_{\mathbb{Z}^d}+\mathfrak{m}\right]^{-1}\right):=\widetilde{\mu}_{\mathfrak{m}},
    \end{aligned}
\end{equation*}
a massive Gaussian free field on $\mathbb{Z}^d$  scaled by $\frac{1}{2\sqrt{\kappa}}$ with mass $\mathfrak{m}$.   Since, for any massive Gaussian free field on $\mathbb{Z}^d$, the correlation function is invariant under translations and rotations by $\frac{\pi}{2}$, it follows that $\widetilde{\mu}_{\mathfrak{m}}\in \mathscr{P}_{\mathrm{const}}(\mathbb{H})$. 

Note that $u=\frac{1}{2d} \mathfrak{m}+1-\frac{1}{8\kappa d}$, it is easy to see that $\widetilde{\mu}_{\mathfrak{m}} $ is a stationary measure to the mean-field SDE \eqref{mean-field-limit-Z}   if and only if there exists a $\mathfrak{m}>0$ such that the following self-consistency condition holds:
\begin{equation}\label{self-Z}
    \begin{aligned}
         & \int v_{\bm{0}} v_{\bm{1}}\,\d \widetilde{\mu}_{\mathfrak{m}}(v)= \frac{1}{4\kappa}G_{\mathbb{Z}^d,\mathfrak{m}}(\bm{0},\bm{1})=u=\frac{1}{2d} \mathfrak{m}+1-\frac{1}{8\kappa d},
    \end{aligned}
\end{equation}
where 
\begin{equation*}
    \begin{aligned}
         & \bm{0}:=(0,\cdots,0)\in \mathbb{Z}^d,\ \  \bm{1}:=(1,0,\cdots,0)\in \mathbb{Z}^d.
    \end{aligned}
\end{equation*}
Define $e_k(x):=e^{2\pi i x\cdot k}$, where $x\in \mathbb{Z}^d$, $k\in [0,1)^d$, note that 
\begin{equation*}
    \begin{aligned}
         & \Delta_{\mathbb{Z}^d} e_k(x)=e^{2\pi i x\cdot k}\sum^d_{j=1} \left( e^{2\pi i k_j}+ e^{-2\pi i k_j}-2 \right)=\Bigg(2\sum^d_{j=1}\big[\cos(2\pi k_j)-1  \big] \Bigg) e^{2\pi i x\cdot k}:=\lambda_k\, e^{2\pi i x\cdot k}.
    \end{aligned}
\end{equation*}
Similar to the finite-volume case, the Green's function $G_{\mathbb{Z}^d, \mathfrak{m}}(x, y)$ satisfies
$$(-\Delta_{\mathbb{Z}^d} + \mathfrak{m}) G_{\mathbb{Z}^d, \mathfrak{m}}(x, y) = \delta_{x, y},\quad x,\, y\in \mathbb{Z}^d,$$
where $\delta_{x, y}$ is defined as $1$ for $x = y$ and $0$ otherwise. By direct computation, we deduce that
\begin{equation}\label{I-correlation}
    \begin{aligned}
         & G_{\mathbb{Z}^d,\mathfrak{m}}(x,y)= \int_{[0,1)^d} \frac{e^{2\pi ik\cdot (x-y)}}{-\lambda_k + \mathfrak{m}}{\rm{d}}k=\int_{[0,1)^d} \frac{\cos\left( 2\pi  k\cdot (x-y) \right)}{-\lambda_k + \mathfrak{m}}\, \d k.
    \end{aligned}
\end{equation}
In particular, we have
\begin{equation*}
    \begin{aligned}
         & G_{\mathbb{Z}^d,\mathfrak{m}}(\bm{0},\bm{1})=\int_{[0,1)^d} \frac{\cos( 2\pi k_1 )}{-\lambda_k + \mathfrak{m}}\, \d k.
    \end{aligned}
\end{equation*}
In the following, we study the limit $\lim_{\mathfrak{m}\rightarrow 0^+} G_{\mathbb{Z}^d,\mathfrak{m}}(\bm{0},\bm{1})$. To this end, it suffices to analyze the behavior of 
\begin{equation}\label{behavior-0-1}
    \begin{aligned}
         & \int_{[0,\varepsilon]^d}\frac{\cos( 2\pi k_1 )}{-\lambda_k+\mathfrak{m} }\, \d k \quad\text{and}\quad \int_{[1-\varepsilon,1)^d}\frac{\cos( 2\pi k_1 )}{-\lambda_k+\mathfrak{m} }\, \d k,
    \end{aligned}
\end{equation}
as $\mathfrak{m} \rightarrow 0^+$ in the regime where $\varepsilon$ is sufficiently small. We focus on the first term in \eqref{behavior-0-1}, as the second term can be analyzed in a similar manner. Note that for $\varepsilon$ small enough, 
\begin{equation*}
    \begin{aligned}
         & \pi^2 k_i^2\le 1-\cos(2\pi k_i)\le 2\pi^2 k_i^2, 
    \end{aligned}
\end{equation*}
where $0\le k_i\le \varepsilon, \ i=1,\cdots,d$. Consequently, we have 
\begin{equation*}
    \begin{aligned}
\int^\varepsilon_0 \frac{r^{d-1}}{\mathfrak{m} +4\pi^2r^2}\, \d r\lesssim \int_{[0,\varepsilon]^d}\frac{\cos( 2\pi k_1 )}{-\lambda_k+\mathfrak{m} }\, \d k\lesssim \int^\varepsilon_0 \frac{r^{d-1}}{\mathfrak{m} +2\pi^2r^2}\, \d r.
    \end{aligned}
\end{equation*}
Moreover, we obtain that for $d = 1, 2$,
$$\lim_{\mathfrak{m} \to 0^+} G_{\mathbb{Z}^d, \mathfrak{m}}(\mathbf{0}, \mathbf{1}) = +\infty,$$
and for $d \geq 3$,
$$\lim_{\mathfrak{m} \to 0^+} G_{\mathbb{Z}^d, \mathfrak{m}}(\mathbf{0}, \mathbf{1}) = \C < +\infty,$$
where 
\begin{equation*}
    \begin{aligned}
         & \C:=\int_{[0,1)^d} \frac{\cos( 2\pi k_1 )}{2\sum^d_{i=1} \left[1-\cos\left( 2\pi k_i \right)  \right]}\, \d k.
    \end{aligned}
\end{equation*}
Note that $\lim_{\mathfrak{m}\rightarrow +\infty} G_{\mathbb{Z}^d,\mathfrak{m}}(\bm{0},\bm{1})=0$, and $G_{\mathbb{Z}^d,\mathfrak{m}}(\bm{0},\bm{1})$ is continuous on $(0,+\infty)$ with respect to $\mathfrak{m}$. For $d = 1, 2$, equation \eqref{self-Z} admits at least one positive solution.  Similarly, for $d \geq 3$, let $\kappa_{\C}:=\frac{1}{4}\C+\frac{1}{8d}$.  Observe that
\begin{equation*}
    \begin{aligned}
         & \kappa_c-\kappa_\C=\frac{1}{4} \int_{[0,1)^d} \frac{1-\cos( 2\pi k_1 )}{2\sum^d_{i=1} \left[1-\cos\left( 2\pi k_i \right)  \right]} \, \d k-\frac{1}{8d}=0.
    \end{aligned}
\end{equation*}
Since $0 < \kappa < \kappa_{c}$, it follows that
\[
\frac{1}{4\kappa} \lim_{\mathfrak{m}\rightarrow 0^+}G_{\mathbb{Z}^d, \mathfrak{m}}(\mathbf{0}, \mathbf{1}) =\frac{1}{4\kappa}\C>1-\frac{1}{8\kappa d}.
\]
This implies that equation \eqref{self-Z} admits at least one positive solution. Therefore, there exists a martingale solution $\mathbf{P}$ to the mean-field SDE~\eqref{mean-field-limit-Z} such that $\mathbf{P}\circ\pi^{-1}_t=\widetilde{\mu}_{\mathfrak{m}}$ for all $0\le t\le T$.

Let $\mathfrak{m}>0$ be any positive solution to equation \eqref{self-Z}. 
Since the martingale solution $\mathbf{P}$ starting from $\widetilde{\mu}_{\mathfrak{m}}$ to the mean-field SDE~\eqref{mean-field-limit-Z} satisfies $\mathbf{P}\circ \pi^{-1}_t=\widetilde{\mu}_{\mathfrak{m}}$ for all $0\le t\le T$, applying It\^o's formula to $\abs{{\Psi_x}(t)}^2$, $x\in \mathbb{Z}^d$ and taking expectations, we have
\begin{equation*}
    \begin{aligned}
         & (8d\kappa{u}+1)\cdot(1-\mathbb{E}[\abs{{\Psi}_x}^2])=0.
    \end{aligned}
\end{equation*}
Since $\mathfrak{m}>0$, we obtain $8d\kappa {u} +1\neq 0$. Therefore,
$\mathbb{E}\left[ |\Psi_x|^2 \right]=1, x\in \mathbb{Z}^d$, which implies that
\begin{equation*}
    \begin{aligned}
     & G_{\mathbb{Z}^d,\mathfrak{m}}(x,x)=\mathbb{E}\left[ 2\sqrt{\kappa}{\Psi}_x \cdot 2\sqrt{\kappa}{\Psi}_x \right]=4\kappa,
    \end{aligned}
\end{equation*}
or, using \eqref{I-correlation}, more explicitly,
\begin{equation}\label{uni-pos}
    \begin{aligned}
    \int_{[0,1)^d}\frac{1}{-\lambda_k + \mathfrak{m}}\,\d k=4\kappa.
    \end{aligned}
\end{equation}
Note that $G_{\mathbb{Z}^d,\mathfrak{m}}(x,x)$ is monotonically decreasing on $(0,+\infty)$ with respect to $\mathfrak{m}$. Therefore, there is a unique positive solution to \eqref{uni-pos}. This completes the proof.
\end{proof}

For $\kappa\ge\kappa_c$, we can also find a stationary measure for the mean-field SDE \eqref{mean-field-limit-Z}. We follow the argument of \cite[Proposition 3.5]{Shen2023}. We recall a relevant result from \cite[Section 4.2]{aru2024}. Recall that $\frac{1}{4\kappa}\,G_{\Lambda_\L,\mathfrak{m}^\L}(\cdot,\cdot)$ denote the correlation function of the Gaussian field $\mu^\L$, as defined in Theorem~\ref{Aru-cor-im}. We set $\kappa_c = \frac{1}{4} G_{\mathbb{Z}^d,0}(\bm{0}, \bm{0})$ throughout, where $G_{\mathbb{Z}^d,0}$ denotes the Green function of the Gaussian free field on $\mathbb{Z}^d$.  It was shown in \cite[Section 4.2]{aru2024} that if $\kappa\le \kappa_c$, for any $x,y\in\mathbb{Z}^d$, 
\begin{equation*}
    \begin{aligned}
         & \frac{1}{4\kappa}\,G_{\Lambda_\L,\mathfrak{m}^\L}(x,y)\rightarrow \frac{1}{4\kappa}\,G_{\mathbb{Z}^d,\mathfrak{m}}(x,y),\quad \text{as}\ \  \L\rightarrow \infty,
    \end{aligned}
\end{equation*}
where $\mathfrak{m}=\lim_{\L\rightarrow \infty} \mathfrak{m}^\L$; if $\kappa>\kappa_c$, for any $x,y\in\mathbb{Z}^d$,
\begin{equation}\label{phase-tran}
    \begin{aligned}
         & \frac{1}{4\kappa}\,G_{\Lambda_\L,\mathfrak{m}^\L}(x,y)\rightarrow  \frac{1}{4\kappa}\,G_{\mathbb{Z}^d,0}(x,y) + \frac{\kappa-\kappa_c}{\kappa},\quad \text{as}\ \  \L\rightarrow \infty.
    \end{aligned}
\end{equation}
Moreover, for any fixed finite box $U \subseteq \mathbb{Z}^d$, which is properly contained in $\Lambda_\L$ for all sufficiently large $\L$,  the above convergence is uniform for all $x, y \in U$.

\begin{remark}
Note that there exists a family of independent Gaussian random variables 
\begin{equation*}
    \begin{aligned}
         & \psi^{\L,k}\laweq\frac{1}{2\sqrt{\kappa}} N\left(0, \left[-\lambda_k^{\L}+\mathfrak{m}^\L\right]^{-1}\right),\ \ k\in \Lambda_\L,
    \end{aligned}
\end{equation*}
such that 
$\sum_{k\in \Lambda_\L} \psi^{\L,k}\, e^\L_k\laweq \mu^\L$, where $e^\L_k$ and $\lambda^\L_k$ are defined in \eqref{elk} and \eqref{spec-lambda}, respectively. According to \cite[Section 4.2]{aru2024}, if $\kappa > \kappa_c$, then 
$$\lim_{\L\rightarrow \infty} \mathfrak{m}^\L \L^d=\frac{1}{4(\kappa-\kappa_c)}.$$
Based on these two facts, it can be observed by direct computation that the constant term on the right-hand side of \eqref{phase-tran} originates from the zero-frequency component $\psi^{\L, \bm{0}} e^\L_{\bm{0}}$.
\end{remark}

In the following, we use a periodic approximation of \eqref{mean-field-limit-Z} to construct the stationary measure for the mean-field SDE \eqref{mean-field-limit-Z}. Considering the periodic extension of the mean-field SDE~\eqref{mean-field-limit}:
\begin{equation}\label{mean-field-limit-L}
    \left\{\begin{aligned}
         \d\Psi^\L_x(t)= &2\kappa\sum_{\N}\Big(  \Psi^\L_y(t)-\mathbb{E}[ \Psi^\L_x(t) \Psi^\L_y(t)] \Psi^\L_x(t)\Big)\,\d t-\frac{1}{2}\Psi^\L_x(t)\,\d t+\d W_x(t),\quad 0\le t\le T,\\
\Psi^\L_x(0)=& \psi^\L_x, \quad x \in\mathbb{Z}^d,\, \L\in \mathbb{N},
\end{aligned}\right.
\end{equation}
where the superscript $\L$  emphasizes the dependence of the equation on $\L$.  Since \eqref{mean-field-limit-L} is a periodic extension to the mean-field SDE \eqref{mean-field-limit}, by Theorem \ref{Aru-cor-im}, there exists a stationary measure to the mean-field SDE \eqref{mean-field-limit-L}, which is a periodic extension of $\mu^\L$,  and is still denoted by $\mu^\L$ for simplicity.

Let $\psi^\L\laweq \mu^\L$ 
be a fixed initial data of the mean-field SDE \eqref{mean-field-limit-L}, and let $\mathbf{P}^\L$ be the law of $\Psi^\L$. Recall the notation introduced in Section~\ref{sec4}. We define $\SSI := C([0, T]; \mathbb{H})$ with $\mathbb{H} := \ell^2_a$, where $C([0, T]; \mathbb{H})$ denotes the space of all continuous functions on $[0, T]$ with values in $\mathbb{H}$. For any $w_1,w_2\in \SSI$, we define the metric
$$
\dI(w_1, w_2) := \sup_{0 \leq t \leq T} \|w_1(t) - w_2(t)\|_2.
$$
Equipped with this metric, $(\SSI,\dI)$ is a Polish space. 

Let $\bm{\mu}$ be a field on $\mathbb{Z}^d$, defined as follows:
\begin{enumerate}
\item if $\kappa<\kappa_c$: let $\bm{\mu}=\muc$, where $\muc$ is introduced in Theorem \ref{Aru-cor-im-Z};

\item if $\kappa=\kappa_c$: let $\bm{\mu}$ be a  Gaussian free field on $\mathbb{Z}^d$ scaled by ${1}/{(2\sqrt{\kappa})}$;

\item if $\kappa>\kappa_c$: let $\bm{\mu}$ be a  Gaussian free field on $\mathbb{Z}^d$ scaled by ${1}/{(2\sqrt{\kappa})}$ plus an independent constant random drift $\sqrt{\frac{\kappa-\kappa_c}{\kappa}}\cdot \mathcal Z$ with $\mathcal Z$ being a standard normal random variable.
\end{enumerate}
\begin{theorem}\label{mu-sm}
The field $\bm{\mu}$ is a stationary measure for the mean-field SDE  \eqref{mean-field-limit-Z}.
\end{theorem}
\begin{proof}
Since \eqref{mean-field-limit-L} is a periodic extension of the mean-field SDE \eqref{mean-field-limit}, and $\psi^{\L} \laweq \mu^\L$, it follows from Theorem \ref{Aru-cor-im} that for the solution $\Psi^\L$ to \eqref{mean-field-limit-L}, we have
\begin{equation}\label{un-1-L}
    \begin{aligned}
         & \mathbb{E}\left[ \abs{\Psi^\L_x}^2(t) \right]=1, \quad \forall\, x\in \mathbb{Z}^d,\, \L\in \mathbb{N},\, t\in [0,T].
    \end{aligned}
\end{equation}    

We first prove that the sequence $\{\mathbf{P}^\L\}_{\L\in\mathbb{N}}$ forms a tight set on $(\SSI,\dI)$. The proof is divided into two steps.

{\bf Step $1$:} 
We prove that the periodic Gaussian field $\mu^\L$ satisfies
\begin{equation}\label{un-L-phi}
    \begin{aligned}
         & \sum_{x\in \mathbb{Z}^d} \int \frac{1}{a^{\abs{x}}}\abs{\Psi_x}^p\,\d \mu^\L( \Psi)\lesssim 1,
    \end{aligned}
\end{equation} 
for some $p>2$ and $a>1$, where the implicit constant is independent of $\L$. 

Since $\mu^\L$ is a Gaussian field satisfying
$$
\int |\Psi_x|^2 \, \mathrm{d} \mu^\L(\Psi) = 1,\quad \forall\, x\in \mathbb{Z}^d, \L\in \mathbb{N},
$$
by Gaussian hypercontractivity, there exist constants $p > 2$ and $a > 1$ such that
$$
\sum_{x \in \mathbb{Z}^d} \int \frac{1}{a^{|x|}} |\Psi_x|^p \, \mathrm{d} \mu^\L(\Psi)
\lesssim \sum_{x \in \mathbb{Z}^d}\frac{1}{a^{|x|}}   \Bigg(\int |\Psi_x|^2 \, \mathrm{d} \mu^\L(\Psi)\Bigg)^{\frac{p}{2}}  
\lesssim 1,
$$
where the implicit constants are independent of $\L$.
This completes the proof of {\bf Step $1$}.

{\bf Step $2$:} We prove that for some $p>2$, 
\begin{equation}\label{Psi-L-1}
    \begin{aligned}
         & \mathbb{E}\left[\sup_{0\le t\le T}\norm{\Psi^\L(t)}_p^p\right] \le C(d,p,\kappa,T) \left(\mathbb{E}\left[\norm{\psi^\L}_p^{p}\right] + 1\right)\lesssim 1,
    \end{aligned}
\end{equation}
and
\begin{equation}\label{Psi-L-2}
    \begin{aligned}
         & \mathbb{E}\left[ \norm{\Psi^\L(t)-\Psi^\L(s)}_p^p \right]\lesssim \left(\abs{t-s}^p +\abs{t-s}^{\frac{p}{2}}  \right),\quad 0\le s\le t\le T,
    \end{aligned}
\end{equation}
where the right-hand sides of \eqref{Psi-L-1} and \eqref{Psi-L-2} are both independent of $\L$.    

By \eqref{un-1-L} and {\bf Step $1$}, the proof of {\bf Step $2$} is similar to those of \eqref{sup-1-p-Zd} and \eqref{lp-t}, so we omit the details.  By the same argument as Lemma \ref{tight-P-2-infty}, we obtain that the sequence $\{\mathbf{P}^\L\}_{\L\in\mathbb{N}}$ forms a tight set on $(\SSI,\dI)$.

We next prove that every tight limit $\mathbf{P}$ of $\{\mathbf{P}^\L\}_{\L\in\mathbb{N}}$ is a martingale solution to \eqref{mean-field-limit-Z} with initial distribution $\bm{\mu}$. Since $\{\mathbf{P}^\L\}_{\L\in\mathbb{N}}$ is tight on $(\SSI,\dI)$, for the sequence $\{\mathbf{P}^\L\}_{\L\in\mathbb{N}}$, there exists a subsequence $\{{\mathbf{P}}^{\L_k}\}_{k\in \mathbb{N}}$ (for simplicity, we still denote the subsequence by $\{\mathbf{P}^\L\}_{\L\in\mathbb{N}}$) such that $\mathbf{P}^\L$ converges weakly to $\mathbf{P}$.  By Skorokhod's  theorem, we can construct a  probability space $(\hat{\Omega},\hat{\mathscr{F}},\hat{\mathbb{P}})$, and random variables $\hat{\Psi}^\L$ and $\hat{\Psi}$ on $(\hat{\Omega},\hat{\mathscr{F}},\hat{\mathbb{P}})$ such that $\hat{\Psi}^\L\laweq \mathbf{P}^\L$, $\hat{\Psi}\laweq \mathbf{P}$, and $\hat{\Psi}^\L\rightarrow \hat{\Psi}$, $\hat{\mathbb{P}}$-\text{a.s.} in $(\SSI,\dI)$, as $\L\rightarrow \infty$.

Since the correlation functions of the Gaussian field defined in Theorem \ref{Aru-cor-im} converge uniformly to those of $\bm{\mu}$ on every finite box that is properly contained in $\Lambda_\L$ for all sufficiently large $\L$, it follows that $\mu^\L \wcon \bm{\mu}$ on $\mathbb{H}$ as $\L\to\infty$. 
Therefore, the proof that $\mathbf{P}$ is a martingale solution to \eqref{mean-field-limit-Z} with initial distribution $\bm{\mu}$ is similar to that of Theorem \ref{Main-Theorem-Z}; in fact, it is even easier.  The main step is to prove that
\begin{equation*}
\begin{aligned}
& \abs{\Pi(\mathbf{P}^\L)-\Pi(\mathbf{P})}\rightarrow 0,\quad
\abs{\bm{\Pi}(\mathbf{P}^\L)-\bm{\Pi}(\mathbf{P})}\rightarrow 0,
\end{aligned}
\end{equation*}
as $\L\rightarrow \infty$, where $\Pi$ and $\bm{\Pi}$ are two functionals defined in Section \ref{sec4}. The details are omitted for brevity.

Given the above preparations, we can now establish the existence of stationary measures for the mean-field SDE \eqref{mean-field-limit-Z}. By $\hat{\Psi}^\L\laweq \mathbf{P}^\L$, $\hat{\Psi}\laweq \mathbf{P}$, and \eqref{un-L-phi},  we obtain that for any $0\le t\le T$, $x\in \mathbb{Z}^d$ and some $p>2$,
\begin{equation*}
    \begin{aligned}
         & \hat{\mathbb{E}}\left[ \abs{\hat{\Psi}^\L_x}^2(t) \right]=\mathbb{E}\left[ \abs{\Psi^\L_x}^2(t) \right]=1,\quad \hat{\mathbb{E}}\left[ \abs{\hat{\Psi}^\L_x}^p(t) \right]\lesssim 1,\quad \hat{\mathbb{E}}\left[ \abs{\hat{\Psi}_x}^2(t) \right]=\mathbb{E}\left[ \abs{\Psi_x}^2(t) \right],
    \end{aligned}
\end{equation*} 
where the implicit constant is independent of $\L$.
Since $\hat{\Psi}^\L\rightarrow \hat{\Psi}$, $\hat{\mathbb{P}}$-\text{a.s.} in $(\SSI,\dI)$, as $\L\rightarrow \infty$, we obtain that $\mathbb{E}[|\Psi_x|^2(t)]=1$ for any $0\le t\le T$, $x\in \mathbb{Z}^d$. Therefore, we obtain $\mathbf{P}\in \mathscr{P}_{\O,\Psi}(\SSI)$. By Lemma \ref{fixed-mu-le-Zd}, Lemma \ref{not-fix-mu-Zd}, and the Yamada-Watanabe argument in the proof of Theorem \ref{Main-Theorem-Z}, we deduce that the martingale solution $\mathbf{P}$ is unique in $\mathscr{P}_{\O,\Psi}(\SSI)$. Moreover, We obtain that $\mathbf{P}^{\L} \wcon \mathbf{P}$ on $(\SSI, \dI)$ without passing to a subsequence. Note that $\mu^\L$ is a stationary measure for \eqref{mean-field-limit-L}. Therefore, the desired result follows.
\end{proof}

Recall the definition of $\mathscr{P}_{\O,\Psi}(\SSI)$ in \eqref{Pone-Phi-I}.   Since the martingale solution $\mathbf{P}$ in the proof of Theorem~\ref{mu-sm} belongs to $\mathscr{P}_{\O,\Psi}(\SSI)$, by applying the Yamada-Watanabe argument as in the proof of Theorem~\ref{strong-solution-Z}, we obtain the following corollary directly.
\begin{corollary}\label{muL-s-Z}
Suppose the initial data $\psi \laweq \bm{\mu}$, then there exists a unique probabilistically strong solution to the mean-field SDE \eqref{mean-field-limit-Z} starting from $\psi$.    
\end{corollary}

\subsection{Convergence of the stationary measures}
In this section, we study the convergence of the spin $O(N)$ model. For convenience, we introduce some notations, which will be employed in Theorem \ref{W2MU} and Theorem \ref{cov-GFF} below. We use $\mathbb{R}^{|\Lambda|k}$ and $\mathbb{H}^k$ to denote the $k$-fold product spaces of $\mathbb{R}^{|\Lambda|}$ and $\mathbb{H}$, respectively. Since both $\mathbb{R}^{|\Lambda|k}$ and $\mathbb{H}^k$ are Polish spaces, we consider the Wasserstein spaces $\mathscr{P}_2(\mathbb{R}^{|\Lambda|k})$ and $\mathscr{P}_2(\mathbb{H}^k)$, equipped with the $2$-Wasserstein distances $ \mathbf{W}_{2,\mathbb{R}^{|\Lambda|k}}$ and $\mathbf{W}_{2,\mathbb{H}^k}$, respectively.

First, we consider the case of finite volume.  By Proposition \ref{invariant-measures}, there exists an invariant measure $\mu_{\Lambda,N,\kappa}$ to the system \eqref{the-new-model}. In the following, we aim to study the large $N$ behavior of $\mu_{\Lambda,N,\kappa}$ using Langevin dynamics and It\^o's calculus. Let $\mathbb{V}^{(k)}: \mathfrak{X}^{N} \rightarrow \mathfrak{X}^{k}$ be the projection onto the first $k$ components, where $\mathfrak{X}$ is a set. We define $\mu^{(k)}_{\Lambda,N,\kappa}:=\mu_{\Lambda,N,\kappa}\circ (\mathbb{V}^{(k)})^{-1}$.  For small $\kappa$, we demonstrate that as $N \rightarrow \infty$, the marginal distribution of the $O(N)$ measure converges to a massive Gaussian free field defined in Theorem~\ref{Aru-cor-im} with a convergence rate of $N^{-\frac{1}{2}}$. Let $\mu_{\Lambda, N, \kappa}^i$ denote the marginal distribution of $\mu_{\Lambda, N, \kappa}$. We recall that $\mu^\L$ is the massive Gaussian free field defined in Theorem~\ref{Aru-cor-im}, and we denote by $(\mu^\L)^{\otimes k}$ its $k$-fold product measure.

Then, we have the following theorem.
\begin{theorem}\label{W2MU}
Suppose $0<\kappa < 1/(32d)$. Then, for fixed $k\in \mathbb{N}$,
\begin{equation}\label{W2-k}
\mathbf{W}_{2,\mathbb{R}^{\abs{\Lambda}k}}(\mu^{(k)}_{\Lambda,N,\kappa},(\mu^\L)^{\otimes k})\lesssim N^{-1/2},
\end{equation}
where the implicit constant depends on $\abs{\Lambda}$, but is independent of $N$.
\end{theorem}
\begin{proof}

In the proof, we fix $N$. Similar to \cite[Lemma 5.7]{Shen2022}, we can construct a stationary coupling $({\Phi}^i, {\Psi}^i)$ of $\mu^i_{\Lambda,N,\kappa}$ and $\mu^\L$ whose components satisfy \eqref{the-new-model} and \eqref{Psi-i-limit}. The stationary of the joint law of $({\Phi}^i, {\Psi}^i)$ implies that ${\Phi}^i-{\Psi}^i$ is stationary. In the following, we prove that 
\begin{equation*}
    \begin{aligned}
          \mathbb{E}\left[\sum_{x\in \Lambda}\abs{{\Phi}_x^i-{\Psi}_x^i}^2\right]\lesssim \frac{1}{N},
    \end{aligned}
\end{equation*}
which, by the definition of the Wasserstein distance, implies \eqref{W2-k}.  

Note that the laws of ${\Phi}^i-{\Psi}^i$ and ${\Phi}^j-{\Psi}^j$, where $i\neq j$, are identical. Therefore, we have
\begin{equation}\label{average}
    \begin{aligned}
          \mathbb{E}\left[\sum_{x\in \Lambda}\abs{{\Phi}_x^i-{\Psi}_x^i}^2\right]= \frac{1}{N}\sum^{N}_{i=1}\mathbb{E}\left[\sum_{x\in \Lambda}\abs{{\Phi}_x^i-{\Psi}_x^i}^2\right].
    \end{aligned}
\end{equation}
Next, we estimate the right-hand side of \eqref{average}.  The idea is similar to that of {\bf Step $3$} in the proof of Theorem \ref{strong-solution}, and we adopt the same notation as in that theorem. By applying It\^o's formula to $({1}/{N})\sum_{i=1}^N|{\Phi}_x^i-{\Psi}^i_x|^2$, we obtain that
\begin{equation*}
    \begin{aligned}
         & \d \frac{1}{N}\sum_{i=1}^N|{\Phi}_x^i-{\Psi}^i_x|^2(t)=\frac{1}{N}\sum_{i=1}^{N} A_1\,\d t-\frac{1}{N}\sum_{i=1}^{N} A_2\,\d t -\frac{1}{N}\sum_{i=1}^{N} A_3\,\d t +\mathbf{Martingale},
    \end{aligned}
\end{equation*}
where
\begin{equation*}
    \begin{aligned}
       \frac{1}{N}\sum_{i=1}^{N} A_1\le&  2\kappa \sum_{\N} \frac{1}{N}\sum_{i=1}^{N} \left( \abs{\Phi^i_x-\Psi^i_x}^2+\abs{\Phi^i_y-\Psi^i_y}^2\right),\\
 -\frac{1}{N}\sum_{i=1}^{N} A_3\le&  -\frac{1}{N}\sum_{i=1}^{N}\abs{\Phi^i_x-\Psi^i_x}^2+\frac{2}{N} +\frac{1}{N}\left(\frac{1}{N}\sum_{i=1}^{N}\abs{\Psi^i_x}^2\right)^{\frac{1}{2}},
    \end{aligned}
\end{equation*} 
and
\begin{equation*}
    \begin{aligned}
         & \frac{1}{N}\sum_{i=1}^{N} A_2=4\kappa \sum^4_{k=1} \sum_{\N} \frac{1}{N}\sum_{i=1}^{N} A_{2k}.
    \end{aligned}
\end{equation*}
In the following, we only provide the estimate for $4\kappa \sum_{\N}\frac{1}{N}\sum_{i=1}^{N} A_{21}$. Estimates for $A_{22}, A_{23}$, and $A_{24}$ are similar, and the details are omitted.  Note that $\mathbb{E}\left[ \Psi^k_x \Psi^k_y \right] = u$ for all $x \sim y$, where $u$ is a constant, and $\mathbb{E}\left[ \abs{\Psi^i_x}^2 \right] = 1$ for all $x \in \Lambda$, we obtain that
\begin{equation*}
    \begin{aligned}
          &4\kappa \sum_{\N}\frac{1}{N}\sum_{i=1}^{N} A_{21}=4\kappa \sum_{\N} \frac{1}{N} \sum_{k=1}^N \left(\Psi^k_x\Psi^k_y- u\right)\cdot \frac{1}{N}\sum^N_{i=1} \Psi^i_x\left(\Phi^i_x -\Psi_x^i\right)\\
\le& 4\kappa \sum_{\N}\left[ \frac{1}{4\varepsilon}\abs{ \frac{1}{N} \sum_{k=1}^N \left(\Psi^k_x\Psi^k_y- u\right)}^2 + \varepsilon \left( \frac{1}{N} \sum^N_{i=1}\abs{\Psi^i_x}^2 \right) \left(\frac{1}{N} \sum^N_{i=1} \abs{\Phi^i_x-\Psi^i_x}^2\right)\right]\\
\le& 4\kappa \sum_{\N}\Bigg[ \frac{1}{4\varepsilon}\abs{ \frac{1}{N} \sum_{k=1}^N \left(\Psi^k_x\Psi^k_y- u\right)}^2 + \varepsilon \left(\frac{1}{N} \sum^N_{i=1} \abs{\Phi^i_x-\Psi^i_x}^2\right)\\
&\quad\quad\quad\quad\quad+\varepsilon \abs{\frac{1}{N} \sum^N_{i=1}\abs{\Psi^i_x}^2 -1} \left(\frac{1}{N} \sum^N_{i=1} \abs{\Phi^i_x-\Psi^i_x}^2\right) \Bigg]\\
\le& 4\kappa \sum_{\N}\Bigg[ \frac{1}{4\varepsilon}\abs{ \frac{1}{N} \sum_{k=1}^N \left(\Psi^k_x\Psi^k_y- u\right)}^2 + 2\varepsilon \left(\frac{1}{N} \sum^N_{i=1} \abs{\Phi^i_x-\Psi^i_x}^2\right)\\
&\quad\quad\quad\quad\quad+\frac{\varepsilon}{4} \abs{ \frac{1}{N} \sum^N_{i=1}\left(\abs{\Psi^i_x}^2 -1\right)}^2 \left(\frac{1}{N} \sum^N_{i=1} \abs{\Phi^i_x-\Psi^i_x}^2\right) \Bigg],\\
    \end{aligned}
\end{equation*}
where the first and third inequalities are due to the basic inequality: for $\varepsilon >0$, $ab \leq \varepsilon a^2 + \frac{1}{4\varepsilon} b^2$.
Taking expectations, and using the stationary property to eliminate the initial data, we obtain
\begin{equation}\label{un-es-stationary}
    \begin{aligned}
        \frac{1}{N}\sum^{N}_{i=1}\mathbb{E}\left[\abs{{\Phi}_x^i-{\Psi}_x^i}^2\right]
\le &  \frac{4\kappa(4\varepsilon +3)}{N} \sum^N_{i=1}\sum_{\N} \mathbb{E}\left[ \abs{\Phi^i_x-\Psi^i_x}^2 \right]\\
 &+ \frac{4\kappa}{N} \sum^N_{i=1}\sum_{\N} \mathbb{E}\left[ \abs{\Phi^i_y-\Psi^i_y}^2 \right]+ \frac{C(\varepsilon,\mathbb{E}[|\Phi_x^i|^4])}{N},
    \end{aligned}
\end{equation}
where $C(\varepsilon,\mathbb{E}[|\Phi_x^i|^4])$ is a constant that depends on $\varepsilon$ and $\mathbb{E}[{|\Phi}_x^i|^4]$.  Summing over $x$, we conclude that
\begin{equation*}
    \begin{aligned}
         & \frac{1}{N}\sum^{N}_{i=1}\mathbb{E}\left[\sum_{x\in \Lambda}\abs{{\Phi}_x^i-{\Psi}_x^i}^2\right]\\
\le& \frac{32\kappa d(\varepsilon+1)}{N}\sum^{N}_{i=1}\mathbb{E}\left[\sum_{x\in \Lambda}\abs{{\Phi}_x^i-{\Psi}_x^i}^2\right]+\frac{C(\varepsilon,\mathbb{E}[|\Phi^i|^4])}{N},
    \end{aligned}
\end{equation*}
where $C(\varepsilon,\mathbb{E}[|\Phi^i|^4])$ is a constant that depends on $\varepsilon$ and $\mathbb{E}[{|\Phi}^i|^4]$.  By choosing $\kappa$ small enough such that
\begin{equation*}
0<\kappa < \frac{1}{32d(\varepsilon+1)},
\end{equation*}
and noting that for $0<\kappa < 1/(32d(\varepsilon+1))$, $\mathbb{E}[|{\Phi}^i|^4]$ is bounded by Lemma \ref{un-N}. We complete the proof.
\end{proof}

By Proposition \ref{infinite-measure}, there exists an invariant measure to the system \eqref{the-new-model-Z}, denoted by $\mu_{N,\kappa}$. Let $\mu^i_{N,\kappa}$ be the marginal distribution of the invariant measure $\mu_{N,\kappa}$. Define $\mu^{(k)}_{N,\kappa} := \mu_{N,\kappa} \circ (\mathbb{V}^{(k)})^{-1}$.  We denote by  $\bm{\mu}$ the stationary measure for \eqref{mean-field-limit-Z}, as defined by Theorem \ref{mu-sm}, and by $\bm{\mu}^{\otimes k}$ its $k$-fold product measure. Similar to the proof of Theorem \ref{W2MU}, we construct a stationary coupling of a pair of stationary measures. By multiplying both sides of \eqref{un-es-stationary} by $\frac{1}{a^{|x|}}$, summing over $x$, and using the estimates $\frac{1}{a^{|x|}} \leq a\,\frac{1}{a^{|y|}}$ for $x \sim y$ and $\sum_{x\in \mathbb{Z}^d} \frac{1}{a^{\abs{x}}}<\infty$, we obtain the following theorem.
\begin{theorem}\label{cov-GFF}
Suppose $0<\kappa < 1/(32d)$. Then, for fixed $k\in \mathbb{N}$,
\begin{equation*}
\mathbf{W}_{2,\mathbb{H}^k}(\mu^{(k)}_{N,\kappa},\bm{\mu}^{\otimes k})\lesssim N^{-1/2},
\end{equation*}
where the implicit constant is independent of $N$.
\end{theorem}

In the following corollary, we continue to use $\mu_{\Lambda_\L,N,\kappa}$ and $\mu^\L$ to denote the periodic extensions of $\mu_{\Lambda_\L,N,\kappa}$ and $\mu^\L$, respectively. Let $\mu^i_{\Lambda_\L,N,\kappa}$ denote the marginal distribution of $\mu_{\Lambda_\L,N,\kappa}$.   By combining Proposition \ref{infinite-measure}, Theorem \ref{W2MU}, and Theorem \ref{cov-GFF}, we have
\begin{corollary}
If $0<\kappa<1/(32d)$, then the following diagram commutes: 
\begin{center}
      \begin{tikzcd}[row sep=3em, column sep=5em]
{\mu^{i}_{\Lambda_\L,N,\kappa}} \arrow[r, "\text{$N\rightarrow \infty$}"] \arrow[d, "\text{$\L\rightarrow \infty$}"'] & \mu^\L \arrow[d, "\text{$\L\rightarrow \infty$}"] \\
 \mu^i_{N,\kappa} \arrow[r, "\text{$N\rightarrow \infty$}"] & \bm{\mu},
\end{tikzcd}  
\end{center}  
where all limits are taken in the weak sense.  
\end{corollary}
\begin{proof}
By Theorem \ref{W2MU}, for $0 < \kappa < 1/(32d)$, we have $\lim_{N \rightarrow \infty} \mu^i_{\Lambda_\L,N,\kappa} = \mu^\L$. Furthermore, according to \cite[Section $4.2$]{aru2024}, the correlation functions of the Gaussian field defined in Theorem \ref{Aru-cor-im} converge uniformly to those of $\bm{\mu}$ on every finite box that is properly contained in $\Lambda_\L$ as $\L$ becomes sufficiently large. Therefore, we obtain that $\lim_{\L \rightarrow \infty} \mu^\L= \bm{\mu}$.

By Proposition \ref{infinite-measure},  for $0<\kappa<1/(32d)$, we have $\lim_{\L\rightarrow \infty} \mu^i_{\Lambda_\L,N,\kappa} = \mu^i_{N,\kappa}$. In addition, by Theorem \ref{cov-GFF}, for $0<\kappa < 1/(32d)$, it follows that $\lim_{N\rightarrow \infty} \mu^i_{N,\kappa}=\bm{\mu}$.

Combining these results, we obtain that for $0 < \kappa < 1/(32d)$,
\begin{equation*}
    \begin{aligned}
         & \lim_{\L\rightarrow \infty}\lim_{N\rightarrow \infty} \mu^i_{\Lambda_\L,N,\kappa}=\lim_{N\rightarrow \infty}\lim_{\L\rightarrow \infty} \mu^i_{\Lambda_\L,N,\kappa}= \bm{\mu}.
    \end{aligned}
\end{equation*} 
\end{proof}

With the help of Theorem \ref{cov-GFF}, we can establish an analogue of Proposition \ref{mu0assum} for the infinite volume case.
\begin{proposition}\label{mu0assum-Z}
Suppose that $0 < \kappa < 1/(32d)$, let $\varphi \laweq \mu_{ N, \kappa}$, and set $\mu_0 = \bm{\mu}$ to be the Gaussian field defined in Theorem~\ref{mu-sm}. Then  $\varphi$ and $\mu_0$ satisfy Assumption~\ref{initial-value-Z}. In particular, the empirical measure of the system \eqref{the-new-model-Z} with initial data $\varphi$ converges in law to the unique martingale solution to the mean-field SDE \eqref{mean-field-limit-Z} with initial distribution $\mu_0$.
\end{proposition}
\begin{proof}
By Theorem \ref{cov-GFF}, the measures $\{\mu_{N,\kappa}\}_{N\in \mathbb{N}}$ are $\bm{\mu}$-chaotic.  Similar to the arguments in the proof of Proposition~\ref{mu0assum}, since $\varphi=(\varphi^1,\cdots,\varphi^N)\laweq \mu_{N,\kappa}$, we obtain that the empirical measure $\frac{1}{N}\sum^N_{i=1} \delta_{\varphi^i}$ converges weakly in probability to $\bm{\mu}$ as $N\rightarrow \infty$. Therefore, setting $\mu_0=\bm{\mu}$, $\varphi$ and $\mu_0$ satisfy Assumptions \ref{initial-value-Z} $(1)$ and $(2)$. By Lemma \ref{un-N-Z}, if $0<\kappa<1/(32d)$, it further follows that $\varphi$ satisfies Assumption \ref{initial-value-Z} $(3)$ for some $p>2$, $a>1$. We complete the proof by Theorem \ref{Main-Theorem-Z}.
\end{proof}

\subsection{Uniqueness of the stationary measure}
In this section, we establish the uniqueness of the stationary measure for the mean-field SDE \eqref{mean-field-limit-n-i} when $\kappa$ is small.  Recall the definitions of   $\mathscr{P}_{1}(\mathbb{R}^{\abs{\Lambda}})$ and $\mathscr{P}_{1}(\mathbb{H})$ in \eqref{Pone-R} and \eqref{Pone-H}, respectively.
\begin{theorem}\label{unique-in-m-c}
If $0 < \kappa < 1/(32d)$,  then there exists at most one stationary measure in $\mathscr{P}_{1}(\mathbb{R}^{\abs{\Lambda}})$  to the mean-field SDE \eqref{mean-field-limit}. 
\end{theorem}
\begin{proof}
Let $\mu\in \mathscr{P}_{1}(\mathbb{R}^{\abs{\Lambda}})$ be any stationary measure to the mean-field SDE \eqref{mean-field-limit}.  By Definition \ref{sta-F}, it is easy to see that the martingale solution $\mathbf{P}$ to the the mean-field SDE \eqref{mean-field-limit} with initial distribution $\mu$ belongs to $\mathscr{P}_{1,\Psi}(\mathbb{R}^{\abs{\Lambda}})$. Moreover, by the Yamada-Watanabe argument in the proof of Theorem \ref{strong-solution}, we obtain that if the initial data $\psi$ of the mean-field SDE \eqref{mean-field-limit} satisfies $\psi \laweq \mu$, then there exists a unique probabilistically strong solution to \eqref{mean-field-limit} starting from $\psi$.

Suppose that $\mu^{(1)},\mu^{(2)} \in \mathscr{P}_{1}(\mathbb{R}^{\abs{\Lambda}})$ are any two stationary measures to the mean-field SDE \eqref{mean-field-limit}.  
Let ${\psi}^{(1)} \overset{\text{law}}{=} \mu^{(1)}$ and ${\psi}^{(2)} \overset{\text{law}}{=} \mu^{(2)}$. The solutions to \eqref{mean-field-limit} with initial data ${\psi}^{(1)}$ and ${\psi}^{(2)}$ are denoted by ${\Psi}^{(1)}$ and ${\Psi}^{(2)}$, respectively.  Since $\mu^{(1)},\mu^{(2)} \in \mathscr{P}_{1}(\mathbb{R}^{\abs{\Lambda}})$ are two stationary measures to the mean-field SDE \eqref{mean-field-limit}, we have
\begin{equation}\label{bound-Psi}
    \begin{aligned}
         &\mathbb{E}\left[\abs{{\Psi}_x^{(i)}(t)}^2\right]\le 1,\quad \forall\, x\in \Lambda,\ t\in[0,T], \ i=1,2.
    \end{aligned}
\end{equation}
Following the proof of Lemma \ref{not-fix-mu}, it follows that for any $T>0$ and $0\le t\le T$,
\begin{equation}\label{stationary-unique}
    \begin{aligned}
         & \frac{\d}{\d t} \mathbb{E}\left[\abs{{\Psi}_x^{(1)}-{\Psi}_x^{(2)}}^2\right](t)\\
\le& 4\kappa\sum_{\N}\mathbb{E}\left[\abs{{\Psi}_y^{(1)}-{\Psi}_y^{(2)}}^2\right](t)+ \left(24\kappa d-1\right)\mathbb{E}\left[\abs{{\Psi}_x^{(1)}-{\Psi}_x^{(2)}}^2\right](t).
    \end{aligned}
\end{equation}    
By summing over $x$ and applying Gronwall's inequality, we obtain that for any $T>0$ and $0\le t\le T$,
\begin{equation}\label{Phi-exp}
    \begin{aligned}
         & \mathbb{E}\left[\sum_{x\in \Lambda}\abs{{\Psi}_x^{(1)}-{\Psi}_x^{(2)}}^2\right](t)\le \exp\big( \left( 32\kappa d-1 \right)t \big) \,\mathbb{E}\left[\sum_{x\in \Lambda}\abs{{\psi}_x^{(1)}-{\psi}_x^{(2)}}^2\right].
    \end{aligned}
\end{equation}
By invariance of $\mu^{(1)}$ and $\mu^{(2)}$, the definition of the Wasserstein distance, \eqref{Phi-exp}  and \eqref{bound-Psi}, we derive that for any $T>0$ and $0\le t\le T$,
\begin{equation*}
    \begin{aligned}
         & \mathbf{W}^2_{2,\mathbb{R}^{\abs{\Lambda}}}(\mu^{(1)},\mu^{(2)})\le \mathbb{E}\left[\sum_{x\in \Lambda}\abs{{\Psi}_x^{(1)}-{\Psi}_x^{(2)}}^2\right](t)\le 4\abs{\Lambda}\exp\big( \left( 32\kappa d-1 \right)t \big).
    \end{aligned}
\end{equation*}
Letting $t\rightarrow \infty$, we conclude that $\mu^{(1)}=\mu^{(2)}$.
\end{proof}

In the case of infinite volume,  we multiply both sides of \eqref{stationary-unique} by $\frac{1}{a^{|x|}}$, sum over $x$, and use the estimates $\frac{1}{a^{|x|}} \leq a\,\frac{1}{a^{|y|}}$ for $x \sim y$ and $\sum_{x\in \mathbb{Z}^d} \frac{1}{a^{\abs{x}}}<\infty$. Similar to the argument in the proof of Theorem~\ref{unique-in-m-c}, we have
\begin{theorem}\label{unique-in-m-c-2}
If $0 < \kappa < 1/(32d)$, then there exists at most one stationary measure in $\mathscr{P}_{1}(\mathbb{H})$ to the mean-field SDE \eqref{mean-field-limit-Z}.
\end{theorem}
\begin{remark}
Examining the proofs of Theorems~\ref{unique-in-m-c} and~\ref{unique-in-m-c-2} shows that the bound $1$ in the definitions of the measure sets $\mathscr{P}_{\O}(\mathbb{R}^{\abs{\Lambda}})$ and $\mathscr{P}_{\O}(\mathbb{H})$ can be replaced by any constant $C \geq 1$, with the corresponding interval for the inverse temperature given by $0 < \kappa < 1/(8d + 24dC)$.
\end{remark}

\section*{Acknowledgements} 
The authors would like to thank  Scott Smith and Xicheng Zhang for their helpful comments on an earlier version of this paper. W.Y. gratefully acknowledges financial support from the NSFC (No.\,12401181). 
R.Z. is grateful to the financial supports of the National Key
R\&D Program of China (No.\,2022YFA1006300), the NSFC (No.\,12426205, No.\,12271030) and the financial supports by the Deutsche Forschungsgemeinschaft (DFG, German Research Foundation) - Project-ID 317210226-SFB 1283.

\end{document}